\documentclass[12pt]{amsart}
\usepackage{amssymb,amsmath,amsthm}
\usepackage{esint}
\usepackage{float} \restylefloat{table}
\usepackage{tikz,tikz-cd}
\usepackage{hyperref}
\usepackage[msc-links, backrefs]{amsrefs}
\hypersetup{
  colorlinks   = true,	
  urlcolor     = blue,	
  linkcolor    = blue,	
  citecolor   = red	
}
\calclayout
\newtheorem{theorem}{Theorem}[section]
\newtheorem{proposition}[theorem]{Proposition}
\newtheorem{corollary}[theorem]{Corollary}
\newtheorem{lemma}[theorem]{Lemma}
\newtheorem{definition}[theorem]{Definition}
\newtheorem{remark}[theorem]{Remark}
\numberwithin{theorem}{section} \numberwithin{equation}{section}
\newcommand{\beq}{\begin{small} \begin{equation}}
\newcommand{\eeq}{\end{equation} \end{small}}
\newcommand{\beqn}{\begin{small} \begin{equation*}}
\newcommand{\eeqn}{\end{equation*} \end{small}}
\newcommand\oiiint[1]{\ensuremath \includegraphics[trim=0 33 0 0,width=1.7em]{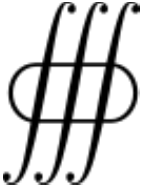}\!\! \raisebox{-11pt}{$\scriptstyle #1$}}
\newcommand{\hpg}[5]{{}_{#1}F_{#2}\! \left(\left.{\genfrac{}{}{0pt}{}{#3}{#4}}\right| #5 \right) }
\newcommand{\heun}{H\!\ell}
\newcommand{\Scale}{0.6}
\newcommand{\ScaleBig}{0.9}
\title[Calabi-Yau realizing symplectically rigid monodromy]{Calabi-Yau manifolds realizing \\symplectically rigid monodromy tuples}
\author{Charles F. Doran}
\address{Dept.\!~of Mathematics, University of Alberta,
Edmonton, Alberta  T6G 2G1} \email{charles.doran@ualberta.ca}
\author{Andreas Malmendier}
\address{Dept.\!~of Mathematics \& Statistics, Utah State University,
Logan, UT 84322}
\email{andreas.malmendier@usu.edu}
\begin{document}
\begin{abstract}
We define an iterative construction that produces a family of elliptically fibered Calabi-Yau $n$-folds with section from a family of elliptic Calabi-Yau varieties of one dimension lower. Parallel to the geometric construction, we iteratively obtain for each family with a point of maximal unipotent monodromy, normalized to be at $t=0$, its Picard-Fuchs operator and a closed-form expression for the period holomorphic at $t=0$, through a generalization of the classical Euler transform for hypergeometric functions.  In particular, our construction yields one-parameter families of elliptically  fibered Calabi-Yau manifolds with section whose Picard-Fuchs operators realize all symplectically rigid Calabi-Yau differential operators with three regular singular points classified by Bogner and Reiter, but also non-rigid operators with four singular points.
\end{abstract}
\maketitle
\section{Introduction}
The study of Calabi-Yau manifolds, i.e., compact K\"ahler manifolds with trivial canonical bundle, has been an active field in algebraic geometry and mathematical physics ever since their christening by Candelas et al.~\cite{MR800347}  in 1985. For every positive integer $n$, the vanishing set of a non-singular homogeneous polynomial of degree $n+2$ in the complex projective space $\mathbb{P}^{n+1}$ is a compact Calabi-Yau manifold of $n$ complex dimensions, or Calabi-Yau $n$-fold for short.  The construction yields for $n=1$ an elliptic curve, while for $n=2$ one obtains a K3 surface. In two complex dimensions K3 surfaces are the only simply connected Calabi-Yau manifolds. The classification of Calabi-Yau threefolds remains an open problem. Results of tremendous ongoing activity in physics that included systematic computer searches have given us a better idea of the landscape of Calabi-Yau threefolds \cites{Johnson:2014aa,MR3006963}. For example, the work has impressively demonstrated the near omnipresence of elliptic fibrations on Calabi-Yau threefolds. Unfortunately, it has also been revealed that it is generally quite difficult to construct examples of families of Calabi-Yau threefolds with small Hodge number $h^{2,1}$ by specialization of multi-parameter families. 
\par The mathematical study of mirror symmetry essentially began with the example of the quintic and mirror quintic family of Candelas, de la Ossa, Green and Parkes~\cite{MR1101784}. The quintic family is a generic quintic hypersurface $X \subset \mathbb{P}^5$ , e.g., the Fermat quintic
\beqn
 X_0^5 + \dots + X_4^5 =0 \,,
\eeqn
which is a Calabi-Yau threefold with Hodge numbers $h^{1,1}=1$ and $h^{2,1} = 101$. The mirror family has a flipped Hodge diamond $h^{2,1}=1$ and $h^{1,1} = 101$ and can be constructed via the Greene-Plesser orbifolding construction from the Dwork pencil
\beqn
 X_0^{5} + \dots + X_{4}^{5} + 5 \lambda \, X_0 X_1 \cdots X_{4} = 0 \,.
\eeqn
The mirror quintic family is a one-dimensional family of Calabi-Yau threefolds defined over the base $\mathbb{P}^1 \backslash \lbrace 0, 1, \infty\rbrace$ and has exactly three singular fibers: the large complex structure limit, the Gepner point, and the stacky point (the first two describe the discriminant locus of ``bad” $N = (2, 2)$-SCFT). There are some further properties coming from the ``special geometry'' of the moduli space. The family has played a crucial role in the spectacular computations which suggested that mirror symmetry could be used to solve long-standing problems in enumerative geometry. These remarkable observations led to an enormous mathematical activity which both tried to explain the observed phenomena and to establish similar mirror recipes and results for other families of Calabi-Yau threefolds.  For example, Batyrev described in \cite{MR1269718} a way to construct the mirror of Calabi-Yau hypersurfaces in toric varieties via dual reflexive polytopes, and a proof of the so-called mirror theorem was given by Givental in \cite{MR1408320} and Lian, Liu, and Yau in \cite{MR1621573}.  However, most of the standard recipes fail to consistently produce families with the desired properties. The goal of this paper is to present a construction which rectifies that.
\par The quintic-mirror family gives rise to a variation of Hodge structure and an associated Picard-Fuchs differential  equation. In turn, the solutions to the differential equation, called periods, determine this variation of Hodge structure. Doran and Morgan \cite{MR2282973} classified certain one-parameter variations of Hodge structure which arise from families of Calabi-Yau threefolds with one-dimensional rational deformation space, i.e., families of Calabi-Yau threefolds resembling the quintic-mirror family. It is worth pointing out that their result was achieved not by constructing families of  Calabi-Yau threefolds, but by classifying all integral weight-three variations of Hodge structure which \emph{can} underlie a family of Calabi-Yau threefolds over a thrice-punctured sphere $\mathbb{P}^1 \backslash \lbrace 0, 1, \infty\rbrace$ -- subject to certain conditions on the monodromy coming from mirror symmetry. Calabi-Yau threefolds with large complex structure limit and $h^{2,1}=1$ have since emerged in a pivotal role for both mathematics and physics, where  classical geometric, toric, and analytical methods are used in their investigation. The Picard-Fuchs equations for other families of Calabi-Yau threefolds were constructed by Batyrev, van Straten and others in \cites{MR1328251, MR1619529} which lead to a general definition of a Calabi-Yau differential operator by Almkvist, van Enckevort, van Straten and Zudilin, rooted purely in the theory of differential operators; see \cites{MR3822913,Almkvist:aa}. A large number of these are known today (currently over 500!) but they are mostly found via computer searches. The reason for the name is that they are conjectured to come from families of Calabi-Yau threefolds having a large complex structure limit and $h^{2,1} = 1$.
\par This paper addresses part of the conjecture. We devise an iterative geometric construction that finds explicit families of elliptically fibered Calabi-Yau manifolds whose Picard-Fuchs operators realize a big class of Calabi-Yau differential operators, including the so-called symplectically rigid operators, but also non-rigid operators.
\section{Summary of results}
The main result of this article is an \emph{iterative twist construction} that produces projective families $\pi: X \to B=\mathbb{P}^1 \backslash \lbrace 0, 1, \infty\rbrace$ of elliptically fibered Calabi-Yau $n$-folds $X_t=\pi^{-1}(t)$ with $t \in B=\mathbb{P}^1 \backslash \lbrace 0, 1, \infty\rbrace$ with section from families of elliptic Calabi-Yau varieties with the same properties in one dimension lower, for $n=1, 2, 3, 4$. In this paper we will always restrict ourselves to elliptic fibrations with sections, so-called Jacobian elliptic fibrations. These families are then presented as Weierstrass models which are ubiquitous in the description of families of elliptic curves and K3 surfaces. In fact, Gross proved~\cite{MR1272978} that there are only a finite number of distinct topological types of elliptically fibered Calabi-Yau threefolds up to birational equivalence. For an elliptically fibered Calabi-Yau threefold, the existence of a global section makes it then possible to find an explicit presentation as a Weierstrass model~\cite{MR977771}. For example, our iterative procedure constructs from extremal families of elliptic curves\footnote{An elliptic fibration $\pi: X \to \mathbb{P}^1$ is called extremal if and only if for the group of sections we have $\operatorname{rank} \operatorname{MW}(\pi)=0$ and the associated elliptic surface has maximal Picard number.} with rational total space, families of Jacobian elliptic K3 surfaces of Picard rank $18$ or $19$, and in turn from these elliptically fibered Calabi-Yau threefolds with $h^{2,1}=1$, and can be continued further on. Moreover, all families are iteratively constructed from a \emph{single} geometric object, the mirror family of Fermat quadrics in $\mathbb{P}^1 \backslash \{1, \infty\}$ given by
\beq
\label{quadric_pencil}
 X_0^2 + X_1^2 + 2 \, t \, X_0  \, X_1= 0 \,.
\eeq
The broad range of families of Jacobian elliptic Calabi-Yau manifolds obtained by our iterative construction 
includes the following noteworthy families with generic fibers of dimension $n$:
\begin{itemize}
\item[$\lbrack n=1\rbrack$] the universal families of elliptic curves over the modular curves for $\Gamma_0(k)$ with $k=1, 2, 3, 4, 5, 6, 8, 9$,
\item[$\lbrack n=2\rbrack$]  families of $M_k$-lattice polarized K3 surfaces over the modular curves for $\Gamma_0(k)^+$ with $k=1, 2, 3, 4, 5, 6, 8, 9$,
and families of $M$-lattice polarized K3 surfaces (and closely related lattices of Picard rank 18),
\item[$\lbrack n=3\rbrack$] families of Calabi-Yau threefolds with $h^{2,1}=1$ 
realizing all 14 one-parameter variations of Hodge structure classified by Doran and Morgan in \cite{MR2282973},
\item[$\lbrack n=4\rbrack$]  families of Calabi-Yau fourfolds realizing all 14 hypergeometric one-parameter variations of Hodge structure of weight four and type 
$(1, 1, 1, 1,1)$ over a one-dimensional rational deformation space,
\item[$\lbrack n \in \mathbb{N}\rbrack$] mirror families of Dwork pencils in $\mathbb{P}^{n+1}$.
\end{itemize}
\par Katz discovered that linearly rigid monodromy tuples are obtained as tensor products and convolutions of rank-one local systems~\cite{MR1366651}; Doran and Morgan~\cite{MR2282973} classified all possible fourteen linearly rigid monodromy-tuples that could come from a B-side variation of Hodge structures. As it turns out, every single one of them admits geometric realization as hypersurfaces or complete intersections in a Gorenstein toric Fano variety or as Calabi-Yau threefolds fibered by high rank K3’s by Clingher et al.~ \cite{Clingher:2013aa}\nocite{MR2369490}. Within the class of irreducible Fuchsian differential operators of Calabi-Yau type, the symplectically rigid differential operators with three regular singularities constitute an important subclass and were classified by Bogner and Reiter~\cite{MR2980467}. In addition to the 14 linearly rigid examples in \cite{MR2282973}, this class includes all operators whose associated monodromy representation is symplectically rigid. Following the results of Deligne~\cite{MR0417174} and Katz~\cite{MR1366651}, there is a necessary and sufficient arithmetic criterion for the generalized  rigidity of a monodromy tuple within any irreducible reductive algebraic subgroup of $\operatorname{GL}(n,\mathbb{C})$: choosing the subgroup to be $\operatorname{GL}(n,\mathbb{C})$ returns the aforementioned notion of linear rigidity, whereas choosing $\operatorname{Sp}(n,\mathbb{C}) \subset \operatorname{GL}(n,\mathbb{C})$ provides us with the more general notion of symplectic rigidity. We know that the elements of monodromy tuples induced by a rank-four Calabi-Yau operator must lie in $\operatorname{Sp}(4,\mathbb{C})$. Bogner and Reiter showed that all of the symplectically rigid monodromy tuples of quasi-unipotent elements admit a decomposition into a sequence of middle convolutions and tensor products of Kummer sheaves of rank one \cite{MR2980467}.  In particular, they are constructible using only tuples of rank-one. Among them, 60 tuples are associated with symplectically rigid Calabi-Yau operators having a maximal unipotent element; they are available in the database of Almkvist et al.\!~\cite{Almkvist:aa},  or AESZ database for short. 
\par 
These results suggest that there should be a geometric explanation for the Bogner and Reiter result, and therefore some kind of ``iterated fibration'' construction, starting with the rank-one Picard-Fuchs operator of the family~(\ref{quadric_pencil}), realizing all symplectically rigid Calabi-Yau operators. The main result of this article is the following:
\begin{theorem}
\label{DoranMalmendier}
All symplectically rigid Calabi-Yau operators having a maximal unipotent element are the Picard-Fuchs operators of families of Jacobian elliptic Calabi-Yau varieties $\pi: X \to \mathbb{P}^1 \backslash \lbrace 0, 1, \infty\rbrace$.  The families are obtained by the iterative twist construction applied to the quadric pencil~(\ref{quadric_pencil}). In particular, for each symplectically rigid differential operator $L_t$ the iterative twist construction produces (1) a family of transcendental cycles $\Sigma(t)$, obtained iteratively from a lower dimensional cycle using a warped product, (2) a holomorphic top-form $\eta_t$ on each fiber $X_t=\pi^{-1}(t)$, represented as a closed differential form, such that the period
\beqn 
  \omega(t) = \int_{\Sigma(t)} \eta_t
\eeqn
is holomorphic on the unit disk about the maximal unipotent monodromy point $t=0$, and solves the Picard-Fuchs equation $L_t \omega(t) =0$.
\end{theorem}
We point out that the restriction to symplectically rigid differential operators in Theorem~\ref{DoranMalmendier} was chosen to simplify this exposition.  For example, our iterative construction also provides a geometric realization of all rank-four, non-rigid Calabi-Yau operators with four regular singular points that were found in \cite{MR3599010}. 
\par The outline of this paper is as follows: in Section~\ref{sec:operators} we recall crucial definitions and properties related to hypergeometric differential operators and Calabi-Yau operators, their behavior under the exterior square operation and the Hadamard product, as well as the notion of rigidity. The details of the proof of our main theorem are quite involved, but the basic idea is simple and already present in a series of examples that we present in Section~\ref{first_example}. In Section~\ref{prelim}, we describe the construction of twisted families with generalized functional invariant in full generality, including several modified variants needed later, and the computation of period integrals. In Section~\ref{sec:modular} we demonstrate that families of Calabi-Yau manifolds obtained by our iterative construction include universal families of elliptic curves over the modular curves for $\Gamma_0(k)$, families of $M_k$-lattice polarized K3 surfaces over the modular curves for $\Gamma_0(k)^+$, and other prominent families of lattice polarized K3 surfaces. In Section~\ref{sec:dwork} we show that a sequence of generalized functional invariants captures all key features of the mirror families of the deformed Fermat pencils, including the existence of (new) elliptic fibrations and relations among their holomorphic periods. In Section~\ref{sec:base_transformations} we apply linear and quadratic transformations to the rational parameter spaces of the twisted families of elliptic curves and K3 surfaces already obtained in previous sections. As we use our twist construction iteratively, applying base transformations between twists turns out to be a crucial step in order to construct a complete set of families realizing \emph{all} symplectically rigid monodromy tuples. The proof of Theorem~\ref{DoranMalmendier} will be completed in Section~\ref{proof}.  In Section~\ref{sec:beyond}, we show that 30 non-rigid Calabi-Yau operators with four singular points are readily obtained by our twist construction as well.
\section*{Acknowledgments}
The authors thank Matt Kerr, Rick Miranda, Dave Morrison, and Noriko Yui for helpful discussions, and the referees for many helpful comments, suggestions, and corrections. The first-named author acknowledges support from the National Sciences and Engineering Research Council, the Pacific Institute for Mathematical Sciences,  and a McCalla Professorship at the University of Alberta. The second-named author acknowledges support from the Kavli Institute for  Theoretical Physics and the University of Alberta's Faculty of Science Visiting Scholar Program.
\section{Hypergeometric and Calabi-Yau type operators}
\label{sec:operators}
\subsection{Hypergeometric functions}
The higher hypergeometric functions ${}_nF_{n-1}$ were introduced by Thomae~\cite{MR1509670} as series
\beq
\label{eqn:nFn-1}
 \hpg{n}{n-1}{\alpha_1, \quad \dots \quad , \alpha_n}{\beta_1, \dots, \beta_{n-1}}{t} = \sum_{k=0}^\infty \dfrac{(\alpha_1)_k \cdots (\alpha_n)_k}{(\beta_1)_k \cdots (\beta_{n-1})_k} \, \frac{t^k}{k!} \,,
\eeq
where $(\alpha)_k = \Gamma(\alpha+k)/\Gamma(\alpha)$ is the Pochhammer symbol. When $n=2$ this is the classical Gauss hypergeometric function.  We always assume that we have rational parameters $\alpha_1, \dots, \alpha_n \in (0,1) \cap \mathbb{Q}$ and $\beta_1, \dots, \beta_{n-1} \in (0,1] \cap \mathbb{Q}$ such that $\alpha_i \not= \beta_j$.   The rank-$n$ and degree-one\footnote{\emph{degree} refers to the highest power in $t$} differential equation satisfied by ${}_nF_{n-1}$ is given by
\beq
\label{hpg_ode}
\Big\lbrack \theta (\theta + \beta_1-1) \cdots (\theta + \beta_{n-1}-1) - t \, (\theta + \alpha_1) \cdots  (\theta + \alpha_n) \Big\rbrack \, \omega(t) = 0 \,,
\eeq
where $\theta=t \frac{d}{dt}$ and there are three regular singular points $t=0,1,\infty$. The differential operator has the Riemann symbol
\beq
\label{RiemannSymbol}
 \mathcal{P}\left. \left(\begin{array}{ccc}0 & 1 & \infty \\ \hline  1-\beta_1 & 0 & \alpha_1 \\ 1-\beta_2 & 1 & \alpha_2\\ \vdots & \vdots & \vdots\\ 1-\beta_{n-1} & n-2 & \alpha_{n-1} \\
 0 & \sum_{j=1}^{n-1} \beta_j -\sum_{j=1}^n \alpha_j & \alpha_n \end{array} \right| t \right) \,.
\eeq
From now on we shall denote the hypergeometric equation~(\ref{hpg_ode}) by 
\beq
\label{ODE}
 L^{(n)}_t\Big( ( \alpha_1, \dots , \alpha_n) ; ( \beta_1, \dots , \beta_{n-1} ) \Big) \, \omega(t) = 0\,.
\eeq
\par In general, given a rank-$n$ differential operator $L_t$ with coefficients in $\mathbb{C}(t)$ and singular locus $S$ with finite cardinality $|S| = r+1$, normalized to include $t=0$ as a singular point, there is an induced rank-$n$ local system $\mathbb{L}$ of solutions on $\mathbb{P}^1 \backslash S$. If all singularities of $L_t$ are regular, we call the operator a \emph{Fuchsian} differential operator. We fix a base point $t_0 \in \mathbb{P}^1 \backslash S$ to obtain the monodromy representation of the fundamental group given by
\beqn
  \pi_1(\mathbb{P}^1 \backslash S, t_0) \to H \subset \operatorname{GL}(\mathbb{L}_{t_0}) \cong \operatorname{GL}(n,\mathbb{C})  \,.
\eeqn
An operator is \emph{irreducible} if the image of its monodromy representation is an irreducible subgroup. If we also fix an orientation and a set of based simple loops, i.e.,
\beqn
 \Big\{ \gamma_s : \, (S^1, *) \to \big(\mathbb{P}^1 \backslash S, t_0\big) \Big\}_{s \in S} \;,
\eeqn
each circling a single point in the singular locus exactly once and an ordering of $S$, we obtain monodromy matrices $g_1, \dots, g_{r+1}$ for the simple loops $\gamma_s$ around the corresponding points together with the relation $g_1  \cdots g_{r+1}=\mathbb{I}$. The latter follows because the product of all paths is a path encircling all of $S$, whence homotopic to the trivial path. The collection of monodromy matrices $T=(g_1, \dots, g_{r+1})$ is called a \emph{monodromy tuple} of rank $n$. Clearly, $H$ is determined by the tuple of matrices $T=(g_1, \dots, g_{r+1})$ with $g_i \in \operatorname{GL}(n,\mathbb{C})$ whose product is the identity, up to global conjugation, i.e., mapping $g_i \mapsto h \cdot g_i \cdot h^{-1}$ for a \emph{single} $h \in \operatorname{GL}(n,\mathbb{C})$. We call a monodromy representation \emph{linearly rigid} if the elements of the monodromy tuple are quasi-unipotent, generate an irreducible subgroup, and are completely determined by their individual conjugacy classes, i.e., Jordan forms.
\par For the generalized hypergeometric function this is the case if $\alpha_i-\beta_j \not \in \mathbb{Z}$ for all $i, j$; see \cite{MR974906}. Therefore, we have the following:
\begin{theorem}[\cite{MR974906}]
The operator $L^{(n)}_t\big((\alpha_1, \dots, \alpha_n);(1,\dots,1)\big)$ with $\alpha_i \in (0,1)\cap \mathbb{Q}$ and $\beta_j=1$ in Equation~(\ref{hpg_ode}) is a rank-$n$ Fuchsian differential operator with the three regular singular points $t=0,1, \infty$ such that $t=0$ is a point of maximally unipotent monodromy and the monodromy representation is linearly rigid.
\end{theorem}
For rank $n=1$, we have
\beq
\label{period0}
 {}_1F_0\!\left(\left. \alpha \right| t \right)   = \big( 1- t \big)^{-\alpha}\;,
\eeq
and the differential operator $L^{(1)}_t(1/2;)$ gives rise to a monodromy tuple of rank one, with monodromies $1,-1$, and $-1$ around the points $t=0,1,\infty$. 
\par The monodromy representation for the differential operator~(\ref{ODE}) and the corresponding differential Galois group -- which carries all information about algebraic relations between the solutions -- were classified by Beukers and Heckman~\cite{MR974906}. 
\subsection{Convolution formulas}
The \emph{Hadamard product} of two power series $f(t) =\sum_{n \ge 0} f_n \, t^n$ and $g(t) =\sum_{n \ge 0} g_n \, t^n$ is defined by   $(f \star g)(t):=\sum_{n \ge 0} f_n \, g_n \, t^n$. Using the Hadamard product, the following cancellation in the coefficients of convergent hypergeometric series is easily observed:
\begin{small}  \begin{gather}
\nonumber
\hpg{n}{n-1}{\alpha_1, \quad \dots \quad , \alpha_n}{\rho_1, \dots, \rho_l, \beta_1, \dots, \beta_{n-l-1}}{t} \star  \hpg{m}{m-1}{\rho_1, \dots , \rho_l, \alpha'_1, \dots, \alpha'_{m-l}}{\gamma_1, \dots, \gamma_{m-1}}{t} \\
\label{eqn:cancellation}
 = \hpg{m+n-l}{m+n-l-1}{\alpha_1, \dots , \alpha_n, \alpha'_1, \dots, \alpha'_{m-l}}{\beta_{1}, \dots, \beta_{n-l-1}, \gamma_1, \dots, \gamma_{m-1},1}{t}
\end{gather}  \end{small}%
with $\alpha_i, \alpha'_{i'} \not= \beta_{j}, \gamma_{j'}$.  The Hadamard product is used in explicit formulas for certain integral convolutions. We have the following:
\begin{lemma}
\label{lem:convolution}
For a function $\omega(t)$ which is holomorphic on the disc of radius $1$ about $t=0$ and has the absolutely convergent series $\omega(t) = \sum_{k\ge 0} f_k t^k$ for $|t|<1$ and $\alpha \in (0,1) \cap \mathbb{Q}$, we have the following convolution formulas:
\beq
\label{eqn:pure_twist}
  \int_0^1 \dfrac{dv}{v^{1-\alpha} \, (1-v)^{\alpha}} \; \omega(tv) = \pi \csc{(\pi\alpha)} \, \sum_{n \ge 0} \frac{f_n \, \left(\alpha\right)_n}{n!} \, t^n =  \pi \csc{(\pi\alpha)} \;  {}_1F_0\!\left(\left. \alpha \right| t \right)  \star \omega(t)\,,
\eeq
and
\beq
\label{eqn:pure_twist_LowerRank}
  \int_{-1}^1 \dfrac{dv}{\sqrt{1-v^2}} \; \omega(tv)  = \pi \, \sum_{n \ge 0} \frac{f_{2n} \, \left(\frac{1}{2}\right)_n}{n!} \, t^{2n} = \pi \; {}_1F_0\!\left(\left. \frac{1}{2} \right|t^2\right)\star \omega(t)\,.
\eeq
\end{lemma}
\begin{proof}
The first identity easily follows from
\beqn
 \int_0^1 t^{a-1} (1-t)^{b-1} = \frac{\Gamma(a)\Gamma(b)}{\Gamma(a+b)}\,,
\eeqn
where $\operatorname{Re}(a), \operatorname{Re}(b)>0$, the formula $(1+z)^{-k}=\sum_{l \ge 0} \frac{\Gamma(l+k)}{\Gamma(k)\Gamma(l+1)} (-z)^l$, and the reflection formula $\Gamma(\alpha)\Gamma(1-\alpha)=\pi \csc{(\pi\alpha)}$. For the second equation we observe that
\beq
  \int_{-1}^1 \dfrac{dv}{\sqrt{1-v^2}} \; \omega(tv)  = \int_{0}^1 \dfrac{dw}{2\sqrt{w(1-w)}} \; \left( \omega(t\sqrt{w}) + \omega(-t\sqrt{w}) \right) .
\eeq
\end{proof}
The integral convolution in Equation~(\ref{eqn:pure_twist}) and~(\ref{eqn:pure_twist_LowerRank}) is also called Euler transform. That is, the Euler transform is an integral transform with parameter that relates (the holomorphic solution of) a Fuchsian differential equation of rank $n$ with three regular singularities to a Fuchsian differential equation of rank $(n+1)$ with three regular singularities. We have the following:
\begin{corollary}
In the situation above, we have
\begin{small}  \begin{gather}
\label{Euler-Integral}
 \hpg{n+1}{n}{  \alpha_1 \;  \dots \; \alpha_{n+1} }{ 1 \; \dots\; 1}{ t}
 =  {}_1F_0\!\left(\left. \alpha_{1} \right| t \right) \star \hpg{n}{n-1}{  \alpha_2 \, \dots \, \alpha_{n+1} }{ 1 \, \dots\, 1}{ t}\\[0.5em]
 \label{Euler-Integral_B}
 =  {}_1F_0\!\left(\left. \alpha_{1} \right| t \right)  \star \dots \star  {}_1F_0\!\left(\left. \alpha_{n+1} \right| t \right) \\
= \left\lbrack\prod_{i=1}^{n} \frac{1}{\pi \csc{(\pi\alpha_i)}} \int_0^1 \! \frac{dz_i}{z_i^{1-\alpha_i} (1-z_i)^{\alpha_i}}\right\rbrack (1-t \, z_1 \cdots z_{n-1})^{-\alpha_{n+1}} \,.
\end{gather}  \end{small}%
\end{corollary}
\qed
\par The Hadamard product in Equation~(\ref{Euler-Integral})  of the hypergeometric function ${}_nF_{n-1}$ with ${}_1F_0$ turns the holomorphic solution of $L^{(n)}_t((\alpha_2, \dots, \alpha_{n+1});(1,\dots,1)) \, \omega(t)=0$ into the holomorphic solution of $L^{(n+1)}_t((\alpha_1, \dots, \alpha_{n+1});(1,\dots,1)) \, \tilde{\omega}(t)=0$. There is a corresponding notion of the Hadamard product (cf.\!~\cite{MR2980467}*{Def.~4.11}) for the differential operators involved: the Hadamard product of the differential operator $L^{(1)}_t(\alpha_1;)$ with the operator $L^{(n)}_t((\alpha_2, \dots, \alpha_{n+1});(1,\dots,1))$ yields the differential operator $L^{(n+1)}_t((\alpha_1, \dots, \alpha_{n+1});(1,\dots,1))$. In~\cite{MR2980467}, this was denoted by $\mathcal{H}_{\alpha_1}(L^{(n)}_t)=L^{(n+1)}_t$, and a corresponding operation, known as \emph{middle Hadamard product}, was introduced for the monodromy tuple induced by a differential operator such that the monodromy tuple induced by $L^{(n+1)}_t$ becomes a sub-factor in the middle Hadamard product of the monodromy tuple induced by $L^{(n)}_t$. We make the following:
\begin{definition}
Differential operators are said to be of geometric origin if they are Picard-Fuchs operators annihilating the periods of a family of complex algebraic varieties. A monodromy tuple is of geometric origin, if it is induced by a differential operator of geometric origin.
\end{definition}
\par Equation~(\ref{Euler-Integral_B}) decomposes the local system of $L^{(n+1)}_t((\alpha_1, \dots, \alpha_{n+1});(1,\dots,1))$ into the convolution of $n+1$ local systems of rank one, each with a holomorphic solution of the type in Equation~(\ref{period0}) with $\alpha=\alpha_i$. This is a special case of a general classification result by Katz~\cite{MR1366651} that applies to every linearly rigid local system; see also ~\cite{MR2980467}. In fact, Katz proved that every linearly rigid local system is obtained as tensor products and convolutions of rank-one local systems associated with the holomorphic solution~(\ref{period0}).  More general, it is known that these operations on the level Fuchsian local systems and monodromy tuples preserve the geometric origin of an operator; see \cite{MR2363121}. However, as we will prove in this article, such a decomposition into rank-one local systems is not necessarily meaningful in terms of geometry. For example, the classical rank-two local system for $L^{(2)}_t((\mu, 1-\mu);(1))$ with $\mu \in \{ \frac{1}{3}, \frac{1}{4}, \frac{1}{6} \}$ is decomposed using Katz' procedure into two rank-one systems using the Hadamard product
\beq
\label{rk2hgf_b}
 \hpg21{ \mu, 1-\mu}{1}{t} =  {}_1F_0\!\left(\left. \mu \right| t\right) \star  {}_1F_0\!\left(\left. 1-\mu \right| t\right) \,.
\eeq
However, this decomposition is not the one to be used if one wants to relate a period of a zero-dimensional family of Calabi-Yau manifolds to a period of a family of elliptic curves. Instead one has to use the decomposition formula
\beq
\label{rk2hgf_bb}
 \hpg21{ \mu, 1-\mu}{1}{t} = \hpg21{ \mu, 1-\mu}{\frac{1}{2}}{t}  \star {}_1F_0\!\left(\left. \frac{1}{2} \right|t \right) \,.
\eeq
Equation~(\ref{rk2hgf_bb}) then allows to build directly families of elliptic curves whose Picard-Fuchs operator is $L^{(2)}_t((\mu, 1-\mu);(1))$ for every $\mu \in \{ \frac{1}{3}, \frac{1}{4}, \frac{1}{6} \}$ from a single geometric object, the family in Equation~(\ref{legendre0b}) with Picard-Fuchs operator $L^{(1)}_t(1/2 ;)$, using a generalized functional invariant which determines $\mu$; see Lemma~\ref{lemmaECmixed}.
\subsection{Rigid Calabi-Yau operators}
Doran and Morgan classified in \cite{MR2282973} all integral weight-three variations of Hodge structure which can underlie a family of Calabi-Yau threefolds over the thrice-punctured sphere $\mathbb{P}^1 \backslash \lbrace 0, 1, \infty\rbrace$ -- subject to conditions on monodromy coming from mirror symmetry -- through the irreducible monodromy representation generated by the local monodromies around the punctures of the base space. The monodromy representations turned out to be identical with the linearly rigid monodromy groups associated with the univariate generalized hypergeometric operators $L^{(4)}_t((\alpha_1, \dots, \alpha_4);(1, \dots, 1))$ with certain rational coefficients $\alpha_1, \dots, \alpha_4$.  
\par Each case was realized as a family of, possibly singular, Calabi-Yau threefolds constructed as hypersurfaces or complete intersections in a Gorenstein toric Fano variety and Calabi-Yau threefolds fibered by high rank K3’s by Clingher et al.~ \cite{Clingher:2013aa} -- where a non-generic geometric transition was needed in one of the cases.  The construction of the 14 cases lead to the following definition of a Calabi-Yau type differential operator (or \emph{Calabi-Yau operator} for short) by Almkvist, van Enckevort, van Straten and Zudilin:
\begin{definition}
A rank-$n$ Calabi-Yau operator is an irreducible Fuchsian differential operator $L^{(n)}_t$ of rank $n$ with coefficients in $\mathbb{C}(t)$ and singular locus $S$ (with only regular singular points), normalized to include $t=0$, such that (1) the monodromy at $t=0$ is maximally unipotent, (2) $L^{(n)}_t$ is self-adjoint, i.e., there is a function $h(t)\not =0$ algebraic over $\mathbb{Q}(t)$ such that $L^{(n)}_t h(t) = (-1)^n h(t) L^{(n) \, \dagger}_t$ where $L^{(n) \, \dagger}_t$ denotes the adjoint of $L^{(n)}_t$, and (3) $L^{(n)}_t \, \omega(t)=0$ has an $N$-integral holomorphic solution $\omega(t)=\sum_{k\ge0} f_k t^k$ at $t=0$, i.e., there exists $N\in \mathbb{N}$ such that $f_k N^k \in \mathbb{N}$ for all $k$.
\end{definition}
\begin{remark}
Condition~(1) is equivalent to all exponents of the Riemann symbol for $L^{(n)}_t$ at $t=0$ being zero. For a general  linear, rank-$n$, Fuchsian differential operator in the variable $t$, given by
\beqn
 L^{(n)}_t  = \partial^n + \sum_{i=0}^{n-1} a_i(t) \; \partial^i \;,
\eeqn
with $\partial=\frac{d}{dt}$ and suitable rational coefficient functions $a_i(t)$ for $1 \le i \le n-1$, the formal adjoint operator is
\beqn
 L^{(n) \, \dagger}_t = \partial^n + \sum_{i=0}^{n-1} (-1)^{n+i} \; \partial^i a_i(t) \,.
\eeqn
The condition of being self-adjoint implies, as a necessary condition, that the function $h(t)$ satisfies the differential equation
\beqn
 h'(t) = - \frac{2}{n} \, a_{n-1}(t) \, h(t) \,.
\eeqn
Condition~(2) implies that for $n$ even the differential Galois group of $L^{(n)}_t$ is contained in $\operatorname{Sp}(n,\mathbb{C})$. Condition~(3) implies that the monodromy matrices at each singular points are quasi-unipotent.
\end{remark}
\par For a monodromy $r$-tuple $T=(g_1, \dots, g_r)$ with $g_i \in G$ where $G$ is a reductive complex algebraic group, we define the \emph{rigidity index of $T$ in $G$} as
\beq
 i_G(T) = \sum_{i=1}^r \operatorname{codim} C_G(g_i) - 2 \, \dim{G} + 2 \, \dim{Z_G}\,,
\eeq
where $C_G(g_i)$ denotes the centralizer of $g_i$ in $G$, and $Z_G$ denotes the center of $G$. Deligne and Katz gave the following criterion:
\begin{proposition}[\cite{MR1366651}]
For $G=\operatorname{GL}(n,\mathbb{C})$ the monodromy $r$-tuple $T=(g_1, \dots, g_r)$  is linearly rigid if and only if $i_G(T) =0$.
\end{proposition}
\par Therefore, one can extend the notion of rigidity from $\operatorname{GL}(n,\mathbb{C})$ to any reductive complex algebraic group $G$ by considering the monodromy tuples with $ i_G(T) =0$. In particular, since elements of monodromy tuples induced by a fourth order Calabi-Yau differential operator lie in $\operatorname{Sp}(4,\mathbb{C})$, we investigate those Calabi-Yau operators inducing an $\operatorname{Sp}(4,\mathbb{C})$-rigid monodromy representation, or \emph{symplectically rigid} for short. 
\par Bogner and Reiter \cite{MR2980467} proved that all $\operatorname{Sp}(4,\mathbb{C})$-rigid monodromy tuples consisting of quasi-unipotent elements can be constructed using tensor products, rational pullbacks and the middle convolution \cite{MR2980467}*{Thm.~3.1} of rank-one local systems associated with the holomorphic solution of the type in Equation~(\ref{period0}).  Simpson had already classified in \cite{MR1158474} all irreducible, $\operatorname{Sp}(4,\mathbb{C})$-rigid monodromy tuples with quasi-unipotent elements \emph{and} one maximally unipotent matrix. It turns out \cite{MR1158474}*{Thm.~4} that these tuples necessarily have three matrices, analogous to Picard-Fuchs operators of families of Calabi-Yau threefolds over the thrice-punctured sphere $\mathbb{P}^1 \backslash \{0,1\infty\}$ -- subject to conditions on monodromy coming from mirror symmetry. Moreover, Simpson divided up these tuples into four families, called the \emph{hypergeometric}, \emph{odd}, \emph{even}, and \emph{extra} family. Bogner and Reiter found that these tuples are realized as monodromy tuples of Fuchsian differential operators. We have the following:
\begin{corollary}[\cite{MR2980467}]
\label{cor:BognerReiter}
There are $60$ rank-four Calabi-Yau operators with three singularities whose associated symplectically rigid monodromy representations are generated by tuples $T=(g_0, g_1, g_\infty)$ with quasi-unipotent elements, having one maximally unipotent matrix, and satisfy $i_G(T) =0$ for $G=\operatorname{Sp}(4,\mathbb{C})$.
\end{corollary}
\begin{remark}
The 60 symplectically rigid Calabi-Yau operators with three singularities are found in the AESZ database~\cite{Almkvist:aa}.  Among them, 14 Calabi-Yau operators are univariate generalized hypergeometric operators $L^{(4)}_t\big((\alpha_1, \dots, \alpha_4);(1, \dots, 1)\big)$ with certain rational coefficients $\alpha_1, \dots, \alpha_4$ determined by Doran and Morgan \cite{MR2282973}.
\end{remark}
\begin{remark}
As we demonstrate in Section~\ref{sec:beyond}, our geometric twist construction can also produce families of Calabi-Yau threefolds whose Picard-Fuchs operators realize all monodromy tuples of low degree with four quasi-unipotent elements and one maximally unipotent matrix.  On the other hand, there are families of Calabi-Yau threefolds whose Picard-Fuchs operators have no point of maximally unipotent monodromy and do not underlie variations of Hodge structure of type $(1, 1, 1, 1)$~\cite{MR2807280}.
\end{remark}
\par The decomposition of the $\operatorname{Sp}(4,\mathbb{C})$-rigid monodromy tuples by Bogner and Reiter suffers from the same problem the decomposition of linearly rigid systems by Katz does -- we demonstrated that with Equation~(\ref{rk2hgf_b}) versus Equation~(\ref{rk2hgf_bb}): whereas it does imply that the symplectically rigid operators are of geometric origin, the decomposition is not constructive in the sense that it produces families of Calabi-Yau manifolds whose Picard-Fuchs operators realize them. In contrast, our iterative twist construction in Section~\ref{prelim} will build  from a single geometric object, a family that we will present in Equation~(\ref{legendre0b}), the families of Calabi-Yau varieties using a generalized functional invariant such that their Picard-Fuchs operators realize all $60$ cases in Corollary~\ref{cor:BognerReiter}.
\subsubsection{The Yifan-Yang pullback}
As explained above, among the $60$ rank-four Calabi-Yau operators with three singularities and symplectically rigid monodromy, 14 Calabi-Yau operators belong to the univariate generalized hypergeometric operators. Another 14 rank-four Calabi-Yau operators are uniquely determined by the fact that they are of degree two in $t$ and their exterior squares are the univariate rank-five hypergeometric operators $L^{(5)}_t\big((\alpha_1, \alpha_2, 1/2, \alpha_3, \alpha_4);(1, \dots, 1)\big)$ for certain rational coefficients $\alpha_1, \dots, \alpha_4$.
\par For a linear differential operator $L^{(n)}_t$ with linearly independent solutions $y_1(t)$, $\dots$, $y_n(t)$, the \emph{exterior square} is the linear differential operator of minimal rank with solutions $y_i(t) y'_j(t) - y'_i(t) y_j(t)$ for all $1\le i < j \le n$. The exterior square of a general differential operator $L^{(4)}_t$ is a rank-six differential operator of the form $\partial^6 + \frac{1}{M(t)} \sum_{i=0}^{5} b_i(t) \, \partial^i$  where $M(t)$ is the right side of Equation~(\ref{Calabi-Yau-cond4}). If $M(t)=0$ for all $t$, then we say that the exterior square of $L^{(4)}_t$ is a rank-five operator.
\par For a general rank-four differential operator $L^{(4)}_t  = \partial^4 + \sum_{i=0}^{3} a_i(t) \, \partial^i$ we have the following:
\begin{lemma}
\label{lemma:ODE_rank4}
The following statements are equivalent:
\begin{enumerate}
\item The operator $L^{(4)}_t$ is self-adjoint.
\item The exterior square of $L^{(4)}_t$ is a rank-five operator.
\item The monodromy group of the operator $L^{(4)}_t$ is a discrete subgroup of $\mathrm{Sp}(4,\mathbb{R})$.
\item  The following condition for the coefficients of $L^{(4)}_t$ holds
\beq
\label{Calabi-Yau-cond4}
 0 = 8 \, a_1(t) - 8 \frac{da_2(t)}{dt} + 4 \frac{d^2a_3(t)}{dt^2} - 4 \, a_2(t) \, a_3(t) + 6 \, a_3(t) \, \frac{da_3(t)}{dt} + a_3(t)^3 \,.
\eeq
\end{enumerate}
\end{lemma}
\begin{proof}
The equivalence of (1), (2), (4) follows by an explicit computation. The equivalence with condition (3) was proved in \cite{MR2980467}.
\end{proof}
\par Similarly, for a general rank-five differential operator $L^{(5)}_t  = \partial^5 + \sum_{i=0}^{4} b_i(t) \, \partial^i$ we have the following:
\begin{lemma}
\label{lem:sqrtL5}
The following statements are equivalent:
\begin{enumerate}
\item The operator $L^{(5)}_t$ is self-adjoint.
\item The operator $L^{(5)}_t$ is the exterior square of a rank-four self-adjoint operator.
\item The monodromy group of the operator $L^{(5)}_t$ is a discrete subgroup of $\mathrm{Sp}(4,\mathbb{R})$.
\item  The following two conditions for the coefficients of $L^{(5)}_t$ hold:
\beq
\label{Calabi-Yau-cond}
 b_2(t) = \frac{3}{2} \, \frac{db_3(t)}{dt}+\frac{3}{5} \, b_4(t) \, b_3(t) - \frac{d^2b_4(t)}{dt^2} - \frac{6}{5} b_4(t) \, \frac{db_4(t)}{dt} - \frac{4}{25} b_4(t)^3
\eeq
and 
\beq
\label{Ext-Sqr-cond}
 \begin{split}
  b_0(t) & = \frac{1}{5} \frac{d^4b_4(t)}{dt^4} - \frac{1}{4} \frac{d^3b_3(t)}{dt^3} + \frac{2}{5} b_4(t) \, \frac{d^3b_4(t)}{dt^3} - \frac{3}{10} b_4(t) \, \frac{d^2b_3(t)}{dt^2}\\
  & + \left( \frac{8}{25} b_4(t)^2 + \frac{4}{5} \frac{db_4(t)}{dt} - \frac{1}{10} b_3(t) \right) \, \frac{d^2b_4(t)}{dt^2} + \frac{1}{2} \frac{db_1(t)}{dt} \\
  & + \left( -\frac{3}{25} b_4(t)^2 - \frac{3}{10} \frac{db_4(t)}{dt} \right) \frac{db_3(t)}{dt} + \frac{12}{25} b_4(t) \, \left(\frac{db_4(t)}{dt}\right)^2\\
  & + \left( - \frac{3}{25} \, b_3(t) \, b_4(t) + \frac{16}{125} \, b_4(t)^3 \right) \, \frac{db_4(t)}{dt} - \frac{2}{125} b_3(t) \, b_4(t)^3 \\
  & + \frac{1}{5} \, b_1(t) \, b_4(t) + \frac{16}{3125} \, b_4(t)^5 \,.
 \end{split}
\eeq
\end{enumerate}
\end{lemma}
\begin{proof}
The equivalence of (1), (2), (4) follows by an explicit computation. Yang and Zudilin \cite{MR2731075} proved that $L^{(5)}_t$ has  a projective monodromy group that is a discrete subgroup of $\mathrm{Sp}(4,\mathbb{R})$ if and only if $L^{(5)}_t$ satisfies conditions  (\ref{Calabi-Yau-cond}) and (\ref{Ext-Sqr-cond}).  In fact, we have the following identification with the polynomials $p_1$, $p_2$, and $p_3$ used in \cite{MR2731075}*{Theorem 4}:
\beq
\begin{split}
 \frac{p_1(t)}{t} & = \frac{1}{10}  b_4(t) - \frac{1}{t} \;,\\
 \frac{p_2(t)}{t^2} & = \frac{1}{5} b_3(t) - \frac{1}{5} \frac{db_4(t)}{dt} - \frac{7}{100} b_4(t)^2 \;,\\
 \frac{p_3(t)}{t^4} & = \frac{1}{250} b_3(t) \, b_4(t)^2 - \frac{3}{10} b_4(t) \, \frac{db_3(t)}{dt} + \frac{1}{50} b_3(t) \, \frac{db_4(t)}{dt} 
 + \frac{17}{50} \left( \frac{db_4(t)}{dt} \right)^2\\ 
 & + \frac{2}{5} \frac{d^3b_4(t)}{dt^3} + \frac{29}{125} b_4(t)^2 \frac{db_4(t)}{dt} + \frac{14}{25} b_4(t) \, \frac{d^2 b_4(t)}{dt^2}+ \frac{9}{1250} b_4(t)^4 \\
 & - \frac{2}{25} b_3(t)^2 + \frac{1}{2} b_1(t) - \frac{9}{20} \frac{d^2b_3(t)}{dt^2}\,.
\end{split}
\eeq
\end{proof}
We make the following:
\begin{remark}
The D-module associated with the differential operator $L^{(4)}_t$ underlies a variation of Hodge Structure $\mathcal{V}$ of rank four and weight three, corresponding to the standard representation of its Mumford-Tate group $\mathrm{Sp}(4,\mathbb{R})$. Since the exterior square of $L^{(4)}_t$ decomposes into a product of rank-five and rank-one operator, the exterior square of the D-module decomposes into a five-dimensional irreducible and a one-dimensional irreducible representation, and $\wedge^2 \mathcal{V}$ decomposes accordingly. The five-dimensional sub-factor has weight six but level four, so the Tate twist $\wedge^2 \mathcal{V}(1)$ has weight four and type $(1, 1, 1, 1, 1)$.
\end{remark}
\par We obtain the following:
\begin{corollary}
\label{cor:YYsqrt}
For a self-adjoint rank-five operator $L^{(5)}_t  = \partial^5 + \sum_{i=0}^{4} b_i(t) \, \partial^i$, a rank-four self-adjoint 
differential operator $L^{(4)}_t$ whose exterior square equals $L^{(5)}_t$  is given by
\beq
\begin{split}
a_{{3}} \left( t \right) & =\frac{2}{5}\,b_{{4}} \left( t \right) \;,\\
a_{{2}} \left( t \right) & =-{\frac {7\, b_{{4}} \left( t
 \right)^{2}}{50}}-\frac{2}{5}\,{\frac d{dt}}b_{{4}}
 \left( t \right) +\frac{1}{2}\,b_{{3}} \left( t \right) \;,\\
 a_{{1}} \left( t \right) & =-{\frac {9\, b_{{4}} \left( t
 \right)^{3}}{250}}-{\frac {12\,b_{{4}} \left( t \right) {
\frac d{dt}}b_{{4}} \left( t \right) }{25}}+\frac{1}{10}\,b_{{4}}
 \left( t \right) b_{{3}} \left( t \right) -\frac{3}{5}\,{\frac {d^{2}}{
d{t}^{2}}}b_{{4}} \left( t \right) +\frac{1}{2}\,{\frac d{dt
}}b_{{3}} \left( t \right),\\
a_{{0}} \left( t \right) & =-\frac{2}{5}\,{\frac {d^{3}}{d{t}^{3}}}b
_{{4}} \left( t \right) +\frac{3}{8}\,{\frac {d^{2}}{d{t}^{2}}}b_{
{3}} \left( t \right) -{\frac {23\,b_{{4}} \left( t \right) {\frac {
d^{2}}{d{t}^{2}}}b_{{4}} \left( t \right) }{50}}
 +\frac{1}{5}\,b_{{4}} \left( t \right) {\frac d{dt}}b_{{3}} \left( t
 \right) \\
 & -{\frac {27\, \left( {\frac d{dt}}b_{{4}} \left( 
t \right)  \right) ^{2}}{100}}+ \left( -{\frac {18\, b_{{4}}
 \left( t \right)^{2}}{125}}-\frac{1}{20}\,b_{{3}} \left( t \right) 
 \right) {\frac d{dt}}b_{{4}} \left( t \right) -{\frac {19
\, b_{{4}} \left( t \right) ^{4}}{10000}}\\
&-{\frac {3\,
 b_{{4}} \left( t \right)^{2}b_{{3}} \left( t
 \right) }{200}}+\frac{1}{16}\, b_{{3}} \left( t \right)^{2}-
\frac{1}{4} \,b_{{1}} \left( t \right) \,.
\end{split}
\eeq
\end{corollary}
\begin{proof}
The proof follows from an explicit computation using Lemmas~\ref{lemma:ODE_rank4} and~\ref{lem:sqrtL5}.
\end{proof}
\par For a self-adjoint rank-five operator $L^{(5)}_t$ satisfying conditions (\ref{Calabi-Yau-cond}) and (\ref{Ext-Sqr-cond}), we denote the rank-four, self-adjoint, linear differential operator $L^{(4)}_t$ in Corollary~\ref{cor:YYsqrt} by $L^{(4)}_t = \vee_2 L^{(5)}_t$. Following \cite{math/0612215}, the latter is also called the \emph{Yifan-Yang pullback} of $L^{(5)}_t$. The following is easy to check:
\begin{lemma}
\label{lem:scalingL}
The rank-five operator $L^{(5)}_t$ satisfies conditions~(\ref{Calabi-Yau-cond}) and (\ref{Ext-Sqr-cond}) if and only if for any algebraic function $g(t)$ the operator 
\beqn
 L^{(5), \, \langle 2  g(t)\rangle}_t := e^{2 \, g(t)} \, L^{(5)}_t  \,e^{-2 \, g(t)}
 \eeqn
 does. Moreover, the operator $\vee_2 L^{(5), \, \langle 2  g(t)\rangle}_t$ coincides with  $L^{(4), \, \langle g(t)\rangle}_t := e^{g(t)} \, L^{(4)}_t  \,e^{-g(t)}$.
\end{lemma}
\begin{proof}
The proof follows by an explicit computation.
\end{proof}
\par We then have the following:
\begin{proposition}
\label{YY5operators} 
The hypergeometric operator $L^{(5)}_t\big( (\alpha_1, \dots, \alpha_5) ; \big(1,\dots,1)\big)$ with $\alpha_1=1-\alpha_5$, $\alpha_2=1-\alpha_4$, $\alpha_3=\frac{1}{2}$, and $\alpha_1=p, \alpha_2=q$, $p, q \in (0,1) \cap \mathbb{Q}$ is a rank-five, self-adjoint differential operator whose Yifan-Yang pullback $L^{(4), \, \langle g(t)\rangle}_t$ is given by 
\beq
\label{opYY1}
\begin{split}
& \theta^{4} 
-  \frac{1}{4} \, t\, \Big( 8 \, \theta^4 + 16 \, \theta^3 - 2 \, ({p}^{2}+q^2-p-q-9) \, \theta^{2}  \\
   & \quad  - 2\, (p^{2}+q^2-p-q-5)\, \theta +  
    2  +p+ q - p q - {p}^{2}- q^2 +p^2 q +p\,{q}^{2} +{p}^{2}{q}^{2} \Big)\\
+ & \frac{1}{16}\,t^2 \, \left( 2\,\theta+2+p-q \right)  \left( 2\,\theta+1+
p+q \right)  \left( 2\,\theta+2-p+q \right) 
 \left( 2\,\theta+3-p-q \right) \,,
 \end{split}
 \eeq
 with $\exp{(-g(t))} = \sqrt[4]{ t^2 \, (t-1)^3}$. In particular, $\exp{(-g(t))} = \sqrt[4]{ t^2 \, (t-1)^3}$ is the unique non-trivial function (up to scaling) that minimizes the degree (in $t$) of $L^{(4), \, \langle g(t)\rangle}_t$.
\end{proposition}
\begin{proof}
Using Lemma~\ref{lem:sqrtL5} and~\ref{lem:scalingL}, the proof follows from an explicit computation.
\end{proof}
\section{First examples from quadratic twists}
\label{first_example}
The details of the proof of Theorem~\ref{DoranMalmendier} are quite involved, but the basic idea is simple and present in the following series of examples for our iterative construction: One starts with a family of pairs of points and produces, by a quadratic twist, a family of elliptic curves whose total space is an extremal rational surface. One continues by constructing a family of Jacobian elliptic K3 surfaces of Picard rank $19$, and in turn, a family of  Calabi-Yau threefolds with $h^{2,1}=1$ from the family of Jacobian elliptic K3 surfaces by two more quadratic twists. If one allows for the Picard rank of the K3 surfaces in the intermediate step to drop from $19$ to $18$, one can also construct a second, closely related family of Calabi-Yau threefolds with $h^{2,1}=1$. The Picard-Fuchs operators for the two families of threefolds realize two simple rank-four and degree-one symplectically rigid Calabi-Yau operators in the AESZ database~\cite{Almkvist:aa}. 
\par The construction of these examples was motivated by physics, in particular the embedding of gauge theory into F-theory~\cites{MR2003960, MR2854198, MR3366121}. An interpretation from the point of view of variations of Hodge structure might be provided by methods in \cites{MR2457736, MR2601630,MR3484368}. The idea of using a quadratic twist to construct an isomorphism between different types of moduli problems also appeared in the work of Besser and Livn\'e in \cite{MR3156417}.
\subsection{A sequence of quadratic twists}
\label{FirstExample}
We start with a pencil of `dimension zero' Calabi-Yau manifolds which consists of the ramified family of pairs of points $\pm y_0$ given by
\beq
\label{legendre0}
 y_0^2 = 1 - t 
\eeq
for $t \in \mathbb{C}$. For this family, we define $\Sigma_0(t)$ to be the point $t$, take the branch cut  branch cut along $\lbrace \, t \, | \, 1\le t \le \infty\rbrace$, and consider the holomorphic $0$-form $1/y_0$ and the period
\beq
\label{period00}
 \omega(t) = \int_{\Sigma_0(t)} \frac{1}{y_0} = \frac{2}{y_0} =   2 \; {}_1F_0\!\left(\left. \frac{1}{2} \right| t \right)  =2\, \big( 1- t \big)^{-\frac{1}{2}}\;,
\eeq
which is a solution of the hypergeometric differential equation $L^{(1)}_t(\frac{1}{2};) \, \omega(t)=0$.
\par To obtain from the family of points~(\ref{legendre0}) a pencil of elliptic curves, one promotes the family parameter $t$ to an additional  complex variable $x$ and carries out the quadratic twist $y_0^2 \mapsto -y_1^2/[x(x-t)]$ in Equation~(\ref{legendre0}). This yields the classical Legendre pencil of elliptic curves given by
\beq
\label{legendre1}
 y_1^2 = x \, (x-1) \, (x-t) \;,
\eeq
where $t \in \mathbb{P}^1 \backslash \lbrace 0,1, \infty\rbrace$ and $\pi: X_t \to \mathbb{P}^1$ is the corresponding projection.  The polarized Hodge filtration on $H_\mathbb{Z} =H^1(X_t,\mathbb{Z})$  of the elliptic curve $X_t$ has two steps, $F^0$ and $F^1$ defining a pure Hodge structure of weight one and type $(1,1)$ in
\beqn
 \left \lbrace H_\mathbb{Z}, Q, F^1 \subset F^0 = H_\mathbb{Z} \otimes \mathbb{C} \right \rbrace \,.
\eeqn
Here,  $F^0$ is the entire cohomology group, and $F^1$ is $H^{1,0}(X_t)$, the one-dimensional space of holomorphic harmonic one-forms. The polarization $Q$ is the natural non-degenerate, integer, bilinear form on $H_\mathbb{Z}$ derived from the cup product and varies holomorphically. The homology group of the elliptic curve is free of rank two, and the periods of $dx/y_1$ satisfy a second-order differential  equation. In fact, Equation~(\ref{legendre1}) defines a double covering of $\mathbb{P}^1$ branched at the four points $x=0,1,t,\infty$. We cut the Riemann sphere from $0$ to $1$ and from $t$ to $\infty$. The two cuts are opened up into two ovals, and the two $y$-sheets are glued with opposite orientations to obtain an elliptic curve. The A-cycle $\Sigma_1(t)$ projects onto the closed cycle encircling the branching points at $x=0$ and $x=t$. Then, flattening out the cycle, we obtain
\beq
\label{period1}
\omega(t) = \oint_{\Sigma_1(t)} \dfrac{dx}{y_1} =  2 \int_0^t  \dfrac{dx}{\sqrt{x\, (x-1) \, (x-t)}}  
= 2 \int_0^1  \dfrac{dx}{\sqrt{x\, (1-x) \, (1-t x)}} \,,
\eeq
that is, the Euler integral representation of $(2\pi)\, {}_2F_1( \frac{1}{2}, \frac{1}{2}; 1| t)$ and satisfies the differential equation $L^{(2)}_t((\frac{1}{2}, \frac{1}{2}); (1)) \, \omega(t)=0$. Alternatively, we can take a Pochhammer contour $C_{\lbrace 0,1 \rbrace}$ around $x=0$ and $x=1$ to obtain
\beqn
\label{period1b}
  \oint_{C_{\lbrace 0,1 \rbrace}}  \dfrac{dx}{\sqrt{x\, (1-x) \, (1-t x)}}  = \pi \sum_{n \ge 0} \frac{\left(\frac{1}{2}\right)_n t^n}{n!}
  \oint_{C_{\lbrace 0,1 \rbrace}}  dx \; x^{n-\frac{1}{2}} (1-x)^{-\frac{1}{2}} = 4\pi^2 \hpg21{ \frac{1}{2}, \frac{1}{2}}{1}{t} \,,
\eeqn 
where we used Equation~(\ref{eqn:beta_function}).
\par A fundamental observation is the following: the quadratic twist has turned the zero-dimensional family~(\ref{legendre0}) into the family of elliptic curves~(\ref{legendre1}); the holomorphic period for the family of elliptic curves is the Hadamard product of the function ${}_1F_0$ -- which accounts for the quadratic twist -- and the holomorphic period~(\ref{period0}) of the zero-dimensional family. That is, for $|t|<1$ we obtain 
\beq
\label{eqn:intro_EC}
 \oint_{\Sigma_1(t)} \dfrac{dx}{y_1} = 2\pi \, {}_1F_0\!\left(\left. \frac{1}{2} \right| t\right) \star  {}_1F_0\!\left(\left. \frac{1}{2} \right| t\right) =  2\pi \,  \hpg21{ \frac{1}{2}, \frac{1}{2}}{1}{t} \,,
\eeq
where we used Equation~(\ref{eqn:cancellation}).
\par To obtain from the family of elliptic curves in Equation~(\ref{legendre1}) a family of K3 surfaces, one again promotes the parameter $t$ to an additional complex variable $u$ and carries out the quadratic twist $y_1^2 \mapsto y_2^2/[u(u-t)]$ in Equation~(\ref{legendre1}). This yields the twisted Legendre pencil given by
\beq
\label{legendre2}
 y_2^2 = x \, (x-1) \, (x-u) \, u \, (u-t)\,.
\eeq
Equation~(\ref{legendre2b}) defines the N\'eron model for a family of elliptically fibered K3 surfaces $X_t$ of Picard rank $19$  with section over $\mathbb{P}^1\ni [u:1]$.  Hoyt~\cite{MR732965} and Endo~\cite{MR2634350} extended arguments of Shimura and Eichler \cite{MR0089928} to show that on the parabolic cohomology group $H_\mathbb{Z} \cong \mathbb{Z}^3$ associated with $\pi: X_t \to \mathbb{P}^1$  there is a natural polarized Hodge filtration given by
\beqn
 \left \lbrace H_\mathbb{Z}, Q, F^2 \subset F^1 \subset H_\mathbb{Z} \otimes \mathbb{C} \right \rbrace \,.
\eeqn
Here, $H_\mathbb{Z} \otimes \mathbb{C}$ consists of cohomology classes spanned by suitable meromorphic differentials of the second kind, and $Q$ is a non degenerate $\mathbb{Q}$-valued bilinear form determined by period relations of the holomorphic two-form $du\wedge dx/y_2$ \cite{MR894512}. The period point lies on the cone $Q=0$. In the situation of Equation~(\ref{legendre2}), it follows $Q=2z_1^2+2z_2^2-2z_3^2$, the local system $R^2\pi_*\mathbb{C}_X$ of middle cohomology is irreducible, and the cohomology group $H_\mathbb{Z}$ carries a pure Hodge structure of weight two and type $(1,1,1)$. 
\par A basis of transcendental cycles is constructed from cycles in the elliptic fiber and carefully chosen curves in the base connecting the cusps $0, 1, t, \infty$~\cite{MR732965}.  As in Shimura~\cite{MR0120372}, for continuously varying families of closed one-cycles $\Sigma_1(u), \check{\Sigma}_1(u)$, that form bases of the first homology of the fiber, the expression
\beq
\label{secondary}
\int_{t}^u du \, \left( \begin{array}{c} \int_{\Sigma_1(u)} \frac{dx}{y_2} \\[0.5em]  \int_{\check{\Sigma}_1(u)} \frac{dx}{y_2} \end{array} \right)
\eeq
defines a vector-valued holomorphic function that converges as $u$ approaches the cusps at $u=0, 1, \infty$. It was shown by Cox and Zucker~\cite{MR538682} that the components of~(\ref{secondary}) are $\mathbb{Q}$-linear combinations of periods of associated meromorphic two-forms on $X$ that are of the second kind and holomorphic on singular fibers; they represent generalized cusp forms of weight three associated with Equation~(\ref{legendre2}). In particular, periods of the holomorphic two-form $du \wedge dx/y_2$ satisfy the third-order differential equation $L^{(3)}_t((\frac{1}{2}, \frac{1}{2}, \frac{1}{2}); (1, 1))\,\omega(t)=0$. Since the singular fiber over $u=0$ and $u=t$ is of Kodaira-type $I^*_2$ and $I_0^*$, respectively -- the latter having monodromy $-\mathbb{I}$ independent of the chosen homological invariant -- there is a unique A-cycle that is transformed into itself as a path encircles the cusps at $u=0$ or $u=t$. One obtains a transcendental two-cycle $\Sigma_2(t)$ on $X_t$ by tracing out this A-cycle $\Sigma_1(u)$ in the fiber over the line segment between the cusps at $u=0$ and $u=t$ in the base.  Then, one integrates the holomorphic two-form $du \wedge dx/y_2$ over the two-cycle $\Sigma_2(t)$  to obtain
\beq
\label{period2}
 \oiint_{\Sigma_2(t)} du \wedge \dfrac{dx}{y_2} =  2\pi \int_0^t \dfrac{du}{\sqrt{u \, (u-t)}} \; \hpg21{ \frac{1}{2}, \frac{1}{2}}{1}{u} 
 = 2i \pi^2  \hpg32{ \frac{1}{2}, \frac{1}{2}, \frac{1}{2}}{1, 1}{t} \,,
 \eeq 
where we used Equation~(\ref{eqn:cancellation}) and Lemma~\ref{lem:convolution}.
\par To obtain from the family of K3 surfaces~(\ref{legendre2}) a family of Calabi-Yau threefolds, one promotes the parameter $t$ to an additional complex variable $v$ and carries out yet another quadratic twist $y_2^2 \mapsto y_3^2/[v(v-t)]$ in Equation~(\ref{legendre2}). One obtains the family
\beq
\label{legendre3}
 y_3^2 = x \, (x-1) \, (x-u) \, u \, (u-v) \, v \, (v-t)\,.
 \eeq
This family constitutes a pencil of elliptically fibered Calabi-Yau threefolds, denoted by $\pi: X_t \to \mathbb{P}^1$.  Each member $X_t$ of the family is fibered by K3 surfaces of Picard rank $19$ over $\mathbb{P}^1$ with affine coordinate $u$. As a consequence, the local system $R^3\pi_* \mathbb{C}_X$ of middle cohomology is irreducible and  the transcendental cohomology group $H_\mathbb{Z}$ carries a pure Hodge structure of weight three and type $(1, 1, 1, 1)$.
\par The natural Hodge structure on the parabolic cohomology group of $X_t$ can be described in terms of periods of the holomorphic three-form $dv\wedge du \wedge dx/y_3$. A transcendental three-cycle $\Sigma_3(t)$ on each threefold $X_t$ is obtained as Lefschetz thimble by tracing out the two-cycle $\Sigma_2(v)$ in the K3 fiber over the line segment between the cusps $v=0$ and $v=t$. If one integrates the holomorphic three-form $dv \wedge du \wedge dx/y_2$ over the cycle $\Sigma_3(t)$, one obtains for the holomorphic period
\beq
\label{period3}
  \oiiint{\Sigma_3(t)} dv \wedge du \wedge \dfrac{dx}{y_3} 
= -2\pi^3   \hpg43{ \frac{1}{2}, \frac{1}{2}, \frac{1}{2}, \frac{1}{2}}{1, 1, 1}{t}\,,
\eeq
where we have used Equation~(\ref{eqn:cancellation}) and Lemma~\ref{lem:convolution}. The period is annihilated by the rank-four and degree-one Picard-Fuchs operator
\beq
\label{ExampleOperator1}
 L^{(4)}_t\left(\Big(\frac{1}{2}, \frac{1}{2}, \frac{1}{2}, \frac{1}{2}\Big); (1, 1, 1)\right) = \theta^4 - t \, \left(\theta+\frac{1}{2}\right)^4\,.
\eeq
The Picard-Fuchs operator~(\ref{ExampleOperator1}) is one of the 14 original Calabi-Yau operators mentioned in the introduction and was labelled ``$(3)$" in the AESZ database~\cite{Almkvist:aa}. We make the following:
 \begin{remark}
For the construction of the cycles $\Sigma_n$ we employed two different strategies, namely the use of either Pochhammer cycles or Lefschetz thimbles. For $n=1$, we used a family of A-cycles equivalent to a Pochhammer contour. For $n=2$ and $n=3$, we used a Lefschetz thimbles to form transcendental cycles with a non-trivial $(n-1)$-cycle in the elliptic or K3 fiber over a line segment between cusps. The reason can be traced back to the equivalent ways of defining Euler's beta function. The beta function is $\mathrm{B}(\alpha,\beta)=\int_0^1 t^{\alpha-1} (t-1)^{\beta-1} dt$ for $\operatorname{Re}(x), \operatorname{Re}(y) >0$. The beta function is then analytically continued for all values of $\alpha$ and $\beta$. This is achieved by converting the Euler integral into an integral over a Pochhammer contour $C_{\lbrace 0,1 \rbrace}$ around $t=0$ and $t=1$ to obtain
 \beq
 \label{eqn:beta_function}
  \big(1-e^{2\pi i \alpha}\big) \big(1-e^{2\pi i \beta}\big) \, \mathrm{B}(\alpha,\beta) = \oint_{C_{\lbrace 0,1 \rbrace}} t^{\alpha-1} (t-1)^{\beta-1} dt \,.
 \eeq
In the context of our iterative construction (see Section~\ref{prelim}) it turns out that contour integration is easier to describe in the general setting.
\end{remark}
\subsection{Closely related examples} 
\label{second}
To obtain from the family of elliptic curves in Equation~(\ref{legendre1}) a family of K3 surfaces of Picard rank 18 instead of Picard rank 19, one again promotes the family parameter $t$ to an additional complex variable $u$, but carries out the quadratic twist $y_1^2 \mapsto y_2^2/[u^2-t^2]$ in Equation~(\ref{legendre1}). One obtains a twisted Legendre pencil given by
\beq
\label{legendre2b}
 y_2^2 = x \, (x-1) \, (x-u) \,  (u^2-t^2)\,.
\eeq
The equation defines the Weierstrass model for a family of elliptically fibered K3 surfaces $X_t$ of Picard rank $18$ with section over $\mathbb{P}^1$,  denoted by $\pi: X_t \to \mathbb{P}^1$.  We define a transcendental two-cycle $\hat{\Sigma}_2(t)$ by by tracing out the A-cycle $\Sigma_1(u)$ in the fiber over the line segment between the cusps at $u=-t$ and $u=t$ (avoiding $u=0$ by using a small arc). Using Lemma~\ref{lem:convolution} we obtain for the period integral
\beq
\label{period2bb}
 \oiint_{\hat{\Sigma}_2(t)} du \wedge \dfrac{dx}{y_2} =  2\pi \int_{-t}^t \dfrac{du}{\sqrt{u^2-t^2}} \, \hpg21{ \frac{1}{2}, \frac{1}{2}}{1}{u}  =   2i\pi^2 {}_1F_0\!\left(\left. \frac{1}{2} \right| t^2\right) \star \hpg21{ \frac{1}{2}, \frac{1}{2}}{1}{t} .
 \eeq 
The quadratic twist has turned Equation~(\ref{legendre1}) into (\ref{legendre2b}) and, similarly, the holomorphic period for the family of K3 surfaces is the Hadamard product of the function ${}_1F_0(t^2)$ -- which accounts for the modified quadratic twist -- and the holomorphic period~(\ref{eqn:intro_EC}). The Hadamard product can be evaluated explicitly and yields
\beqn
 {}_1F_0\!\left(\left. \frac{1}{2} \right| t^2 \right) \star \hpg21{ \frac{1}{2}, \frac{1}{2}}{1}{t}   =   \hpg43{ \frac{1}{4}, \frac{1}{4}, \frac{3}{4},  \frac{3}{4}}{1, 1, \frac{1}{2}}{t^2}  \,.
 \eeqn
The period~(\ref{period2bb}) annihilated by the rank-four and degree-two Picard-Fuchs operator
\beqn
 {\theta}^{3} \left( \theta-1 \right) - \frac{t^2}{16}\, \left( 2\,\theta+3 \right) ^{2} \left( 1+2\,\theta \right) ^{2} \,.
\eeqn
\par To obtain from the family of K3 surfaces in Equation~(\ref{legendre2b}) a family of Calabi-Yau threefolds, one promotes the family parameter $t$ to an additional complex variable $v$ and carries out a quadratic twist $y_2^2 \mapsto y_3^2/[v^2-t^2]$. One obtains the family
\beq
\label{legendre3b}
 y_3^2 = x \, (x-1) \, (x-u) \, (u^2-v^2) \, (v^2-t^2)\,,
 \eeq
that constitutes a non-trivial pencil of elliptically fibered Calabi-Yau threefolds, denoted by $\pi: X_t \to \mathbb{P}^1$. As before, the natural Hodge structure on the parabolic cohomology group of $X_t$ can be described in terms of periods  and period relations of the holomorphic three-form $dv\wedge du \wedge dx/y_3$. We construct a transcendental three-cycle $\hat{\Sigma}_3(t)$ on $X_t$ as Lefschetz thimble by tracing out the cycle $\hat{\Sigma}_2(v)$ in the K3 fiber over the line segment between the cusps $v=-t$ and $v=t$ (avoiding $v=0$ by using a portion of a small circle). If one integrates the holomorphic three-form $dv \wedge du \wedge dx/y_2$ over the cycle $\hat{\Sigma}_3(t)$, one obtains for the holomorphic period
\beq
\label{period3b}
 \oiiint{\hat{\Sigma}_3(t)} dv \wedge du \wedge \dfrac{dx}{y_3} = -2\pi^3  \hpg43{ \frac{1}{4}, \frac{1}{4}, \frac{3}{4},  \frac{3}{4}}{1, 1, 1}{t^2}  \,,
\eeq
where we have applied the cancellation formula~(\ref{eqn:cancellation}) to conclude
\beqn
{}_1F_0\left(\left. \frac{1}{2} \right| t^2\right) \star   \hpg43{ \frac{1}{4}, \frac{1}{4}, \frac{3}{4},  \frac{3}{4}}{1, 1, \frac{1}{2}}{t^2} =   \hpg43{ \frac{1}{4}, \frac{1}{4}, \frac{3}{4},  \frac{3}{4}}{1, 1, 1}{t^2} \,.
\eeqn
 The period~(\ref{period3b}) is annihilated by the fourth-order and degree-two Picard-Fuchs operator (rescaled by $2^4$)
\beq
\label{ExampleOperator2}
  L^{(4)}_{t^2}\left(\Big(\frac{1}{4}, \frac{1}{4}, \frac{3}{4}, \frac{3}{4}\Big); (1, 1, 1)\right)
  =   \theta^{4} - \frac{t^2}{16}\, \left( 2\,\theta+1 \right)^{2} \left( 2\,\theta+ 3 \right) ^{2} \,.
 \eeq
 The Picard-Fuchs operator~(\ref{ExampleOperator2}) is another of the 14 original Calabi-Yau operators mentioned in the introduction and was labelled ``$(10)$" in the AESZ database~\cite{Almkvist:aa}.  
\section{The twist construction}\label{prelim}
In this section we describe the construction of twisted families with generalized functional invariant in full generality, including several modified variants needed later, and the computation of period integrals. 
\subsection{Elliptic fibrations}
 In this section we will recall facts  about the geometry of elliptically fibered varieties; we refer to the papers \cites{1307.7997,MR977771,MR1242006} for details. We also adopt the definition of terminal, canonical, and log-terminal singularity of a projective variety from aforementioned articles. 
\par We define an elliptic fibration $\pi: X \to S$ to be a proper surjective morphism with connected fibers between normal  complex varieties $X$ and $S$ whose general fibers are nonsingular elliptic curves. We assume that $\pi$ is smooth over an open subset $S^0$ whose complement is a divisor with only normal crossings. Then, the local system $H_0^i:=R^i\pi_*\mathbb{Z}_{X}\mid_{S^0}$ forms a variation of Hodge structure over $S^0$. There is a canonical bundle formula for the elliptic fibration:  with the fundamental line bundle denoted by $\mathcal{L}:=\left(R^1\pi_*\mathcal{O}_X\right)^{-1}$, the canonical bundles $\pmb{\omega}_X:=\wedge^{\mathrm{top}}T^*X$ and $\pmb{\omega}_S:=\wedge^{\mathrm{top}}T^*S$ are related by
\beq
\label{EquationDualizingSheaf}
 \pmb{\omega}_X \cong \pi^*\Big( \pmb{\omega}_S \otimes \mathcal{L}\Big) \otimes \mathcal{O}_X(D)\,,
\eeq
where $D$ is a certain effective divisor on $X$ that only depends on divisors of $S$ over which $\pi$  has multiple fibers and divisors of $X$ that give $(-1)$-curves in the fibers of $\pi$.  The existence of a section for the elliptic fibration $\pi: X \to S$ prevents the presence of multiple fibers. And the presence of $(-1)$-curves in the fibers is avoided by imposing a minimality criterion. In the case of an elliptic surface, we assume that the fibration is relatively minimal, i.e., that there are no $(-1)$-curves in the fibers of $\pi$. In the case of an elliptic threefold, we assume that no contraction of a surface in $X$ is compatible with the fibration. $X$ is a Calabi-Yau manifold if $h^i(X, \mathcal{O}_X) =0$ for $0 < i < \dim{X}$ and $\pmb{\omega}_X \cong \mathcal{O}_X$. It is known \cite{MR1109635} that for any elliptic fibration on a Calabi-Yau threefold, the base surface can have at worst log-terminal orbifold singularities. In this article we will take the base surface $S$ always to be a blow-up of $\mathbb{P}^2$ or a Hirzebruch surface $\mathbb{F}_k$.
\par It is a well-known that for any elliptic fibration $\pi: X \to S$ with section $\sigma:S \to X$, there always exists a Weierstrass model $W$ over $S$, i.e., there is a complex variety $W$ and  a proper flat surjective morphism  $p: W \to S$ with canonical section whose fibers are irreducible cubic curves in $\mathbb{P}^2$ together with a birational map from $X$ to $W$  that maps $\sigma$ to the canonical section of the Weierstrass model.  The map from $X$ to $W$ blows down all components of fibers which do not intersect $\sigma(S)$. It is also known that for a relatively minimal elliptic fibration with section, the morphism on the Weierstrass model is in fact a resolution of the singularities of $W$. 
\subsection{Weierstrass models}
Let $\mathcal{L}$ be a line bundle on $S$, and $g_2$ and $g_3$ sections of $\mathcal{L}^{4}$ and $\mathcal{L}^{6}$, respectively, such that the discriminant $\Delta=g_2^3-27 \, g_3^2$ is a section of $\mathcal{L}^{12}$ not identically zero. Define $\mathbf{P} :=\mathbb{P}( \mathcal{O}_S \oplus \mathcal{L}^2 \oplus \mathcal{L}^3)$ and let $p:\mathbf{P} \to S$ be the natural projection and $\mathcal{O}_{\mathbf{P}}(1)$ be the tautological line bundle. We denote by $x$, $y$, and $z$ the sections of $\mathcal{O}_{\mathbf{P}}(1) \otimes \mathcal{L}^{2}$,  $\mathcal{O}_{\mathbf{P}}(1) \otimes \mathcal{L}^{3}$, and $\mathcal{O}_{\mathbf{P}}(1)$, respectively, which correspond to the natural injections of $\mathcal{L}^{2}$,  $\mathcal{L}^{3}$, and $\mathcal{O}_S$ into $\pi_* \mathcal{O}_{\mathbf{P}}(1) = \mathcal{O}_S \oplus \mathcal{L}^{2}  \oplus\mathcal{L}^{3}$. We denote by $W$ the projective variety in $\mathbf{P}$ defined by the equation
\beq
\label{Weierstrass}
 y^2  z = 4  x^3 - g_2  x z^2 - g_3   z^3 \,.
\eeq
Its canonical section $\sigma: S \to W$ is given by the point $[x:y:z]=[0:1:0]$ such that $\Sigma:=\sigma(S) \subset W$ is a Cartier divisor on $W$, and its normal bundle is isomorphic to the fundamental line bundle by $p_* \mathcal{O}_{\mathbf{P}}\big(\!-\!\Sigma\big)\cong \mathcal{L}$. It then follows that $W$ is normal if $S$ is normal; and $W$ is Gorenstein if $S$ is, and the formula (\ref{EquationDualizingSheaf}) for the dualizing sheaf reduces to
\beq
\label{canonical}
 \pmb{\omega}_W =\pi^*\Big( \pmb{\omega}_S \otimes \mathcal{L}\Big) \,.
\eeq 
Thus, the total space is a Calabi-Yau variety if and only if the line bundle $\mathcal{L}$ is the anti-canonical bundle of the base, i.e., $\mathcal{L} = \pmb{\omega}_S^{-1}=\mathcal{O}_S(-K_S)$.  For the Weierstrass model $p: W \to S$ of an elliptic fibration $\pi: X \to S$ with section we can compare  the discriminant locus $\Delta(\pi)$, i.e., the points over which $X_p$ is singular, with the vanishing locus of $\Delta(p)=g_2^3-27 \, g_3^2$.  We then have $\Delta(p) \subset \Delta(\pi)$. In fact, the morphism from $X$ to $W$ is a resolution of singularities if and only if $\Delta(p) = \Delta(\pi)$ in which  case it is also a small resolution.
\par We call a Weierstrass model minimal if there is no prime divisor $D$ on $S$ such that $\operatorname{div}(g_2) \ge 4 \, D$ and $\operatorname{div}(g_3) \ge 6 \, D$. We call two minimal Weierstrass models that are smooth over an open subset $S_0\subset S$ equivalent if there is an isomorphism of the Weierstrass models over $S_0$ which preserves the  canonical sections. Every Weierstrass fibration is birationally isomorphic to a minimal Weierstrass fibration.  A criterion for $W$ having only rational singularities can  then be stated as follows: If the reduced discriminant divisor $(\Delta)_{\mathrm{red}}$ has only normal crossings, then $W$ has only rational Gorenstein singularities if and only if the Weierstrass model is minimal.  
\subsection{Families of Weierstrass models}
We now turn to projective families of Calabi-Yau $n$-folds over $B=\mathbb{P}^1 \backslash \lbrace 0, 1, \infty\rbrace$ which we denote by $\pi: X \to B$. Each element $X_t=\pi^{-1}(t)$ of the family is a compact complex $n$-fold with trivial canonical bundle $\pmb{\omega}_{X_t}\cong \mathcal{O}_{X_t}$ and is assumed to be elliptically fibered with section over a fixed normal complex variety $S$. Such a family is described as a family of minimal normal Weierstrass models. That is, each complex $n$-fold $\pi_t: X_t \to S$ is given as a minimal Weierstrass model $p_t: W_t\to S$. For each $t \in \mathbb{P}^1 \backslash \lbrace 0, 1, \infty\rbrace$ and the affine coordinate chart $u=(u_1, \dots, u_{n-1}) \in \mathbb{C}^{n-1} \subset S$, the Weierstrass model $W_t$ has the form
\beq
\label{eq_Weq}
 y^2 z = 4x^3 - g_2\left( t, u  \right)  \, x z^2 - g_3\left(t, u  \right) \, z^3\,.
\eeq
There is a holomorphic sub-bundle $\mathcal{H}\to B$ of the vector bundle $V=R^n\pi_*\mathbb{C}_{X} \to B$ whose fibers are $H^0(\pmb{\omega}_{X_t})\subset H^n(X_t,\mathbb{C})$. The vector bundle $\mathcal{V}=V \otimes \mathcal{O}_B$ carries a canonical flat connection $\nabla$, called the the \emph{Gauss-Manin connection} \cites{MR0229641,MR0233825,MR0282990,MR0258824}. Representing the family $X_t$ as the family of Weierstrass models $W_t$ determines an explicit meromorphic section $\eta \in \Gamma(B,\mathcal{V})$ such that such that $\eta_t \in H^n(X_t,\mathbb{C})$ can be represented by the closed differential form
\beq
\label{eqn:flat_section}
  \eta_t = du_1 \wedge \dots \wedge du_{n-1} \wedge \frac{dx}{y} \,,
\eeq 
in the affine chart $z=1$ on $W_t$. A local parallel section $\Sigma$ of the dual bundle $\mathcal{H}^*$ for some open $U \subset B$ is represented  by a closed transcendental $n$-cycle $\Sigma(t)$ on each fiber $X_t$ with $t \in U$. The \emph{period sheaf} $\Pi(\mathcal{H},\eta) \to B$ is the sheaf whose stalks are generated by local analytic functions obtained by paring the global section $\eta$ with local parallel sections of $\mathcal{H}^*$, i.e., of the form 
\beq
t \in U \mapsto \omega(t) = \langle \Sigma, \eta \rangle = \oint_{\; \Sigma(t)} \eta_t \,,
\eeq
for open $U \subset B$ and the fiberwise Poincar\'e pairing $\langle. , .\rangle$. We call the local non-vanishing analytic function $\omega(t)$ a \emph{period integral} or \emph{period over $\Sigma(t)$}.
\par The Picard-Fuchs differential equation is then obtained as follows: we fix a meromorphic vector field $\frac{d}{dt}$ on the curve $B$ for the choice of affine coordinate $t$ on $\mathbb{P}^1$.  The vector field $\frac{d}{dt}$ induces a covariant derivative operator $\nabla_{d/dt}$ on $\mathcal{V}$. Since $\mathcal{V}$ has rank $n$, the meromorphic section $\eta$ and its derivatives $\nabla_{d/dt}^i \eta$ for $1 \le 1 \le n$ must be linearly dependent over the field of meromorphic functions on $B$, and there is a relation
\beqn
 \sum_{i=0}^m a_i(t) \, \nabla_{d/dt}^i \eta = 0 \,,
 \eeqn
where $m \le n$ and we always normalize to have $a_m=1$. Since $\nabla$ vanishes on the parallel section $\Sigma$ of $\mathcal{H}^*$, it follows that the period $\omega(t)$ satisfies the differential equation
\beqn
 \frac{d^m}{dt^m} \omega(t) + \sum_{i=0}^{m-1} a_i(t) \, \frac{d^i}{dt^i} \omega(t)  = 0 \,.
\eeqn 
\subsection{Twisted family of Weierstrass models}
\label{ssec:twisting}
We define a \emph{generalized functional invariant} to be a triple $(i, j, \alpha)$ with $i, j \in \mathbb{N}$ such that $1 \le i, j \le 6$ and $\alpha \in \{ \frac{1}{2}, 1\}$. The general notion of generalized functional invariant was first introduced in \cite{MR1877754}. A generalized functional invariant defines a family of ramified covering maps $\mathbb{P}^1 \to \mathbb{P}^1$ of degree $i+j$ given by 
\beq
\label{eqn:gfi}
 [v_0:v_1] \mapsto \Big[ c_{ij} v_1^{i+j} \tilde{t} \ : v_0^i (v_0+v_1)^j \Big] \;, 
 \eeq
which is totally ramified over $0$, ramified to degrees $i$ and $j$ over $\infty$, and has a simple ramification point over $\tilde{t}$. 
\par For a family $\pi: X \to B$ with Weierstrass model~(\ref{eq_Weq}) such that
\begin{small}  \begin{gather}
\label{cond:normal}
0 \le  \deg_{t}{(g_2)} \le \min\Big(\frac{4}{i}, \frac{4\alpha}{j}\Big), \quad
0 <  \deg_{t}{(g_3)} \le \min\Big(\frac{6}{i}, \frac{6\alpha}{j}\Big)\, ,
\end{gather}  \end{small}%
we construct a new family $\tilde{\pi}: \tilde{X} \to B$ such that each element $\tilde{X}_{\tilde{t}}=\tilde{\pi}^{-1}(\tilde{t})$  is a compact complex $(n+1)$-fold and also elliptically fibered with section over $\mathbb{P}^1 \times S$. We call this new family the \emph{twisted family with generalized functional invariant $(i, j, \alpha)$ of $X \to B$}. For $\tilde{t} \in B$, there is a local coordinate chart $\{[v_0:v_1],(u_1, \dots, u_{n-1}) \in \mathbb{P}^1 \times S\}$ such that the Weierstrass model $\tilde{W}_{\tilde{t}}$ for $\tilde{X}_{\tilde{t}}$ is given by
\begin{small}  \begin{gather}
\label{eqn:twist_model}
 \tilde{y}^2 \tilde{z}= 4 \, \tilde{x}^3 - g_2\left( \frac{c_{ij} \tilde{t} v_1^{i+j}}{v_0^i(v_0+v_1)^j}, \, u  \right) \, v_0^4 v_1^{4-4\alpha} (v_0+v_1)^{4\alpha} \tilde{x} \tilde{z}^2\\
 \nonumber
 - g_3\left(\frac{c_{ij} \tilde{t} v_1^{i+j}}{v_0^i(v_0+v_1)^j}, \, u  \right)  \, v_0^6 v_1^{6-6\alpha} (v_0+v_1)^{6\alpha} \, \tilde{z}^3 
\end{gather}  \end{small}%
with $c_{i j}=(-1)^i \, i^i \, j^j/(i+j)^{i+j}$.  The Weierstrass model~(\ref{eqn:twist_model}) is smooth over the open subset $D^0 \times S^0$  where $D^0$ is the complement in $\mathbb{P}^1$ of the curves $v_0=0$, $v_1=0$, $v_0+v_1=0$, and $c_{i j} \tilde{t} v_1^{i+j} -v_0^i(v_0+v_1)^j =0$. Conditions~(\ref{cond:normal}) ensure that the Weierstrass model is minimal and normal if the original Weierstrass model~(\ref{eq_Weq}) is. Over the new base $\tilde{S}=\mathbb{P}^1 \times S$, the fundamental line bundle and canonical bundle are given by
\beq
 \tilde{\mathcal{L}} = \mathcal{M} \boxtimes \mathcal{L} \,, \quad \pmb{\omega}_{\tilde{S}} = \pmb{\omega}_{\mathbb{P}^1}  \boxtimes \pmb{\omega}_{S} \,.
\eeq
Here the external tensor product $\mathcal{M} \boxtimes \mathcal{L}= \operatorname{pr}_1^*\mathcal{M} \otimes \operatorname{pr}_2^*\mathcal{L}$ denotes the tensor product on $\tilde{S}=\mathbb{P}^1 \times S$ of the two pullback bundles to $\tilde{S}$, along the canonical projection maps $\operatorname{pr}_1: \tilde{S} \to \mathbb{P}^1$ and $\operatorname{pr}_2: \tilde{S} \to S$. Two $\mathbb{C}^*$-group actions, acting on Equation~(\ref{eqn:twist_model}), are given by assigning weights to the defining variables as listed in Table~\ref{tab:deg3} where \emph{deg} denotes the total weight of Equation~(\ref{eqn:twist_model}) and \emph{sum} denotes the sum of weights of the defining variables. Since the total weight equals the sum of weights of the variables, we have $\mathcal{M} \otimes \pmb{\omega}_{\mathbb{P}^1}\cong  \mathcal{O}_{\mathbb{P}^1}$. Therefore, the Calabi-Yau condition $\tilde{\mathcal{L}} = \pmb{\omega}_{\tilde{S}}^{-1}$ is satisfied for the new family $\tilde{\pi}: \tilde{X} \to B$ such that $\pmb{\omega}_{\tilde{X}_{\tilde{t}}} \cong \mathcal{O}_{\tilde{X}_{\tilde{t}}}$, if it is satisfied for the original family~(\ref{eq_Weq}).
\begin{table}[H]
\scalebox{\ScaleBig}{
\begin{tabular}{|c|c|ccc|cc|c|}
\hline
$\mathbb{C}^*$ & deg & $x$ & $y$ & $z$ & $v_0$ & $v_1$ & $\Sigma$ \\
\hline
\hline
$\lambda_1$ & $3$ & $1$ & $1$ & $1$ & $0$ & $0$ & $3$\\
$\lambda_2$ & $12$ & $4$ & $6$ & $0$ & $1$ & $1$ & $12$\\
\hline
\end{tabular}}
\caption{Weights of variables in Weierstrass equation}\label{tab:deg3}
\end{table}
\subsubsection{Period integrals of twisted family}
The differential form $\eta_t$ in Equation~(\ref{eqn:flat_section}) defines a flat section $\eta \in \Gamma(B,\mathcal{H})$ for the family in Equation~(\ref{eq_Weq}). A flat section $\tilde{\eta} \in \Gamma(B,\tilde{\mathcal{H}})$ for the twisted family~(\ref{eqn:twist_model}) is then given by the differential form 
\beq
  \tilde{\eta}_{\tilde{t}} = dv \wedge du_1 \wedge \dots \wedge du_{n-1} \wedge \frac{d\tilde{x}}{\tilde{y}} \,,
\eeq 
in the affine chart $[v_0:v_1]=[v:1]$ and $\tilde{z}=1$, such that
\beq
  \tilde{\eta}_{\tilde{t}} = \frac{dv}{v(v+1)^\alpha} \wedge \eta_t \,,
\eeq
where we have used $\tilde{x}= v_0^2 v_1^{2-2\alpha} (v_0+v_1)^{2\alpha} x$ and $\tilde{y}= v_0^3 v_1^{3-3\alpha} (v_0+v_1)^{3\alpha} y$. Thus, if $\omega(t)$ is a local section of the period sheaf $\Pi(\mathcal{H},\eta) \to B$, a section of the new period sheaf $\Pi(\tilde{\mathcal{H}},\tilde{\eta}) \to \tilde{B}$ is given by
\beq
\label{eqn:period_new}
  \tilde{\omega}(\tilde{t}) = \oint_{C} \frac{dv}{v(v+1)^\alpha}  \; \omega\left( \frac{c_{ij} \tilde{t}}{v^i(v+1)^j}\right) \,,
\eeq
where $C$ is a non-contractible loop in the punctured $v$-plane.  We have the following:
\begin{proposition}
\label{prop:period_integral_general}
Let $\omega(t)$ be a period integral for the family~(\ref{eq_Weq}) with absolutely convergent series $\omega(t) = \sum_{k\ge 0} f_k t^k$ for $|t|<1$, and the twisted family~(\ref{eqn:twist_model}) with generalized functional invariant $(i, j, \alpha)$ satisfy conditions~(\ref{cond:normal}). Let $C_{1/2}(0)$ be the circle  $|v|=\frac{1}{2}$ in the $v$-plane with counterclockwise orientation. For every $\tilde{t}\in \mathbb{C}$ with $|\tilde{t}|<1/(2^{i+j+1}|c_{i j}|)$ and $c_{i j}=(-1)^i \, i^i \, j^j/(i+j)^{i+j}$, the period integral~(\ref{eqn:period_new}) for $C=C_{1/2}(0)$ has an absolutely convergent series given by the Hadamard products:
\beq
\label{eq:period_new}
\begin{split}
  \text{if $\alpha=1$:}& \qquad \tilde{\omega}(\tilde{t})  =  \,(2\pi i) \;\,  \hpg{i+j-1}{i+j-2}{  \frac{1}{i+j} \quad  \dots \quad \frac{i+j-1}{i+j} }{ \frac{1}{i} \; \dots\;  \frac{i-1}{i} \; \frac{1}{j} \; \dots\;  \frac{j-1}{j} }{\tilde{t}}    \star  \omega(\tilde{t}) \;,\\[0.2em]
  \text{if $\alpha=\frac{1}{2}$:}& \qquad \tilde{\omega}(\tilde{t})  =  \,(2\pi i) \; \,  \hpg{i+j}{i+j-1}{  \frac{\alpha}{i+j} \quad  \dots \quad \frac{\alpha+i+j-1}{i+j} }{ \frac{1}{i} \; \dots\;  \frac{i-1}{i} \; \frac{\alpha}{j} \; \dots\;  \frac{\alpha+j-1}{j} }{\tilde{t}}    \star  \omega(\tilde{t}) \,.
\end{split}
\eeq
\end{proposition}
\begin{proof}
For $|v|=\frac{1}{2}$ and $|\tilde{t}|<1/(2^{i+j+1}|c_{i j}|)$ we have $|t|=|\frac{c_{i j} \, \tilde{t}}{v^i (v+1)^j} |<1$. We use the absolutely and uniformly convergent series $\omega(t) = \sum_{k\ge 0} f_k t^k$ and carry out a term-by-term integration to obtain
\beqn
\label{first_eval}
\begin{split}
 \oint_{|v|=\frac{1}{2}} \frac{dv}{v (v+1)^\alpha} \; \omega\left(  \frac{c_{ij} \, \tilde{t}}{v^i  (v+1)^j}\right) 
=&  \sum_{k \ge 0} f_k \, \left(c_{ij} \tilde{t}\right)^k  \oint_{|v|=\frac{1}{2}} \frac{dv}{v^{ik+1} (v+1)^{jk+\alpha}}\,.
\end{split}
\eeqn
Using the formula $(1+z)^{-k}=\sum_{l \ge 0} \frac{\Gamma(l+k)}{\Gamma(k)\Gamma(l+1)} (-z)^l$, we obtain from a residue computation the identity
\beq
\begin{split}
& \sum_{k \ge 0}  f_k \, \left(c_{ij} \tilde{t}\right)^k \sum_{l\ge0} \frac{\Gamma(jk+l+\alpha)}{\Gamma(jk+\alpha) \, \Gamma(l+1)} (-1)^l  \oint_{|v|=\frac{1}{2}} \frac{dv}{v^{ik-l+1}}\\
= &  \, (2\pi i)\,   \sum _{k \ge 0} f_k \, \dfrac{\prod_{m=0}^{i+j-1} \left(\frac{\alpha+m}{i+j}\right)_k}{k!  \prod_{m=1}^{i-1} \big(\frac{m}{i}\big)_k \, \prod_{m=0}^{j-1} \big(\frac{m+\alpha}{j}\big)_k} \; \tilde{t}^k \,,
\end{split}
\eeq
where we used Gauss' multiplication formula for the Gamma function
\beq
\label{eqn:GaussMultiplication}
 \Gamma(n z)=(2\pi)^{(1-n)/2} n^{n z-\frac{1}{2}} \prod_{r=0}^{n-1} \Gamma\left( z + \frac{r}{n}\right) \,.
\eeq
Since we also have $|\tilde{t}|<1$ Equation~(\ref{first_eval}) can be further simplified using Gamma-function identities and the Hadamard product. We obtain 
\beq
\begin{split}
 & \, (2\pi i)\,   \sum _{k \ge 0} f_k \, \dfrac{\prod_{m=0}^{i+j-1} \left(\frac{\alpha+m}{i+j}\right)_k}{k!  \prod_{m=1}^{i-1} \big(\frac{m}{i}\big)_k \, \prod_{m=0}^{j-1} \big(\frac{m+\alpha}{j}\big)_k} \; \tilde{t}^k\\
 = & \,(2\pi i) \;   \hpg{i+j}{i+j-1}{ \frac{\alpha}{i+j}, \quad  \dots \quad \frac{\alpha+i+j-1}{i+j} }{\frac{1}{i}, \dots, \frac{i-1}{i}, \frac{\alpha}{j}, \dots, \frac{\alpha+j-1}{j} }{\tilde{t}} 
\star \omega(\tilde{t}) \,.
\end{split}
\eeq
The remaining equation is obtained by setting $\alpha=1$ and observing an obvious cancellation in the coefficients of the hypergeometric series.
\end{proof}
\subsection{Variants of the twist construction}
\subsubsection{Twists with two parameters and subfamilies}
\label{sssec:2parameters}
The twist construction with generalized functional invariant $(i, j, \alpha)=(1,1,1)$ can be generalized in a way that introduces two parameters. It is based on the two-parameter family of ramified covering maps $\mathbb{P}^1 \to \mathbb{P}^1$ of degree two given by 
\beq
\label{eqn:2parameter_map}
 [v_0:v_1] \mapsto \Big[ 4 a v_0 (v_0+v_1) + (a-b)  v_1^{2} \ : \ 4 v_0 (v_0+v_1) \Big]\,,
\eeq
which is totally ramified over $a$ and $b$. For $a=0$ and $b=\tilde{t}$, this family coincides with the covering maps used in the twist construction with generalized functional invariant $(i, j, \alpha)=(1,1,1)$ in Equation~(\ref{eqn:gfi}).
\par We restrict the two-parameter family to the locus $a=-b=\tilde{t}$ instead, and obtain a new family $\tilde{\pi}: \tilde{X} \to B$ such that each element $\tilde{X}_{\tilde{t}}=\tilde{\pi}^{-1}(\tilde{t})$  is a compact complex $(n+1)$-fold with trivial canonical bundle and also elliptically fibered with section over $\mathbb{P}^1 \times S$. The Weierstrass model $\tilde{W}_{\tilde{t}}$ for $\tilde{X}_{\tilde{t}}$ is given by
\begin{small}  \begin{gather}
\label{eqn:twist_model_LowerRank}
 \tilde{y}^2 \tilde{z}= 4 \, \tilde{x}^3 - g_2\left( a + \frac{(a-b) v_1^{2}}{4v_0(v_0+v_1)}, \, u  \right) v_0^4  (v_0+v_1)^{4} \tilde{x} \tilde{z}^2\\
 \nonumber
 - g_3\left( a + \frac{(a-b) v_1^{2}}{4v_0(v_0+v_1)}, \, u   \right)  v_0^6 (v_0+v_1)^{6} \, \tilde{z}^3 \, ,
\end{gather}  \end{small}%
where we have set $a=-b=\tilde{t}$.
 Thus, if $\omega(t)$ is a local section of the period sheaf $\Pi(\mathcal{H},\eta) \to B$, a section of the new period sheaf $\Pi(\tilde{\mathcal{H}},\tilde{\eta}) \to \tilde{B}$ is given by
\beq
\label{eqn:period_new_LowerRank}
  \tilde{\omega}(\tilde{t}) = \oint_{C} \frac{dv}{v(v+1)}  \; \omega\left( \tilde{t}\big(1 + \frac{1}{2v(v+1)}\big)\right) \,,
\eeq
where $C$ is a non-contractible loop in the punctured $v$-plane.  We have the following:
\begin{proposition}
\label{prop:period_integral_general_2}
Let $\omega(t)$ be a period integral for the family~(\ref{eq_Weq}) with absolutely convergent series $\omega(t) = \sum_{k\ge 0} f_k t^k$ for $|t|<1$, and the twisted family~(\ref{eqn:twist_model_LowerRank}) with generalized functional invariant $(i, j, \alpha)=(1,1,1)$ and $a=-b=\tilde{t}$ satisfy conditions~(\ref{cond:normal}). Let $C_{1/2}(0)$ be the circle in the $v$-plane $|v|=\frac{1}{2}$ oriented counterclockwise. For every $\tilde{t}\in \mathbb{C}$ with $|\tilde{t}|<1/2$, the period integral~(\ref{eqn:period_new_LowerRank}) for $C=C_{1/2}(0)$ has an absolutely convergent series given by the Hadamard product
\beq
\label{eqn:pure_twist_LowerRank_b}
\begin{split}
  \tilde{\omega}(\tilde{t})  =  \,(2\pi i) \; \sum_{n \ge 0} \frac{f_{2n} \, \left(\frac{1}{2}\right)_n}{n!} \, t^{2n} =  \,(2\pi i) \, {}_1F_0\!\left(\left. \frac{1}{2} \right| t^2 \right) \star \omega(t)\,.
\end{split}
\eeq
\end{proposition}
\begin{proof}
The proof is analogous to the proof of Proposition~\ref{prop:period_integral_general}, where we use the additional identity
\beq
 \sum_{l=0}^n \binom{n}{l} \binom{2l}{l} \left( - \frac{1}{2}\right)^l = \left\lbrace \begin{array}{ll} 0 & \text{if $n$ is odd,}\\ \frac{\left(\frac{1}{2}\right)_k}{k!} & \text{if $n=2k$ is even.}\end{array}\right.
\eeq
\end{proof}
\subsubsection{Relation to quadratic twists}
\label{sssec:connection_with_qtwist}
Let  $\pi: X \to B$  be a family with Weierstrass model~(\ref{eq_Weq}) such that
\begin{small}  \begin{gather}
\label{cond:normal_qt}
0 \le  \deg_{t}{(g_2)} \le 4, \quad 0 <  \deg_{t}{(g_3)} \le 6\, .
\end{gather}  \end{small}%
The \emph{quadratic-twist family} $\hat{\pi}: \hat{X} \to B$ is constructed by promoting the family parameter $t$ to an additional complex variable and carrying out a quadratic twist. Each element $\hat{X}_{\tilde{t}}=\hat{\pi}^{-1}(\tilde{t})$  is a compact complex $(n+1)$-fold and also elliptically fibered with section over $\mathbb{P}^1 \times S$. For $\tilde{t} \in B$ and local coordinate chart $\{[w_0:w_1],(u_1, \dots, u_{n-1}) \in \mathbb{P}^1 \times S\}$, the Weierstrass model $\hat{W}_{\tilde{t}}$ for $\hat{X}_{\tilde{t}}$ is given by
\begin{small}  \begin{gather}
\label{eqn:twist_model_pure}
 \hat{y}^2 \hat{z}= 4 \, \hat{x}^3 - g_2\left( \tilde{t} \, \frac{w_0}{w_1}, \, u  \right) w_1^{4}  w_0^2 (w_0- w_1)^{2} \,\hat{x} \hat{z}^2\\
 \nonumber
 - g_3\left( \tilde{t} \, \frac{w_0}{w_1}, \, u  \right)  w_1^{6} w_0^3  (w_0- w_1)^{3} \, \hat{z}^3 \, .
\end{gather}  \end{small}%
The Weierstrass model~(\ref{eqn:twist_model}) is smooth over an open subset of the form $D^0 \times S^0$  where $D^0$ is the complement in $\mathbb{P}^1$ of $w_0=0$, $w_1=0$, and $w_0-\tilde{t}w_1=0$. Conditions~(\ref{cond:normal_qt}) ensure that the Weierstrass model is  minimal and normal if the Weierstrass model~(\ref{eq_Weq}) is. Moreover, a similar argument as in Section~\ref{ssec:twisting} shows that the Calabi-Yau condition is satisfied such that $\pmb{\omega}_{\hat{X}_{\tilde{t}}} \cong \mathcal{O}_{\hat{X}_{\tilde{t}}}$. We have the following:
\begin{lemma}
The quadratic-twist family $\hat{\pi}: \hat{X} \to B$ with Weierstrass model~(\ref{eqn:twist_model_pure}) is birationally equivalent to the twisted family $\tilde{\pi}: \tilde{X} \to B$ with generalized functional invariant $(i, j, \alpha)=(1,1,1)$ and Weierstrass model~(\ref{eqn:twist_model}). In particular, the periods of the two families coincide.
\end{lemma}
\begin{proof}
Setting
\beqn
 w = - \frac{1}{4 v(v+1)}\,, \quad \hat{x} = \frac{(1+2v)^2 \tilde{x}}{2^4v^4(v+1)^4}\,, \quad \hat{y} = \frac{(1+2v)^3 \tilde{x}}{2^6v^6(v+1)^6}\,,
\eeqn
in the affine charts $w_1=\hat{z}=v_1=\tilde{z}=1$, $w_0=w$ and $v_0=v$, transforms Equation~(\ref{eqn:twist_model_pure}) into Equation ~(\ref{eqn:twist_model}) such that $\hat{\eta}_{\tilde{t}}= \tilde{\eta}_{\tilde{t}}$.
\end{proof}
\begin{remark}
The sequence of examples from Section~\ref{FirstExample} is precisely based on iteratively applying Equation~(\ref{eqn:twist_model_pure}) in the affine chart $\hat{z}=1$ and
$[w_0:w_1]=[u:t]$ (or $[w_0:w_1]=[v:t]$, etc.)
\end{remark}
There is also a two-parameter family with Weierstrass models $\hat{W}_{a, b}$ given by
 \begin{small}  \begin{gather}
\label{eqn:twist_model_pure_LowerRank_b}
 \hat{y}^2 \hat{z}= 4 \hat{x}^3 - g_2\left(  \frac{w_0}{w_1}, \, u  \right) w_1^{4}  (w_0-aw_1)^2(w_0- b w_1)^{2}\, \hat{x} \hat{z}^2\\
 \nonumber
 - g_3\left( \frac{w_0}{w_1}, \, u  \right)  w_1^{6}  (w_0-aw_1)^3 (w_0- b w_1)^{3}  \, \hat{z}^3 \, .
\end{gather}  \end{small}%
The quadratic-twist family with Weierstrass model~(\ref{eqn:twist_model_pure_LowerRank_b}) is isomorphic to the family with Weierstrass model~(\ref{eqn:twist_model_LowerRank}). This is seen by setting
\beq
\label{RationalMapK3}
 w = a+ \frac{a-b}{4 \, v \, (v+1)}\,, \; 
 \hat{x} = \dfrac{(b-a)^2(1+2v)^2x}{2^4 v^4  (v+1)^4}\,, \; 
 \hat{y} =  \dfrac{(b-a)^3(1+2v)^3y}{2^6 v^6  (v+1)^6} \,.
\eeq
in the affine charts $w_1=\hat{z}=v_1=\tilde{z}=1$, $w_0=w$ and $v_0=v$.
\par Let us explain the geometric relationship between the quadratic-twist family and twisted family with generalized functional invariant $(i, j, \alpha)=(1,1,1)$ for a K3 surface. Let $X \to \mathbb{P}^1$ be the Jacobian rational elliptic surface given by 
\beqn
 y^2 z = 4x^3-g_2(t)\, x z^2-g_3(t) \, z^3 \,.
\eeqn
If the map $f: \mathbb{P}^1 \to \mathbb{P}^1$ is given by $f(v) = a+(a-b)/(4 v  (v+1))$, then the pull-back $t=f(v)$ of the rational elliptic surface is a two-parameter family of Jacobian elliptic K3 surface $\tilde{X} \to \mathbb{P}^1$. On $\tilde{X}$, we have the deck transformation $\imath$ given by $v \mapsto -v-1$ and the elliptic involution $-\mathrm{id}$. The composition $\jmath=-\mathrm{id} \circ \imath$ is a Nikulin involution leaving the holomorphic two-form invariant, and the minimal resolution of the quotient $\tilde{X}/\jmath$ is the Jacobian elliptic K3 surface $\hat{X} \to \mathbb{P}^1$ given by
\beqn
 \hat{y}^2 \hat{z}= 4\hat{x}^3-g_2(t) \, (t-a)^2 (t-b)^2 \hat{x} \hat{z}^2 -g_3(t)  \, (t-a)^3 (t-b)^3 \hat{z}^3\,.
\eeqn
That is, the K3 surface $\hat{X}$ is the quadratic-twist family of $X$ having fibers of Kodaira-type $I_0^*$ over
the two ramification points of $f$. The situation is summarized in Figure~\ref{fig_geom}.
\begin{figure}[ht!]
\begin{tikzpicture}[thick,scale=1, every node/.style={scale=\ScaleBig}]
  \matrix (m) [matrix of math nodes, row sep=3em,column sep=3em]{
  			&  \tilde{X}				&	\\
 X			& \mathbb{P}^1 	&  \hat{X}\\
 \mathbb{P}^1	& 				& \mathbb{P}^1	\\ };
  \path[-stealth]
  (m-1-2) edge (m-2-1)
  (m-1-2) edge (m-2-3)
  (m-1-2) edge (m-2-2)
  (m-2-1) edge (m-3-1)
  (m-2-3) edge (m-3-3)
  (m-2-2) edge node [above]	{$f$} (m-3-1)
  (m-2-2) edge node [above]	{\phantom{I}$f$} (m-3-3)  
  (m-3-1) edge [-,double distance = 1.5pt](m-3-3);
\end{tikzpicture}
\caption{Relation between quadratic-twist family and twisted family}\label{fig_geom}
\end{figure}
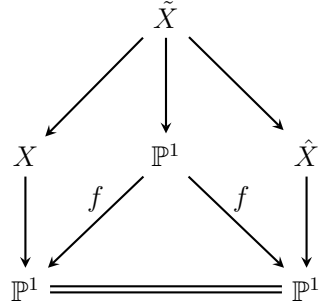
For each pair of K3 surfaces obtained as quadratic-twist family and twisted family with generalized functional invariant, we can easily check that the determinants of 
the discriminant groups of their transcendental lattices always differ by a square of the form $(1/2)^{2\alpha}$ with $0 \le \alpha \le 2$. 
As the transcendental lattices are related by the isometry given in Equation~(\ref{RationalMapK3}) and is generally not a Hodge isometry,
this is in perfect agreement with a general result obtained by Mehran~\cite{MR2306633}.
\begin{remark}
The Picard-Fuchs systems for the full two-parameter family in Equations~(\ref{eqn:twist_model_LowerRank}) and~(\ref{eqn:twist_model_pure_LowerRank_b}) are coupled linear partial differential equations in two complex variables; see \cites{MR960834, MR960835}. For families of elliptic curves, the Picard-Fuchs systems are obtained from the generalized hypergeometric system satisfied by the Appell function $F_1$ by rational pullback. For K3 surfaces, the Picard-Fuchs systems are obtained from the generalized hypergeometric system satisfied by the Appell function $F_2$ by rational pullback. Some explicit examples were determined in \cites{MR3767270, Malmendier:2019aa}. 
\end{remark}
\subsubsection{Twists by rational surfaces}
\label{sssec:product_twists}
A rational elliptic surface $X' \to B=\mathbb{P}^1 \backslash \lbrace 0, 1, \infty\rbrace$ has a singular fiber of Kodaira-type $IV^*$, $III^*$, or $II^*$ over $t=\infty$, if and only if its Weierstrass model is given by
\beq
\label{eqn:rational_family_twist}
 y^2 z = 4x^3-g'_2(t)\, x z^2-g'_3(t) \, z^3 \,,
\eeq 
and the degrees of the polynomials $g'_2(t)$ and $g'_3(t)$ are at most $1$ and $2$, respectively. In fact, we will focus on three cases $X'_k$ for $k=1, 2, 3$ in Table~\ref{tab:twist_rationalsurface} where the rational elliptic surfaces are extremal and given in Table~\ref{tab:3ExtRatHg}; the surfaces themselves will be discussed in Section~\ref{ssec:ec_families}, and the number $\mu$ associated with each $X'_k$ surface will be explained in Corollary~\ref{curve_reduction}.
\begin{table}[H]
\scalebox{\ScaleBig}{
\begin{tabular}{|c|c|c|c||c|c|}
\hline
$k$ & $X'_k$ & $\mu$ & fiber at $t=\infty$ & $\deg_t(g_2)$ & $\deg_t(g_3)$ \\
\hline
\hline
&&&&&\\[-0.9em]
$4$ & $X'_4=X_{431}$ & $\frac{1}{3}$ & $IV^*$ & $1$ & $2$ \\[0.2em]
$3$ & $X'_3=X_{321}$ & $\frac{1}{4}$ & $III^*$ & $1$ & $1$ \\[0.2em]
$2$ & $X'_2=X_{211}$ & $\frac{1}{6}$ & $II^*$  & $0$ & $1$  \\[0.2em]
\hline
\end{tabular}}
\caption{Rational surfaces used for twisting}\label{tab:twist_rationalsurface}
\end{table}
\par We fix any such rational surface $X'_k$. Given a second rational elliptic surface $X \to B$ with Weierstrass model
\beq
\label{eqn:rational_family}
 y^2 z = 4x^3-g_2(t)\, x z^2-g_3(t) \, z^3 \,,
\eeq
with degrees of $g_2(t)$ and $g_3(t)$ being at most $4$ and $6$, respectively, we define a family of Gorenstein threefolds $\tilde{\pi}: \tilde{X} \to B$ such that each element $\tilde{X}_{\tilde{t}}=\tilde{\pi}^{-1}(\tilde{t})$  is a compact complex threefold and also elliptically fibered with section over the Hirzebruch surface $\mathbb{F}_n$. We call this family the \emph{twist family of $X \to B$ with $X'_k \to B$}. Here, the Hirzebruch surface is given by $\mathbb{F}_n = \mathbb{P}(\, \mathcal{O}_{\mathbb{P}^1}(-n) \oplus  \mathcal{O}_{\mathbb{P}^1})$ with $n=0, \dots, k$.
\par We define polynomials
\beq
\label{eqn:product_twist_gs}
 g_2(t, u, v) := h(u, v)^4 v^4 g_2\left(\frac{t}{v}\right)\,, \quad  g_3(t, u, v) := h(u, v)^6 v^6 g_3\left(\frac{t}{v}\right)\,,
\eeq
where we have used $h(u, v)^2= 4u^3 - g'_2(v) \, u - g'_3(v)$.  We consider $\mathbb{F}_n$ the quotient space of $\mathbb{C}^4 \backslash \lbrace u_0=u_1=v_0=v_1=0 \rbrace$ by the action of $\mathbb{C}^* \times \mathbb{C}^*$ given by
\beqn
 (\lambda_2, \lambda_3) \cdot [u_0:u_1:v_0:v_1 ] = [  \lambda_3 u_0:  \lambda_3 u_1: \lambda_3^n \lambda_2 v_0: \lambda_2 v_1 ] \,,
\eeqn
It follows that $\mathbb{F}_0 = \mathbb{P}^1 \times \mathbb{P}^1$, and $\mathbb{F}_1$ is the non-minimal surface obtained as the blow-up of $\mathbb{P}^2$ in one point.
For $\tilde{t} \in \mathbb{P}^1$ and coordinates $[u_0:u_1:s_0:s_1 ]\in\mathbb{F}_n$ and $[\tilde{x}:\tilde{y}:\tilde{z}] \in \mathbb{P}(2,3,1)$, let the Weierstrass model $\tilde{W}_{\tilde{t}}$ for $\tilde{X}_{\tilde{t}}$ be given by
\beq
\label{CY3mtWEf}
 \tilde{y}^2 \tilde{z}= 4 \tilde{x}^3 - g_2\left( \tilde{t}, \frac{u_0}{u_1}, \frac{v_0}{v_1 \, u_1^n} \right) \,  u_1^{4(n+2)} v_1^8 \tilde{x} \, \tilde{z}^2  
- g_3\left( \tilde{t}, \frac{u_0}{u_1}, \frac{v_0}{v_1 \, u_1^n} \right) \, u_1^{6(n+2)} v_1^{12} \tilde{z}^3 \,,
\eeq 
with $g_2(t, u, v)$ and $g_3(t, u, v)$ given in Equation~(\ref{eqn:product_twist_gs}). We have the following:
\begin{lemma}
\label{lem:rational_twist}
In the situation above, Equation~(\ref{CY3mtWEf}) is a minimal and normal Weierstrass model over $\mathbb{F}_n$ for $n=0, \dots,k$ such that $\pmb{\omega}_{\tilde{X}_{\tilde{t}}} \cong \mathcal{O}_{\tilde{X}_{\tilde{t}}}$ for $\tilde{t} \in B$.
\end{lemma}
\begin{proof}
The first part follows by checking that the Weierstrass model obtained for each rational elliptic surface $X'_k$ in Table~\ref{tab:twist_rationalsurface} is minimal. Then, three $\mathbb{C}^*$-group actions on the defining variables in Equation~(\ref{CY3mtWEf}) are given by the weights listed in Table~\ref{tab:deg5} where \emph{deg} denotes the total weight of Equation~(\ref{CY3mtWEf}) and \emph{sum} denotes the sum of weights of the defining variables.  If the condition is satisfied that
the total weight of Equation~(\ref{CY3mtWEf}) equals the sum of weights of the defining variables for each acting $\mathbb{C}^*$ in Equation~(\ref{CY3mtWEf}),
then a Calabi-Yau threefold is obtained by removing the loci $\lbrace u_0 = u_1 = 0\rbrace$, $\lbrace v_0 = v_1 = 0\rbrace $, $\lbrace x = y = z = 0\rbrace$ from
the solution set of Equation~(\ref{CY3mtWEf}), and taking the quotient $(\mathbb{C}^*)^3$.  Details are explained in \cite{MR1412112}.
\end{proof}
\begin{table}[H]
\scalebox{\ScaleBig}{
\begin{tabular}{|c|c|ccc|cc|cc|c|}
\hline
$\mathbb{C}^*$ & deg & $x$ & $y$ & $z$ & $v_0$ & $v_1$ & $u_0$ & $u_1$ & sum \\
\hline
\hline
$\lambda_1$ & $3$ & $1$ & $1$ & $1$ & $0$ & $0$ & $0$ & $0$ & $3$\\
$\lambda_2$ & $12$ & $4$ & $6$ & $0$ & $1$ & $1$ & $0$ & $0$ & $12$\\
$\lambda_3$ & $6(n+2)$ & $2(n+2)$ & $3(n+2)$ & $0$ & $n$ & $0$ & $1$ & $1$ & $6(n+2)$\\
\hline
\end{tabular}}
\caption{Weights of variables in Weierstrass equation}\label{tab:deg5}
\end{table}
\par The differential one-form $\eta_t=dx/y$ (and $\eta'_t=dx/y$) defines a flat section of the conormal bundle for the family in Equation~(\ref{eqn:rational_family}) (resp.~Equation~(\ref{eqn:rational_family_twist})) in the affine chart $z=1$. A flat section $\tilde{\eta} \in \Gamma(\mathbb{P}^1,\tilde{\mathcal{H}})$ for the twisted family~(\ref{CY3mtWEf}) is then given by the differential form 
\beq
  \tilde{\eta}_{\tilde{t}} = dv \wedge du \wedge \frac{d\tilde{x}}{\tilde{y}} \,,
\eeq 
in the affine chart $[u_0:u_1:v_0:v_1]=[u:1:v:1]$ and $\tilde{z}=1$, such that
\beq
  \tilde{\eta}_{\tilde{t}} = \frac{dv}{v} \wedge \frac{du}{h(u, v)} \wedge \eta_{\tilde{t}/v} = \frac{dv}{v} \wedge \eta'_{v} \wedge \eta_{\tilde{t}/v} \,.
\eeq
Thus, if $\omega(t)$ (and $\omega'(t)$) is a local section of the period sheaf of the rational elliptic surface $X \to \mathbb{P}^1$ (resp.~$X' \to \mathbb{P}^1$), a section of the period sheaf $\Pi(\tilde{\mathcal{H}},\tilde{\eta}) \to \mathbb{P}^1$ is given by
\beq
\label{eqn:period_new_producttwist}
  \tilde{\omega}(\tilde{t}) = \oint_{C} \frac{dv}{v}  \; \omega'(v) \; \omega\left( \frac{\tilde{t}}{v}\right) \,,
\eeq
where $C$ is a non-contractible loop in the punctured $v$-plane. We have the following:
\begin{proposition}
\label{prop:period_integral_general_producttwist}
Let $\omega(t)$ (and $\omega'(t)$) be a period integral for the rational elliptic surface $X \to \mathbb{P}^1$ (resp.~$X' \to \mathbb{P}^1$ with a singular fiber of type $IV^*$, $III^*$, or $II^*$ at $t=\infty$) with absolutely convergent series $\omega(t) = \sum_{k\ge 0} f_k t^k$ (resp.~$\omega'(t) = \sum_{k\ge 0} f'_k t^k$) for $|t|<1$. Let $C_{1/2}(0)$ be the circle  $|v|=\frac{1}{2}$ in the $v$-plane with counterclockwise orientation. For every $\tilde{t}\in \mathbb{C}$ with $|\tilde{t}|<1/2$, the period integral~(\ref{eqn:period_new_producttwist}) for $C=C_{1/2}(0)$ has an absolutely convergent series given by the Hadamard product
\beq
\label{eq:period_new_producttwist}
\begin{split}
 \tilde{\omega}(\tilde{t}) = (2\pi i) \; \omega'(\tilde{t}) \star  \omega(\tilde{t}) \,.
\end{split}
\eeq
\end{proposition}
\begin{proof} 
The proof is analogous to the proof of Proposition~\ref{prop:period_integral_general}. For $C=C_{1/2}(0)$ a simple residue computation in Equation~(\ref{eqn:period_new_producttwist}) yields
\beqn
  \tilde{\omega}(\tilde{t}) = \sum_{m,n} f'_m  f_n t^n \, \oint_{C} \frac{dv}{v^{1-m+n}} = (2\pi i) \;  \sum_{m} f'_m \, f_m t^m \,,
\eeqn
\end{proof}
 \section{Modular elliptic families and related families}
 \label{sec:modular}
In this section, we will show that all extremal, modular elliptic families of elliptic curves or K3 surfaces with three singular fibers are in fact twisted families for some generalized functional invariant $(i, j, \alpha)$.  In our iterative construction, the universal starting point is always the same pencil $\pi: X \to B$ of `dimension zero' Calabi-Yau manifolds, namely the ramified family of pairs of points  encountered before $\{\pm y\} \in X_t=\pi^{-1}(t)$ given by
\beq
\label{legendre0b}
 y^2 = 1 - t 
\eeq
with $t \in \mathbb{C}$. For the family~(\ref{legendre0b}), we define $\Sigma_0(t)$ to be the point $t$ for $t \in B$, take the branch cut  branch cut along $\lbrace \, t \, | \, 1\le t \le \infty\rbrace$, and consider the holomorphic $0$-form $1/y_0$ with the period
\beq
 \int_{\Sigma_0(t)} \frac{2}{y_0} = \frac{1}{y_0} =  {}_1F_0\!\left(\left. \frac{1}{2} \right| t \right)  = \big( 1- t \big)^{-\frac{1}{2}}\;,
\eeq
which is a solution to the hypergeometric differential equation $L^{(1)}_t(\frac{1}{2};) \, \omega(t)=0$. 
\subsection{Extremal families of elliptic curves}
\label{ssec:ec_families}
Applying our twist construction to the family~(\ref{legendre0b}), we obtain families of one-dimensional Calabi-Yau manifolds, namely modular families of elliptic curves. We have the following:
\begin{lemma}
\label{lemmaECmixed}
For $(i, j, \alpha)$ with $1\le i \le 2$, $1 \le j \le 2\, \alpha$ and $\alpha \in \lbrace \frac{1}{2}, 1\rbrace$, the twisted families with generalized functional invariant $(i, j, \alpha)$ given by
\beq
\label{family_of_curves}
 \tilde{y}^2 = \left( 1 - \dfrac{ c_{ij}\tilde{t}}{\tilde{x}^i (\tilde{x}+1)^j} \right) \, \tilde{x}^2 \, (\tilde{x}+1)^{2\alpha} \;,
\eeq
are families of genus-one curves. For $(i, j, \alpha)=(1,1,1)$, $(2,1,1)$, $(1,1,\frac{1}{2})$, and $(2,1,\frac{1}{2})$, the families admit a $\mathbb{C}(\tilde{t})$-rational Weierstrass point. The corresponding families of elliptic curves are rational elliptic surfaces with singular fibers given in Table \ref{tab:twist_params}; the explicit Weierstrass models and Mordell-Weil groups are given in Table \ref{tab:3ExtRatHg}.
\end{lemma}
\begin{proof}
The result follows by explicit computation.
\end{proof}
\begin{remark}
The points of maximal unipotent monodromy for an elliptic surface comprise the support of singular fibers of type $I_n$ for $n \ge 1$ \cite{MR1779161}*{Cor.~1}.
\end{remark}
\begin{remark}
The names of the Jacobian elliptic surfaces used in Table \ref{tab:twist_params} and Table~\ref{tab:3ExtRatHg} coincide with the classical notation used by Herfurtner \cite{MR1129371}. 
\end{remark}
\begin{table}[H]
\scalebox{\ScaleBig}{
\begin{tabular}{|l|l|r|c|}
\hline
 &&&\\[-0.9em]
$(i, j, \alpha)$ & $\mu$ & singular fibers & notation\\[0.2em]
\hline
\hline
&&& \\[-0.9em]
$(1,1,1)$ & $\frac{1}{2}$ & $I_1^* (\tilde{t}=\infty), \; I_4 (\tilde{t}=0), \;  I_1 (\tilde{t}=1) $ & $X_{141}$\\[0.4em]
$(2,1,1)$ & $\frac{1}{3}$ & $IV^*  (\tilde{t}=\infty), \; I_3 (\tilde{t}=0), \;  I_1 (\tilde{t}=1) $ & $X_{431}$\\[0.4em]
$(1,1,\frac{1}{2}) $ & $\frac{1}{4}$ & $ III^* (\tilde{t}=\infty), \; I_2 (\tilde{t}=0), \;  I_1 (\tilde{t}=1)$ & $X_{321}$ \\[0.4em]
$\left(2,1,\frac{1}{2}\right)$ & $\frac{1}{6}$ & $II^*  (\tilde{t}=\infty), \; I_1 (\tilde{t}=0), \;   I_1 (\tilde{t}=1)$ & $X_{211}$ \\[0.4em]
\hline
\end{tabular}}
\caption{Families of elliptic curves}\label{tab:twist_params}
\end{table}
\begin{remark}
Conditions~(\ref{cond:normal}) allow for a fifth case with $(i, j, \alpha)=(2,2,1)$. In this remaining case, the twisted family with generalized functional invariant $(i, j, \alpha)=(2,2,1)$ is a family of genus-one curves whose relative Jacobians coincide with the elliptic curves obtained for the twisted family with $(i, j, \alpha)=(1,1,\frac{1}{2})$.
\end{remark}
\begin{table}[H]
\scalebox{\Scale}{
\begin{tabular}{|c|c|lcl|c|c|c|c|} \hline
name & $G$& \multicolumn{3}{|c|}{$g_2, g_3, \Delta, J$} & \multicolumn{4}{|c|}{ramification of $J$ and singular fibers} \\[0.2em]
\cline{2-9} 
& $\operatorname{MW}(\pi,\sigma)$ &  \multicolumn{3}{|c|}{sections}  & $t$ & $J$ & $m(J)$ & fiber   \\
\hline
\hline
&&&&&&&& \\[-0.9em]
$X_{141}$            & $\Gamma_0(4)$ 
  & $g_2$     & $=$ & $\frac{1}{3}\left(4 t^ 2 -64 \, t+ 64 \right)$                      	& $8 \pm 4 \sqrt{3}$ 			                          & $0$       & $3$ & smooth  \\[0.4em]
$\mu=\frac{1}{2}$ & 
  & $g_3$     & $=$ & $\frac{8}{27}\big(2 -t\big) \big(32 -32 t -t^2\big)$   		& $-16\pm12 \sqrt{2}, 2$				        & $1$       & $2$ & smooth  \\[0.4em]
&& $\Delta$ & $=$ & $-256 \, t^4 \, (t-1)$    							& $0$								& $\infty$ & $4$ & $I_4 \;\; (A_3)$     \\[0.0em]
&& $J$         & $=$ & $-\frac{(t^2-16 t+16)^3}{108 t^4 (t-1)}$       			& $1$                                        				& $\infty$ & $1$ & $I_1$     \\[0.3em]
\cline{2-5}
&&&&&&&& \\[-0.9em]			   			   
&$\mathbb{Z}/4\mathbb{Z}$& $(X,Y)_1$  &  $=$ & $(\frac{2}{3}t-\frac{4}{3},0)$ 		 		  	& $\infty$								& $\infty$ & $1$ & $I_1^* \; (D_5)$ \\[0.4em]
&& $(X,Y)_{2,3}$  &  $=$ & $(-\frac{1}{3}t-\frac{4}{3},\pm4\,i\,t)$ 			& &&&\\[0.4em]
\hline
&&&&&&&& \\[-0.9em]
$X_{431}$            & $\Gamma_0(3)$ 
  & $g_2$     & $=$ & $27 - 24 \, t$                      						  & $\infty$	                                                  & $0$       & $1$ & $IV^* \; \; (E_6)$  \\[0.4em]
$\mu=\frac{1}{3}$ & 
  & $g_3$     & $=$ & $27-36 \, t + 8 \, t^2$  							  & $\frac{9}{8}$ 						& $0$      & $3$ &  smooth \\[0.4em]
&& $\Delta$ & $=$ & $-1728 \, t^3 \, (t-1)$   							  & $\frac{9}{4}\pm\frac{3\sqrt{3}}{4}$			& $1$      & $2$ &  smooth    \\[0.1em]
&& $J$         & $=$ & $\frac{(-9+8 \, t)^3}{64 \, t^3 (t-1)}$                                    & $0$                                        			& $\infty$ & $3$ & $I_3 \; \; (A_2)$     \\[0.3em]
\cline{2-5}
&&&&&&&& \\[-0.9em]			   			   
&$\mathbb{Z}/3\mathbb{Z}$& $(X,Y)_{1,2}$  &  $=$ & $(-\frac{3}{2},\pm 2\sqrt{2}\, i \,t)$ 			 	  & $1$								& $\infty$ & $1$ & $I_1$ \\[0.4em]
\hline
&&&&&&&& \\[-0.9em]
$X_{321}$            & $\Gamma_0(2)$ 
  & $g_2$     & $=$ & $\frac{16}{3}- 4 \, t$                      			  		  & $\frac{4}{3}$					                & $0$       & $3$ & smooth  \\[0.4em]
$\mu=\frac{1}{4}$ & 
  & $g_3$     & $=$ & $-\frac{64}{27}+\frac{8}{3} \, t$    					  & $\frac{8}{9}$    						& $1$       & $2$ & smooth  \\[0.4em]
&& $\Delta$ & $=$ & $- 64 \, t^2 \, (t-1)$    							  & $\infty$							& $1$ 	& $1$ & $III^*\; \; (E_7)$     \\[0.0em]
&& $J$         & $=$ & $\frac{(- 4 + 3 t)^3}{27 \, t^2 (t-1)}$      	   	  		  & $0$                                        			& $\infty$ & $2$ & $I_2 \; \; (A_1)$     \\[0.4em]
\cline{2-5}
&&&&&&&& \\[-0.9em]			   			   
&$\mathbb{Z}/2\mathbb{Z}$& $(X,Y) $  &  $=$ & $(\frac{2}{3},0)$ 						 		  & $1$								& $\infty$ & $1$ & $I_1 $ \\[0.4em]
\hline
&&&&&&&& \\[-0.9em]
$X_{211}$            &  $\Gamma_0(1)$
  & $g_2$     & $=$ & $3$                      			  		  			  & $\infty$					                & $0$        & $2$ & $II^*\; \; (E_8)$  \\[0.4em]
$\mu=\frac{1}{6}$ & 
  & $g_3$     & $=$ & $-1 + 2 \, t$    						  			  & $\frac{1}{2}$    						& $1$       & $2$ & smooth  \\[0.4em]
&& $\Delta$ & $=$ & $- 108 \, t \, (t-1)$    							  & $0$								& $\infty$ 	& $1$ &  $I_1$    \\[0.0em]
\cline{2-2}
&&&&&&&& \\[-0.9em]			   			   
&$\lbrace \mathbb{I} \rbrace$& $J$         & $=$ & $- \frac{1}{4 \, t \, (t-1)}$      	   	  				  & $1$                                        			& $\infty$  & $1$ &  $I_1$     \\[0.4em]
\hline
\end{tabular}}
\caption{Extremal rational fibrations}\label{tab:3ExtRatHg}
\end{table}
A Jacobian elliptic surfaces is called \emph{extremal} if the rank of its Mordell-Weil group of sections vanishes. We have the following:
\begin{corollary} 
\label{cor:extremal}
The elliptic surfaces in Table \ref{tab:twist_params} and Table \ref{tab:3ExtRatHg} constitute all extremal families of elliptic curves with three singular fibers and rational total space  (up to quadratic twist and two-isogeny). Moreover, $\tilde{t}=0$ is a point of maximal unipotent monodromy for each family.
\end{corollary}
\begin{proof}
The extremal families of elliptic curves with three singular fibers were classified in \cite{MR867347}*{Tab.~5.2}. The surface $X_{411}$ in \cite{MR867347} is obtained from $X_{141}$ by quadratic twist. Similarly, $X_{222}$ is obtained from $X_{411}$ by fiberwise two-isogeny. The point of maximal unipotent monodromy in the base curve of the elliptic surfaces is the support of the singular fiber of type $I_n$  for $n \ge 1$ \cite{MR1779161}*{Cor.~1}. 
\end{proof}
Application of results in \cite{MR0265365} yields the following:
\begin{corollary}
\label{cor:modular_rational}
The families of elliptic curves in Table \ref{tab:twist_params} with fiber of type $I_n$ for $n=2, 3, 4$ and $n=1$ are the universal families of elliptic curves over the genus-zero modular curves for the congruence subgroups $\Gamma_0(n)$, and for $n=1$ for the unique normal subgroup of index $2$ in $\mathrm{PSL}(2,\mathbb{Z})$ denoted by $\Gamma_0(1)$. The family parameter $\tilde{t}$ is the corresponding Hauptmodul of the modular curve.  
\end{corollary}
\qed
\par The twist construction also provides us, near $t=0$, with a family of non-contractible one-cycles for each family in Lemma~\ref{lemmaECmixed}. We have:
\begin{lemma}
\label{lemma_cycle}
For $\tilde{t}\in \mathbb{C}$ with $|\tilde{t}|<1/(2^{i+j+1}|c_{i j}|)$, the circle $C=C_{1/2}(0)$, given by $|\tilde{x}|=\frac{1}{2}$ in the $\tilde{x}$-plane with counterclockwise orientation, determines a family of non-contractible A-cycles $\Sigma_1(\tilde{t})$ for each family of genus-one curves in Lemma~\ref{lemmaECmixed}.
\end{lemma}
\begin{proof}
For $|\tilde{t}|<1/(2^{i+j+1}|c_{i j}|)$, each family $\tilde{X} \to \mathbb{P}^1 \ni \tilde{t}$ of genus-one curves in Lemma~\ref{lemmaECmixed} has two branch points with $|\tilde{x}|<\frac{1}{2}$ and two branch points with $|\tilde{x}|>\frac{1}{2}$. Therefore, there is a non-contractible one-cycle $\Sigma_1(\tilde{t}) \subset \tilde{X}_{\tilde{t}}$ that projects onto the circle $C=C_{1/2}(0)$, i.e., $|\tilde{x}|=\frac{1}{2}$ with counterclockwise orientation, and varies continuously for $|\tilde{t}|<1/(2^{i+j+1}|c_{i j}|)$. At $\tilde{t}=0$, the two branch points inside $C$ coalesce, and $\Sigma_1(\tilde{t})$ constitutes a family of A-cycles.
\end{proof}
We then have the following:
\begin{corollary}
\label{curve_reduction}
For the twisted families in Lemma~\ref{lemmaECmixed} with generalized functional invariant $(i, j, \alpha)=(1,1,1)$, $(2,1,1)$, $(1,1,\frac{1}{2})$, or $(2,1,\frac{1}{2})$, the period integral~(\ref{eqn:period_new}) is annihilated by the Picard-Fuchs operator $L^{(2)}_{\tilde{t}}\big((\mu, 1-\mu);(1)\big)$. In particular, the period over $\Sigma_1(\tilde{t})$ is holomorphic at $\tilde{t}=0$ and given by
\beq
\label{rk2hgf}
\omega(\tilde{t})= (2 \pi i) \;\hpg21{ \mu, 1-\mu}{1}{\tilde{t}} 
\eeq
with $\mu=\frac{\alpha}{i+j}$. 
\end{corollary}
\begin{proof}
The proof follows by application of Proposition~\ref{prop:period_integral_general} for the period integral $\omega(t)= {}_1F_0(\frac{1}{2}|t)$ and generalized functional invariant $(i, j, \alpha)$.
\end{proof}
\subsection{Families of  \texorpdfstring{$M_n$}{Mn}-lattice polarized K3 surfaces}
\label{ssec:Mn-polarized}
The procedure described in Section~\ref{sssec:connection_with_qtwist} allows us to construct the quadratic twists of the rational elliptic surfaces in Lemma~\ref{lemmaECmixed}. The four resulting families of K3 surfaces realize precisely the families of K3 surfaces considered by Hoyt in \cite{MR894512}. These K3 surfaces are not modular elliptic surfaces, but rather rational covers of them. To obtain modular elliptic K3 surfaces, we consider the twisted families with generalized functional invariant $(i, j, \alpha)=(1,1,1)$ of the Jacobian elliptic surfaces in Lemma~\ref{lemmaECmixed} instead. We denote by $M_n$ the lattices $M_n = H \oplus E_8 \oplus E_8 \oplus \langle -2n \rangle$ for $n\in \mathbb{N}$. We have the following:
\begin{lemma}
\label{lemmaK3mixed}
For $(n,\mu) \in \left\lbrace (1,\frac{1}{6}), (2, \frac{1}{4}), (3,\frac{1}{3}), (4,\frac{1}{2})\right\rbrace$, the twisted families with generalized functional invariant $(i,j,\alpha)=(1,1,1)$ given by
\beq
\label{lemmaK3mixed_eq}
 Y^2 = 4 \, X^3 - \underbrace{g_2\left( -\frac{t}{4u(u+1)} \right) \, \Big(u \, (u+1) \Big)^4}_{=: \,g_2(t,u)} \, X - \underbrace{g_3\left( -\frac{t}{4u(u+1)} \right) \, \Big(u \, (u+1) \Big)^6}_{=: \,g_3(t,u)} \,,
\eeq
with $g_2(t)$ and $g_3(t)$ determined by $(n,\mu)$ in Table~\ref{tab:3ExtRatHg}, define families of Jacobian elliptic K3 surfaces of Picard rank 19 with two singular fibers of Kodaira-type $II^*, III^*, IV^*$, or $I_1^*$ at $u=0$ and $u=-1$,  a fiber of Kodaira-type $I_{2n}$ at $u=\infty$, and two singular fibers of type $I_1$ at $u=-1/2 \pm \sqrt{1-t}/2$.
The families are families of $M_n$-lattice polarized K3 surfaces for $n=1, \dots, 4$.
\end{lemma}
\begin{proof}
The proof amounts to checking the Kodaira-types of singular fibers from $G_2, G_3, \Delta$ and comparing with the list of all Jacobian elliptic surfaces
given in the Shimada classification \cite{MR1813537} of Jacobian elliptic K3 surfaces. The fact that the constructed families of K3 surfaces are $M_n$-polarized 
was explained by Dolgachev in \cite{MR1420220}. Moreover, it is easy to see that all torsion sections of the rational elliptic surfaces
survive the mixed-twist construction.  The torsion sections are listed in Table~\ref{tab:torsions3}.
\end{proof}
 \begin{table}[H]
 \scalebox{\ScaleBig}{
 \begin{tabular}{|c|c|lcl|}
 \hline
 &&&&\\[-0.9em]
 $\Lambda$ & torsion &  \multicolumn{3}{|c|}{sections} \\[0.2em]
  \hline
  \hline
  &&&&\\[-0.9em]
 $M_4$ & $[4]$ & $(X,Y)_1$ & $=$ &$\left(-\frac{1}{6}\, \left( u+1 \right) u \left( 8\,{u}^{2}+t+8\,u \right) ,0\right)$\\[0.2em]
 && $(X,Y)_{2,3}$ & $=$ &$\left( \frac{1}{12}\, \left( u+1 \right) u \left( -16\,{u}^{2}+t-16\,u
 \right) ,\pm it{u}^{2} \left( u+1 \right) ^{2}\right)$\\[0.2em]
 \hline
   &&&&\\[-0.9em]
 $M_3$ & $[3]$ & $(X,Y)_{1,2}$ & $=$ &$\left(-\frac{3}{2}\,{u}^{2} \left( u+1 \right) ^{2},\pm \frac{i}{2}\sqrt {2}t{u}^{2} \left( u+1 \right) ^{2} \right)$\\[0.2em]
\hline
   &&&&\\[-0.9em]
 $M_2$ & $[2]$ & $(X,Y)$ & $=$ &$\left(\frac{2}{3}\,{u}^{2} \left( u+1 \right) ^{2},0\right)$\\[0.2em]
\hline
  &&&&\\[-0.9em]
 $M_1$ & $[1]$ & $-$ &  &\\[0.2em]
\hline
  \end{tabular}}
\caption{Torsions sections of Equation~(\ref{lemmaK3mixed_eq})}\label{tab:torsions3}
\end{table}
The configurations of singular fibers, the Mordell-Weil groups, the determinants of the discriminant groups, and the lattice polarizations are summarized in Table \ref{tab:Twists_3ExtRatHg1}.
\begin{table}[H]
\scalebox{\ScaleBig}{
\begin{tabular}{|c|c|c|ccc|c|c|c|} \hline
\multicolumn{2}{|c|}{derived from}& $\rho$ & \multicolumn{3}{|c|}{Configuration of singular fibers} & $\operatorname{MW}(\pi,\sigma)$ & $\operatorname{disc}Q$ & $\Lambda$\\
\cline{1-2}\cline{4-6}
 &&&&&&&& \\[-1em]
Srfc & $\mu$ && $u=\infty$ & $u^2+u+t/4$ & $u=0,-1$ &&&\\
 \hline
 \hline
 &&&&&&&& \\[-0.8em]
 $X_{141}$ & $\frac{1}{2}$ & $19$ & $I_8 \; (A_7)$ 		& $2 \, I_1$ 	& $2 \, I_1^* \; (2\, D_5)$  	& $[4]$ & $2^3$   & $M_4$\\[0.4em]
\hline 
&&&&&&&& \\[-0.8em]
$X_{431}$ & $\frac{1}{3}$ & $19$ & $I_6 \; (A_5)$ 	& $2 \, I_1$ 	& $2 \, IV^* \; (2\, E_6)$ 		& $[3]$ & $2\cdot3$ & $M_3$\\[0.4em]
\hline
&&&&&&&& \\[-0.8em]
$X_{321}$ & $\frac{1}{4}$ & $19$ & $I_4 \; (A_3)$ 	& $2 \, I_1$ 	& $2 \, III^* \; (2\, E_7)$ 		& $[2]$ & $2^2$ &  $M_2$\\[0.4em]
\hline
&&&&&&&& \\[-0.8em]
$X_{211}$ & $\frac{1}{6}$ & $19$ & $I_2 \; (A_1)$ 	& $2 \, I_1$ 	& $2 \, II^* \; (2\, E_8)$ 		& $[1]$ & $2$  &$M_1$\\[0.4em]
\hline 
\end{tabular}}
\caption{K3 surfaces from extremal rational surfaces}\label{tab:Twists_3ExtRatHg1}
\end{table}
\par Let us denote by $\Gamma_0(n)^+$ the modular group $\Gamma_0(n)$ extended by the Fricke involution, i.e., the element corresponding to $\tau \mapsto -1/(n\tau)$. 
It is a classical result due to Dolgachev that the moduli spaces of pseudo-ample $M_n$-polarized K3 surfaces are isomorphic to the rational modular curves that are the
compactification of the curves $\mathbb{H}/\Gamma_0(n)^+$  \cite{MR1420220}.
We therefore have the following:
\begin{corollary}
\label{MES+}
The twisted families with generalized functional invariant $(1,1,1)$ of the families of elliptic curves in Lemma~\ref{lemmaECmixed} are families of $M_n$-lattice polarized K3 surfaces over the rational modular curves $\mathbb{H}/\Gamma_0(n)^+$ for $n=1, 2, 3, 4$.
\end{corollary}
\begin{proof}
The proof follows directly by checking that the singular fibers and Mordell-Weil groups for the families constructed in Lemma~\ref{lemmaK3mixed}
agree with the ones given by Dolgachev in \cite{MR1420220}.
\end{proof}
\par For each twisted family in Equation~(\ref{lemmaK3mixed_eq}), we define a family of closed two-cycles $\Sigma_2(t)$ as follows: for $t \in \mathbb{C}$ with $|t|<1/2$ let $C=C_{1/2}(0)$ be the circle $|u|=\frac{1}{2}$ in the $u$-plane with counterclockwise orientation. For every $u \in C$, a cycle $\Sigma_1'(t, u)$ in the elliptic fiber is obtained from $\Sigma_1(-\frac{t}{4u(u+1)})$ -- where $\Sigma_1(t)$ was defined in Lemma~\ref{lemma_cycle} -- by rescaling $(X,Y) \to (u^2 (u+1)^2 X, u^3 (u+1)^3 Y)$. For $t\in \mathbb{C}$ with $|t|<1/2$, we obtain a continuously varying family of closed two-cycles as a warped product $\Sigma_2(t) = C \times_u \Sigma'_1(t, u)$.  By warped product we mean that the cycle $\Sigma_1(-\frac{t}{4u(u+1)})$ is warped, i.e. it is rescaled by a function of the affine coordinate $u$. We then have the following:
\begin{corollary}
\label{cor:modular_K3_surfaces}
For the twisted families with generalized functional invariant $(1,1,1)$ in Lemma~\ref{lemmaK3mixed}, the period integral~(\ref{eqn:period_new}) is annihilated by the Picard-Fuchs operator $L^{(3)}_{t}((\mu, 1/2, 1-\mu);(1,1))$. In particular, the period over $\Sigma_2(t)$ is holomorphic at $t=0$ and given by
\beq
\label{period_lemmaK3pure}
  \omega  = (2 \pi i)^2 \;  \hpg32{ \mu, \frac{1}{2}, 1-\mu}{1, 1}{t}  \,.
\eeq
\end{corollary}
\begin{proof}
Application of Proposition~\ref{prop:period_integral_general} for the period integral $\omega(t)= {}_2F_1\big( \mu , 1-\mu; 1  \big| \, t \big)$ and the twisted families of Weierstrass models in Equation~(\ref{lemmaK3mixed_eq}) yields the following formula
\beq
  {}_1F_0\!\left(\left. \frac{1}{2} \right| t \right)  \star  \hpg21{\mu, 1-\mu}{1}{t}  =   \hpg32{ \mu, \frac{1}{2}, 1-\mu}{1, 1}{t}   \,.
\eeq
\end{proof}
\par By Clausen's identity each period in Corollary~\ref{cor:modular_K3_surfaces} is a perfect square, namely
\beq
\label{Clausen}
   \hpg32{ \mu, \frac{1}{2}, 1-\mu}{1, 1}{t} = \left[  \, \hpg21{\frac{\mu}{2} , \frac{1-\mu}{2}}{1}{t} \, \right]^2 \,.
\eeq
Ratios of solutions of the hypergeometric differential equation with holomorphic solution ${}_2F_1\big( \frac{\mu}{2} , \frac{1-\mu}{2}; 1  \big| \, t \big)$ are so-called Schwarzian $s$-maps, and the corresponding triangle groups are the modular groups $\Gamma_0(n)^+$ in Corollary~\ref{MES+} and listed in Table \ref{tab:triangle}.
\begin{table}[H]
\scalebox{\ScaleBig}{
\begin{tabular}{|c|c|c|}
\hline
& &\\[-0.9em]
$n$ & $\mu$ & triangle group \\
\hline
\hline
&& \\[-0.9em]
$n$ & $\mu$ & $(2,\frac{2}{1-2\mu},\infty)$ \\[0.4em] 
$1$ & $\frac{1}{6}$ & $(2,3,\infty)$ \\[0.4em]
$2$ & $\frac{1}{4}$ & $(2,4,\infty)$ \\[0.4em]
$3$ & $\frac{1}{3}$ & $(2,6,\infty)$ \\[0.4em]
$4$ & $\frac{1}{2}$ & $(2,\infty,\infty)$\\[0.4em]
\hline
\end{tabular}}
\caption{Triangle groups}\label{tab:triangle}
\end{table}
The so-called \emph{Kummer identity} relates the hypergeometric function on the right hand side of~(\ref{Clausen}) back to the original period, i.e., for
$t=4 \, T \, (1-T)$ it follows
\beq
\label{rel_quad}
\hpg21{\frac{\mu}{2} , \frac{1-\mu}{2}}{1}{t}= \hpg21{\mu ,1-\mu}{1}{T} \,.
\eeq
\par The geometric origin of Equation~(\ref{Clausen}) and~(\ref{rel_quad}) is the fact that an $M_n$-polarized K3 surface admits a Shioda-Inose structure relating it to a Kummer surface associated with the product of two isogenous elliptic curves; see \cite{MR1779161}.
\subsection{Families of  K3 surfaces of Picard-rank 18}
\label{ssec:M-polarized}
The twisted families with generalized functional invariant $(1,1,1)$ of the elliptic surfaces in Lemma~\ref{lemmaECmixed} are restrictions of two-parameter families of K3 surfaces with affine parameters $a, b \in \mathbb{C}$, as explained in Section~\ref{sssec:2parameters}. The families of Section~\ref{ssec:Mn-polarized} are obtained for $a=0$. The restriction $a=-b$ yields different one-parameter families of K3 surfaces of Picard rank 18.  We denote the relevant rank-18 lattices by $M=H\oplus E_8 \oplus E_8$, $\tilde{M}=H\oplus E_7^{\oplus \, 2} \oplus A_1^{\oplus \,2}/\mathbb{Z}_2$, $M'=H\oplus E_6^{\oplus \, 2} \oplus A_2^{\oplus \,2}/\mathbb{Z}_3$, and $\tilde{M}'=H\oplus D_5^{\oplus \, 2} \oplus A_3^{\oplus \,2}/\mathbb{Z}_4$. We have the following:
\begin{lemma}
\label{lemmaK3mixed18}
For $(n,\mu) \in \left\lbrace (1,\frac{1}{6}), (2, \frac{1}{4}), (3,\frac{1}{3}), (4,\frac{1}{2})\right\rbrace$,  the twisted families with generalized functional invariant $(i,j,\alpha)=(1,1,1)$ and $a=-b=t$ given by
\beq
\label{lemmaK3mixed18_eq}
\begin{split}
 Y^2 = 4 \, X^3 &- \underbrace{g_2\left(t \, \Big( 1+\frac{1}{2\,u(u+1)} \Big) \right) \, \Big(u \, (u+1) \Big)^4}_{=: \,g_2(t,u)} X \\
 & - \underbrace{g_3\left( t \, \Big( 1+\frac{1}{2\,u(u+1)} \Big)  \right) \, \Big(u \, (u+1) \Big)^6}_{=: \,g_3(t,u)} \;,
 \end{split}
\eeq
with $g_2(t)$ and $g_3(t)$ determined by $(n,\mu)$ in Table~\ref{tab:3ExtRatHg}, define families of Jacobian elliptic K3 surfaces of Picard rank 18 with two singular fibers of Kodaira-type $II^*, III^*, IV^*$, or $I_1^*$ at $u=0$ and $u=-1$,  two fibers of Kodaira-type $I_{n}$ at $2u^2+2u+1=0$, and two singular fibers of type $I_1$ at the roots of $p(t, u)=2(t-1)u(u+1)+t$. The families are polarized K3 surfaces with lattice polarization $M$, $\tilde{M}$, $M'$, and $\tilde{M}'$ for $n=1, \dots, 4$.
\end{lemma}
\begin{proof}
The proof amounts to checking the Kodaira-types of singular fibers from $G_2, G_3, \Delta$ and comparing with the list of all Jacobian elliptic surfaces
given in the Shimada classification \cite{MR1813537} of Jacobian elliptic K3 surfaces. The torsion sections for the families in Equation~(\ref{lemmaK3mixed18_eq}) are listed in Table~\ref{tab:torsions5}.
\end{proof}
 \begin{table}[H]
 \scalebox{\ScaleBig}{
 \begin{tabular}{|c|c|lcl|}
 \hline
 &&&&\\[-0.9em]
 $\Lambda$ & torsion &  \multicolumn{3}{|c|}{sections} \\[0.2em]
  \hline
  \hline
  &&&&\\[-0.9em]
 $M_4$ & $[4]$ & $(X,Y)_1$ & $=$ &$\left(\frac{1}{3}\, \left( u+1 \right) u \left( 2\,t{u}^{2}+2\,tu-4\,{u}^{2}+t-4\,u \right) ,0\right)$\\[0.2em]
 && $(X,Y)_{2,3}$ & $=$ &$\left( -\frac{1}{6}\, \left( u+1 \right) u \left( 2\,t{u}^{2}+2\,tu+8\,{u}^{2}+t+8\,u \right) ,\right.$\\
 && &  &$\quad \left.\pm it \left( 2\,u+1+i \right)  \left( -2\,u-1+i \right) {u}^{2} \left( u+1 \right) ^{2}
\right)$\\[0.2em]
 \hline
   &&&&\\[-0.9em]
 $M_3$ & $[3]$ & $(X,Y)_{1,2}$ & $=$ &$\left(-\frac{3}{2}\,{u}^{2} \left( u+1 \right) ^{2}, \right.$\\
 && &  &$\quad \left.\pm \frac{i}{2}\sqrt {2}t\left( 2\,u+1+i \right)  \left( -2\,u-1+i \right) {u}^{2} \left( u+1 \right) ^{2}\right)$\\[0.2em]
\hline
   &&&&\\[-0.9em]
 $M_2$ & $[2]$ & $(X,Y)$ & $=$ &$\left(\frac{2}{3}\,{u}^{2} \left( u+1 \right) ^{2},0\right)$\\[0.2em]
\hline
  &&&&\\[-0.9em]
 $M_1$ & $[1]$ & $-$ &  &\\[0.2em]
\hline
  \end{tabular}}
\caption{Torsions sections of Equation~(\ref{lemmaK3mixed18_eq})}\label{tab:torsions5}
\end{table}
\par The configurations of singular fibers, the Mordell-Weil groups, the determinants of the discriminant groups, and lattice polarizations are summarized in Table \ref{tab:Twists_3ExtRatHg3}.
\begin{table}[H]
\scalebox{\ScaleBig}{
\begin{tabular}{|c|c|c|ccc|c|c|c|} \hline
\multicolumn{2}{|c|}{derived from} & $\rho$ & \multicolumn{3}{|c|}{Configuration of singular fibers} & $\operatorname{MW}(\pi,\sigma)$ & $\operatorname{disc}Q$ & $\Lambda$\\
\cline{3-6}
 &&&&&&&& \\[-1em]
Srfc & $\mu$ && $2 u^2+2 u+1$ & $p(t,u)=0$ & $u=0,-1$ &&&\\
 \hline
 \hline
 &&&&&&&& \\[-0.9em]
$X_{141}$ & $\frac{1}{2}$ & $18$ & $2 \, I_4 \; (2 \, A_3)$ 	& $2 \, I_1$ 	& $2 \, I_1^* \; (2\, D_5)$ 	& $[4]$ & $2^4$ & $\tilde{M}'$\\[0.4em]
\hline 
 &&&&&&&& \\[-0.9em]
$X_{431}$ & $\frac{1}{3}$ & $18$ & $2 \, I_3 \; (2 \, A_2)$ 	& $2 \, I_1$ 	& $2 \, IV^* \; (2\, E_6)$ 	& $[3]$ & $3^2$ & $M'$\\[0.4em]
\hline
 &&&&&&&& \\[-0.9em]
$X_{321}$ & $\frac{1}{4}$ & $18$ & $2 \, I_2 \; (2 \, A_1)$ 	& $2 \, I_1$ 	& $2 \, III^* \; (2\, E_7)$ 	& $[2]$ & $2^2$ & $\tilde{M}$\\[0.4em]
\hline
 &&&&&&&& \\[-0.9em]
$X_{211}$ & $\frac{1}{6}$ & $18$ & $2 \, I_1$ 	& $2 \, I_1$ 	& $2 \, II^* \; (2\, E_8)$ 			& $[1]$ & $1$ & $M$\\[0.4em]
\hline 
\end{tabular}}
\caption{K3 surfaces from extremal rational surfaces}\label{tab:Twists_3ExtRatHg3}
\end{table}
We make the following:
\begin{remark}
Families of lattice polarized K3 surfaces in Lemma~\ref{lemmaK3mixed18} are restrictions of the general two-parameter families introduced in Section~\ref{sssec:2parameters}. The two-parameter family of $M$-lattice polarized K3 surfaces $\tilde{X}$ admits a Shioda-Inose structure, relating it to Kummer surfaces $\hat{X}$ associated with two non-isogenous elliptic curves \cite{MR2369941}. In fact, the latter admit a Jacobian elliptic fibration with singular fibers $2 I_0^*, II^*, 2 I_1$. It turns out that the double cover induced by the degree-two map~(\ref{eqn:2parameter_map}) on the base $\mathbb{P}^1$ induces a Hodge isometry between transcendental lattices $T_{\tilde{X}}(2) \cong T_{\hat{X}}$; see Figure~\ref{fig_geom}. Thus, in the case of $M$-lattice polarized K3 surfaces $\tilde{X}$ the period domain is 
\beqn
\Big( \mathrm{PSL}(2,\mathbb{Z}) \times \mathrm{PSL}(2,\mathbb{Z})  \Big) \rtimes \mathbb{Z}/2\mathbb{Z} \backslash \Big( \mathbb{H} \times \mathbb{H}\Big) \,,
\eeqn
the moduli space of two elliptic curves, with a generator of $\mathbb{Z}/2\mathbb{Z}$ acting on $\mathrm{PSL}(2,\mathbb{Z}) \times \mathrm{PSL}(2,\mathbb{Z})$ by exchanging the two sides.
\end{remark}
\par For each twisted family with generalized functional invariant $(i,j,\alpha)=(1,1,1)$ and $a=-b=t$   in Equation~(\ref{lemmaK3mixed18_eq}), we define a family of closed two-cycles $\Sigma_2(t)$ as follows: for $t \in \mathbb{C}$ with $|t|<1/2$ let $C=C_{1/2}(0)$ be the circle given by $|u|=\frac{1}{2}$ in the $u$-plane with counterclockwise orientation. For every $u \in C$, a cycle $\Sigma_1'(t, u)$ in the elliptic fiber is obtained from $\Sigma_1(t ( 1+\frac{1}{2\,u(u+1)}))$ -- where $\Sigma_1(t)$ was defined in Lemma~\ref{lemma_cycle} -- by rescaling $(X,Y) \to (u^2 (u+1)^2 X, u^3  (u+1)^3 Y)$. For $t\in \mathbb{C}$ with $|t|<1/2$, we obtain a continuously varying family of closed two-cycles as a warped product $\Sigma_2(t) = C \times_u \Sigma'_1(t, u)$. We have the following:
\begin{corollary}
\label{cor:modular_K3_surfaces_18}
For the twisted families  in Lemma~\ref{lemmaK3mixed18} with generalized functional invariant $(i,j,\alpha)=(1,1,1)$ and $a=-b=t$, the period integral~(\ref{eqn:period_new}) is annihilated by the Picard-Fuchs operator  $L^{(4)}_{t^2}\big((\frac{\mu}{2}, \frac{1-\mu}{2}, \frac{1+\mu}{2});(1,1,\frac{1}{2})\big)$. In particular, the period over $\Sigma_2(t)$ is holomorphic at $t=0$ and given by
\beq
\label{period_lemmaK3pure_18}
  \omega  = (2 \pi  i)^2 \,  \hpg43{\frac{\mu}{2}, \frac{1-\mu}{2}, \frac{1+\mu}{2},  1-\frac{\mu}{2}}{ 1, 1, \frac{1}{2} }{t^2} \,.
\eeq
\end{corollary}
\begin{proof}
We apply Proposition~\ref{prop:period_integral_general_2} to the period integral $\omega(t)= {}_2F_1\big( \mu , 1-\mu; 1  \big| \, t \big)$ and the twisted families in Equation~(\ref{lemmaK3mixed_eq}), and use the identity 
\beq
\label{UneasyTwist}
  {}_1F_0\!\left(\left. \frac{1}{2} \right| t^2 \right)  \star  \hpg21{\mu, 1-\mu}{1}{t}  =   \hpg43{\frac{\mu}{2}, \frac{1-\mu}{2}, \frac{1+\mu}{2},  1-\frac{\mu}{2}}{ 1, 1, \frac{1}{2} }{t^2}  \,.
\eeq
\end{proof}
 \section{Elliptic fibrations on the mirror families}
 \label{sec:dwork}
Non-trivial generalized functional invariants can be used to analyze elliptic fibrations on the mirror families obtained from the Dwork pencil, i.e., the one-parameter family of deformed Fermat hypersurfaces in $\mathbb{P}^{n}=\mathbb{P}(X_0,\dots,X_n)$ given by
\beq
\label{Fermat}
 X_0^{n+1} + X_1^{n+1} + \dots + X_{n}^{n+1} + (n+1) \, \lambda \, X_0 X_1 \cdots X_{n} = 0 \,.
\eeq
For each integer $n\in \mathbb{N}$,  Equation~(\ref{Fermat}) constitutes a family of $(n-1)$-dimensional Calabi-Yau hypersurfaces $X^{(n-1)}_{\lambda}$. For $n=4$ Equation~(\ref{Fermat}) is the quintic family of Candelas et al.~\cite{MR1101784}. For the family~(\ref{Fermat}) the discrete group of symmetries for the Greene-Plesser orbifolding procedure is easily identified: it is generated by the action $(X_0,X_j) \mapsto (\zeta_{n+1}^n X_0, \zeta_{n+1} X_j)$ for $1 \le j \le n$ with $\zeta_{n+1}=\exp{(\frac{2\pi i}{n+1})}$.  In virtue of the fact that the product of all generators multiplies the homogeneous coordinates by a common phase,  the symmetry group is $G_{n-1}=(\mathbb{Z}/(n+1)\, \mathbb{Z})^{n-1}$. The new affine variables
\begin{small}\begin{gather*}
 t=\frac{(-1)^{n+1}}{\lambda^{n+1}} \,, \;  x_1 = \frac{X_1^n}{(n+1)\, X_0 \cdot X_2 \cdots X_n \, \lambda}  \,, \;
 x_2 =\frac{X_2^n}{(n+1)\, X_0 \cdot X_1 \cdot X_3 \cdots X_n \, \lambda} ,  \; \dots \; ,
 \end{gather*}\end{small}
are invariant under the action of $G_{n-1}$. Hence, they descend to coordinates on the quotient $X^{(n-1)}_{\lambda}/G_{n-1}$. A second family of hypersurfaces $Y^{(n-1)}_{t}$ is then defined in terms of the new variables $x_1, \dots, x_n$ by the remaining relation between those, namely the equation
\beq
\label{mirror_family}
 f_n(x_1,\dots, x_n, t) = x_1 \cdots x_n \, \Big( x_1 + \dots + x_n + 1 \Big) +  \frac{(-1)^{n+1} \, t}{(n+1)^{n+1}} =0 \,.
\eeq
\par
It was proved in ~\cite{MR1416334} that the family of special Calabi-Yau hypersurfaces $Y^{(n-1)}_{t}$ of degree $(n+1)$ in $\mathbb{P}^n$ obtained from Equation~(\ref{mirror_family}) is in fact the mirror family of a general hypersurface $\mathbb{P}^n$ of degree $(n+1)$ and co-dimension one in $\mathbb{P}^n$. The subspace of the cohomology $H^{n-1}(X^{(n-1)}_{\lambda},\mathbb{Q})$ invariant under the obvious action of $G_{n-1}$ or, equivalently, the cohomology $H^{n-1}(Y^{(n-1)}_{t},\mathbb{Q})$ has dimension $n$ and the Hodge numbers $(1, \dots, 1)$. We have the following:
\begin{lemma}
\label{MirrorRecursive}
For every $n\ge 2$ the family of hypersurfaces $Y^{(n-1)}_{t}$ given by Equation~(\ref{mirror_family}) is  a fibration over $\mathbb{P}^1$ by hypersurfaces $Y^{(n-2)}_{\tilde{t}}$ where $x_n$ is the affine base coordinate, and 
 \beq
\label{rational_function}
 t = -\frac{n^n}{(n+1)^{n+1}  x_n \, (x_n+1)^n} \,  \tilde t\,.
\eeq 
\end{lemma}
\begin{proof}
For each $x_n \not= 0, -1$  substituting $\tilde x_i= x_i/(x_n+1)$ for $1\le i \le n-1$ and $\tilde t = -n^n \, t/((n+1)^{n+1}  x_n \, (x_n+1)^n)$
defines a fibration of the hypersurface~(\ref{mirror_family}) by $f_{n-1}(\tilde x_1, \dots, \tilde x_{n-1} \tilde t)=0$ since
\beq
\label{fibration_mirrors}
 f_n(x_1,\dots, x_n, t) = x_n \, (x_n+1)^n \, f_{n-1}(\tilde x_1, \dots, \tilde x_{n-1}, \tilde t\,) =0 \,.
\eeq
\end{proof}
The rational function on the right hand side of Equation~(\ref{rational_function}) relating $t$ to $\tilde{t}$ has precisely the characteristic form of a generalized functional invariant~(\ref{eqn:gfi}) with $(i,j)=(n,1)$. The unique holomorphic $(n-1)$-form on $Y^{(n-1)}_{t}$ is given by
\beq
\label{eqn:holm_form}
 \eta^{(n-1)}_t = \dfrac{dx_2 \wedge dx_3 \wedge \dots \wedge dx_n}{\partial_{x_1} f_n( x_1,\dots, x_n, t)} \,.
\eeq
One defines an $(n-1)$-cycle $\Sigma_{(n-1)}$ on $Y^{(n-1)}_{t}$ by requiring that the period integral of $\eta_t$ over $\Sigma_{(n-1)}$ emerges as residue in $x_1$ in the integral over the  torus $T^n  = S^1 \times \dots  \times S^1$.  The corresponding section of the period sheaf  is given by
\beq
\label{residue_integral}
\begin{split}
 \omega_{n-1}(t) = \oint_{\Sigma_{n-1}}  
 \dfrac{dx_2 \wedge \dots\wedge  dx_n}{\partial_{x_1} f_n( x_1,\dots, x_n, t)} \,.
\end{split}
\eeq
We have the following:
\begin{proposition}
\label{PeriodLemma}
For $n \ge 1$ and $|t|\le 1$, there is a transcendental $(n-1)$-cycle $\Sigma_{n-1}$ on $Y^{(n-1)}_{t}$
such that
\beq
\label{eqn:mirror_period}
   \omega_{n-1}(t)  =  \oint_{\Sigma_{n-1}} \eta_t^{(n-1)}= (2\pi i)^{n-1} \, \hpg{n}{n-1}{  \frac{1}{n+1} \quad  \dots \quad \frac{n}{n+1} }{ 1 \; \dots\; 1}{ t}  \,.
\eeq
The iterative relation~(\ref{rational_function}) induces an iterative relation between periods, namely
\beq
\label{iterative_period}
\begin{split}
   \omega_{n-1}(t) & =  (2\pi i) \, \hpg{n}{n-1}{  \frac{1}{n+1} \quad  \dots \quad \frac{n}{n+1} }{ \frac{1}{n} \; \dots\;  \frac{n-1}{n}}{ t}    \star  \omega_{n-2}(t) \quad
    \text{for $n \ge 2$}\,.
\end{split}   
\eeq
where the cycle $\Sigma_{n-1}$ is determined by $T^n(r_n) := \frac{n}{n+1} \cdot \big( T^{n-1}(r_{n-1}) \times S^1_{r_{n-1}}\big)$ where $r_n = 1 - \frac{n}{n+1}$ and $\frac{n}{n+1} \cdot \Big( T^{n-1}(r_{n-1}) \times S^1_{r_{n-1}}\big)$ means rescaling by $\frac{n}{n+1}$.
\end{proposition}
\begin{proof}
Rescaling and writing Equation~(\ref{mirror_family}) in the form $1- \tau \phi(w_1,\dots ,w_n)=0$ with $\tau = \frac{1}{n+1} t^{\frac{1}{n+1}}$ and $\phi = w_1 + \dots + w_n + \frac{1}{w_1 \cdots w_n}$ leads to the residue period 
\beqn
 \frac{\omega_{n-1}}{(2\pi i)^{n-1}} = \frac{1}{(2\pi i)^n} \oint_{T^n} \dfrac{ \frac{dw_1}{w_1} \wedge \dots  \wedge \frac{dw_n}{w_n}}{1 - \tau \,  \phi(w_1,\dots ,w_n)}
 = \sum_{l \ge 0} [\phi^l]_0 \, \tau^l \,,
\eeqn
where $[\phi^l]_0$ for the constant term in $\phi^l$. Using Equation~(\ref{eqn:GaussMultiplication}) we obtain the series expansion of the hypergeometric function in Equation~(\ref{eqn:mirror_period}).
\par For $x_n \not= 0, -1$ the coordinate transformation $\tilde x_i= x_i/(x_n+1)$ for $1\le i \le n-1$ and $\tilde t = -n^n \, t/((n+1)^{n+1}  x_n \, (x_n+1)^n)$
yields
\beq
  \eta_t^{(n-1)} = \tilde{\eta}_{\tilde{t}}^{(n-2)} \wedge \frac{dx_n}{x_n(x_n+1)} \,.
\eeq
For any $x_n$ on $S^1_{r_n}$ with $r_n = 1 - \frac{n}{n+1}$ write $1/(x_n+1)=R \, e^{i\varphi}$ with $\frac{n+1}{n+2} \le R\le \frac{n+1}{n}$,  the transformation $\tilde x_i= x_i/(x_n+1)$ maps the circle $\tilde x_i = r_{n-1} e^{it}$ to the circle $x_i = R\, r_{n-1}e^{i(t+\varphi)}$ with $0<R \, r_{n-1} \le \frac{n+1}{n^2} <1$ as $n \ge 2$.  We obtain
\begin{equation}
 \underbrace{\int \hspace*{-0.2cm}\dots \hspace*{-0.2cm}\int}_{T^n(r_n)}
    \dfrac{d x_1 \wedge \dots \wedge dx_{n}}{f_{n}( x_1,\dots, x_{n},t)} 
    = \oint_{|x_n|=\frac{1}{2}} \frac{dx_n}{x_n(x_n+1)} \, \underbrace{\int \hspace*{-0.2cm}\dots \hspace*{-0.2cm}\int}_{T^{n-1}(r_{n-1})}
    \dfrac{d \tilde{x}_1 \wedge \dots \wedge d\tilde{x}_{n-1}}{f_{n-1}( \tilde{x}_1,\dots,\tilde{x}_{n-1},\tilde{t})} \,.
\end{equation}
Using Proposition~\ref{prop:period_integral_general} Equation~(\ref{iterative_period}) follows.
\end{proof}
\subsection{Mirror family of pairs of points}
For $n=1$ the family $Y^{(0)}_{t}$ is a family of pairs of points in $\mathbb{P}^1$ given in an affine chart by the equation
\beq
\label{quadratic}
  x_1 \, \Big( x_1 + 1 \Big) + \frac{t}{4} =0
\eeq
with $t\in \mathbb{P} \backslash \lbrace 0, 1, \infty\rbrace$. For $n=1$ the deformed quadratic Fermat pencil $X_\lambda$ and $Y_t$ are equivalent. That is, the family in Equation~(\ref{quadratic}) satisfies, $X^{(0)}_{\lambda}\cong Y^{(0)}_{t}$ with $t=\lambda^2$. Moreover, if we set $x_1=(y_0-1)/2$ in Equation~(\ref{quadratic}), we obtain Equation~(\ref{legendre0}), that is, precisely the universal starting point of our twist construction.
\subsection{Mirror family of elliptic curves}
For $n=2$ the family $Y^{(1)}_t$ is equivalent to the elliptic modular surface over the rational 
modular curve for $\Gamma_0(3)$. In fact, using the birational transformation
\beq
\label{hesse_transfo}
 x_1={\frac {4\, t}{3\,(2\,X+3)}}\;, \qquad x_2=\frac {i\sqrt {2} \, Y-4\,t}{6\, (2\,X+3)} -\frac{1}{2} 
\eeq
in Equation~(\ref{mirror_family}), we recover the Weierstrass normal form given by
\beq
\label{eqn:mfn_1}
Y^2 = 4 X^3 - \underbrace{(27 - 24 \, t)}_{g_2(t)} \, X - \underbrace{(27-36 \, t + 8 \, t^2)}_{g_3(t)} \,.
\eeq
Equation~(\ref{eqn:mfn_1}) is the Weierstrass model for $X_{431}$  obtained in Lemma~\ref{lemmaECmixed}. Corollary~\ref{curve_reduction} proves that the period integral of $dX/Y$ over a suitable family of one-cycles $\Sigma_1(t)$ equals the hypergeometric function
\beq
\label{rk2hgf_mirror}
 \frac{\omega(t)}{2 \pi i} = \hpg21{ \mu, 1-\mu}{1}{t} 
\eeq
with $\mu=\frac{1}{3}$. In other words, the extremal family of elliptic curves $X_{431}$ is obtained from the family of pairs of points in Equation~(\ref{legendre0}) using our twist construction with generalized functional invariant $(i,j,\alpha)=(2,1,1)$ as proved in Lemma~\ref{lemmaECmixed}. On the level of periods, this fact manifests as application of the cancellation formula~(\ref{eqn:cancellation}) in the Hadamard product
\beq
\hpg21{ \frac{1}{3}, \frac{2}{3} }{1}{t}  =\hpg21{ \frac{1}{3}, \frac{2}{3} }{\frac{1}{2}}{t} \star
{}_1F_0\left(\left. \frac{1}{2} \right| t \right)\,.
\eeq
Since the transformation (\ref{hesse_transfo}) maps the holomorphic one-form $3\sqrt{2}idX/Y$ to $\eta_t=dx_2/f_{2,x_1}$ in Equation~(\ref{eqn:holm_form}), the period~(\ref{rk2hgf_mirror}) is the period for the mirror cubic family. 
\subsection{Mirror families of K3 surfaces}
For $n=3$ the family $Y^{(2)}_t$ is equivalent to a family of minimal Weierstrass models given by the equation
\beq
\label{QuarticMixedTwist}
 Y^2 = 4 X^3 - \underbrace{g_2\left( -\frac{3^3 \, t}{4^4 \, u^3(u+1)} \right)  \Big(u  (u+1) \Big)^4}_{= \,g_2(t,u)} \, X - \underbrace{g_3\left( -\frac{3^3 t}{4^4 \,u^3(u+1)} \right)  \Big(u  (u+1) \Big)^6}_{= \,g_3(t,u)}, 
\eeq
where we choose for $g_2(t)$ and $g_3(t)$ the Weierstrass coefficients in Equation~(\ref{eqn:mfn_1}). This is seen by applying the birational transformation
\beq
\label{quartic_transfo}
\begin{split}
 x_1 & = - \frac{9 \, (u+1) \, t}{64 \, \left(3\, u^4+6 \, u^3+3\, u^2+2\, X\right)} \,,\\
 x_2 & = \frac { -64\,i \, \sqrt {2}\, Y+9\, \left( 64\, u^5+128\, u^4+64 \, u^3+ \left( 3 t+\frac {128}{3} \right) \, u+t \right)  \left( u+1 \right)}{  1152 \left(u^4 + 2 \, u^3+ u^4+\frac{2}{3} \, u \right) \left( u+1 \right) } \,,\\[0.2em]
x_3 & = -\,(u+1)\,,
\end{split}
\eeq
in Equation~(\ref{mirror_family}). We have the following:
\begin{lemma}
Equation~(\ref{QuarticMixedTwist}) defines a family of $M_2$-polarized K3 surfaces.
\end{lemma}
\begin{proof}
Equation~(\ref{QuarticMixedTwist}) defines a family of Jacobian elliptic K3 surfaces of Picard rank 19 with a singular fiber of  Kodaira-type $IV^*$ over $u=0$, a singular fiber of Kodaira-type $I_{12}$ over $u=\infty$, and four fibers of Kodaira-type $I_1$. The Mordell-Weil group is pure three-torsion generated by the sections  $(X,Y)=(-3/2\,{u}^{2} \left( u+1 \right) ^{2}, \, \pm {\frac {27\,i}{128}}\sqrt {2}t{u}^{2})$.  It follows that the determinant of the discriminant group equals $3 \cdot 12/3^2=2^2$.
\end{proof}
\par In other words, the  family of Jacobian elliptic K3 surfaces  in Equation~(\ref{QuarticMixedTwist}) is the twisted family with generalized functional invariant $(i,j,\alpha)=(3,1,1)$ of the elliptic curves in Equation~(\ref{eqn:mfn_1}). Application of Proposition~\ref{prop:period_integral_general} together with 
the cancellation formula~(\ref{eqn:cancellation}) proves that the period integral of the holomorphic two-form $du\wedge dX/Y$ over a suitable two-cycle $\Sigma_2(t)$ equals the hypergeometric function
 \beq
\label{SFId}
 \begin{split}
  \frac{\omega(t)}{(2\pi i)^2} = \hpg32{\frac{1}{4}, \frac{2}{4}, \frac{3}{4}}{1, 1}{t} = \hpg32{\frac{1}{4}, \frac{2}{4}, \frac{3}{4}}{\frac{1}{3}, \frac{2}{3}}{t} 
 \star\hpg21{ \frac{1}{3}, \frac{2}{3} }{1}{t} \,.
 \end{split}
\eeq
Since the transformation~(\ref{quartic_transfo}) maps the holomorphic two-form $3\sqrt{2} i du \wedge dX/Y$ to $\eta_t=dx_2\wedge dx_3/f_{3,x_1}$  in Equation~(\ref{eqn:holm_form}), the period in Equation~(\ref{SFId}) is the period for the mirror quartic family. We make the following:
\begin{remark}
Equation~(\ref{QuarticMixedTwist}) defines a family of $M_2$-lattice polarized K3 surfaces $Y$ with transcendental lattice $T(Y) =H \oplus \langle 4 \rangle$. Following Dolgachev~\cite{MR1420220} its mirror partner $Y^{\vee}$  is the family of generic quartic surfaces in $\mathbb{P}^3$ with $\operatorname{NS}(Y^{\vee})=\langle 4 \rangle$ since $T(Y) = H \oplus \operatorname{NS}(Y^{\vee})$. Equivalently, it was proved in~\cite{MR1877764} that the mirror quartic is the family of the Calabi-Yau varieties arising from the polytope $P_0^*$ in dimension $3$.  The family $X$ is the family of the Calabi-Yau varieties arising from the reflexive polytope $P_0$ and is the family of generic quartic surfaces in $\mathbb{P}^3$.
\end{remark}
\subsection{Mirror families of Calabi-Yau threefolds}
For $n=4$ the family $Y^{(3)}_t$ is equivalent to the family of minimal Weierstrass models given by the equation
\beq
\begin{split}
\label{QuinticMixedTwist}
 Y^2 = 4 \, X^3 & - \underbrace{g_2\left( \frac{3^3 \, t}{5^5 \, u(u+1) \, v^3 \, (v+1)^2} \right) \, \Big(u \, (u+1) \, v\, (v+1)\Big)^4}_{=: \, g_2(t, u, v)} \, X\\
 &  - \underbrace{g_3\left( \frac{3^3 \, t}{5^5\,u \, (u+1) \, v^3 \, (v+1)^2} \right) \, \Big(u \, (u+1)\, v\, (v+1) \Big)^6}_{=: \, g_3(t, u, v)}, \\[-0.5em]
 \end{split}
\eeq
where we choose for $g_2(t)$ and $g_3(t)$ the Weierstrass coefficients in Equation~(\ref{eqn:mfn_1}). This is seen by applying the birational transformation
\beq
\label{QuinticTransfo}
\begin{split}
x_1 &= \scalebox{1}{$\frac{ 36 \, t \, u \left( u+1 \right) }{9375\,{v}^{2}
 \left( v+1 \right) ^{2}{u}^{4}+18750\,{v}^{2} \left( v+1 \right) ^{2}
{u}^{3}+9375\,{v}^{2} \left( v+1 \right) ^{2}{u}^{2}+6250\,X}$} \;,\\
x_2 & = \scalebox{1}{$\frac {-3125\,i\sqrt {2}Y+28125
\, \left( {v}^{3} \left( v+1 \right) ^{2}{u}^{4}+2\,{v}^{3} \left( v+1
 \right) ^{2}{u}^{3}+ \left( {v}^{5}+2\,{v}^{4}+{v}^{3}-{\frac {12\,t}{3125}} \right) {u}^{2}-{\frac {12\,ut}{3125}}+\frac{2}{3}\,v\,X \right)  \left( u+1 \right) u \left( v+1 \right)}{ \left( 56250\,{v}^{2} \left( v+1 \right) ^{2}{u}^{4}+112500\,{v}^{2} \left( v+1 \right) ^{2}{u}^{3}+56250\,{v}^{2}
 \left( v+1 \right) ^{2}{u}^{2}+37500\,X \right)  \left( u+1
 \right) u \left( v+1 \right) }$}\, , \\[0.2em]
 x_3 & =  u \, (v+1) \, , \qquad 
 x_4  = - (u+1) \, (v+1) \,,
\end{split}
\eeq
in Equation~(\ref{mirror_family}). Equation~(\ref{QuinticMixedTwist}) defines a family of minimal Weierstrass models over the two-dimensional base $\mathbb{P}^1 \times \mathbb{P}^1$ with affine coordinates  $u$ and $v$. Restricted to any generic $v$-slice we obtain a family of Jacobian elliptic K3 surfaces with $M_3$-polarization.
\par By inspection, the  family  in Equation~(\ref{QuinticMixedTwist})  is obtained from the family of elliptic curves in Equation~(\ref{eqn:mfn_1}) by applying a sequence of twists, with generalized functional invariant  $(i,j,\alpha)=(1,1,1)$ first, and $(i,j,\alpha)=(3,2,1)$ second. Application of Proposition~\ref{prop:period_integral_general} and  the cancellation formula~(\ref{eqn:cancellation}) proves that the period integral of the holomorphic three-form $dv\wedge du\wedge dX/Y$  over a suitable three-cycle $\Sigma_3(t)$ equals the hypergeometric function
\beq
\label{SFId_5}
   \frac{\omega(t)}{(2\pi i)^3} =  \hpg43{ \frac{1}{5},  \, \frac{2}{5},  \, \frac{3}{5}, \, \frac{4}{5}}{1, \, 1, \, 1}{t} =
\hpg43{ \frac{1}{5},  \, \frac{2}{5},  \, \frac{3}{5}, \, \frac{4}{5}}{\frac{1}{3},  \, \frac{2}{3},\, \frac{1}{2} }{t} \star {}_1F_0\!\left(\left. \frac{1}{2} \right| t \right) \star 
\hpg21{ \frac{1}{3}, \frac{2}{3} }{1}{t} \,.
\eeq
Since the transformation~(\ref{QuinticTransfo}) maps the holomorphic two-form $3\sqrt{2} i dv \wedge du \wedge dX/Y$ to $\eta_t=dx_2\wedge dx_3 \wedge dx_4/f_{4,x_1}$ in Equation~(\ref{eqn:holm_form}), the period in Equation~(\ref{SFId_5}) is the period for the mirror quintic family. The period in Equation~(\ref{SFId_5}) is annihilated by the fourth-order Picard-Fuchs operator  $L^{(4)}_t\big((\frac{1}{5}, \dots, \frac{4}{5});(1, \dots, 1)\big)$. The Picard-Fuchs operator is one of the 14 original Calabi-Yau operators mentioned in the introduction and labelled ``$(1)$" in the AESZ database~\cite{Almkvist:aa}. 
\begin{remark}
It was shown in~\cite{MR1101784} that the family of Calabi-Yau threefolds $Y^{(3)}_t$ has a general fiber with Hodge numbers 
$h^{2,1}(Y^{(3)}_t)=1$ and $h^{1,1}(Y^{(3)}_t)=101$. Following~\cite{MR1101784}  its mirror $Y^{\vee}$  is the general family 
of quintic surfaces in $\mathbb{P}^4$ with Hodge numbers $h^{1,1}(Y^{\vee})=1$ and $h^{2,1}(Y^{\vee})=101$.
\end{remark}
\section{Combining twists and base transformations}
\label{sec:base_transformations}
In this section, we apply linear and quadratic transformations to the rational parameter space of the twisted families of elliptic curves, K3 surfaces, and Calabi-Yau threefolds already constructed.
\subsection{Transformations of extremal families of elliptic curves}
We apply the linear transformation $t  \mapsto \frac{t}{t-1}$ to the rational deformation space of any extremal family of elliptic curves in Lemma~\ref{lemmaECmixed} to obtain the Weierstrass models
\beq
\label{transfo_ref}
 Y^2 = 4X^3 - \underbrace{g_2\left(\frac{t}{t-1}\right) (1-t)^4}_{=: \, \tilde{g}_2(t)}   X -  \underbrace{g_3\left(\frac{t}{t-1}\right)  (1-t)^6}_{=: \, \tilde{g}_3(t)} \,,
\eeq
where $g_2(t)$ and $g_3(t)$ are given in Table~\ref{tab:3ExtRatHg}. The transformation produces isomorphic families of elliptic curves that we denote by $\tilde{X}_{141}$, $\tilde{X}_{431}$, $\tilde{X}_{321}$, and $\tilde{X}_{211}$. They have the same number/type of singular fibers and Mordell-Weil groups as the families $X_{141}$, $X_{431}$, $X_{321}$, and $X_{211}$, but with the singular fibers over $t=1$ and $t=\infty$ interchanged. The families $X_{141}$, $X_{431}$, $X_{321}$, or $X_{211}$ (and hence $\tilde{X}_{141}$, $\tilde{X}_{431}$, $\tilde{X}_{321}$, or $\tilde{X}_{211}$) are labelled by $\mu=\frac{1}{2}, \frac{1}{3}, \frac{1}{4}$, or $\mu=\frac{1}{6}$ (and by $\tilde{\mu}=\frac{1}{2}, \frac{1}{3}, \frac{1}{4}$, or $\tilde{\mu}=\frac{1}{6}$).
Let $\tilde{\Sigma}_1(t)$ be the family of one-cycles obtained from $\Sigma_1(\tilde{t})$ in Lemma~\ref{lemma_cycle} with $\tilde{t}=\frac{t}{t-1}$ by rescaling $(X,Y)\mapsto ((1-t)^2 X, (1-t)^3 Y)$. For $t \in \mathbb{C}$ with $|\tilde{t}|<1/2$, $\tilde{\Sigma}_1(t)$ defines a family of A-cycles on $\tilde{X}_{141}$, $\tilde{X}_{431}$, $\tilde{X}_{321}$, and $\tilde{X}_{211}$ in the neighborhood of $t=0$. We have the following:
\begin{corollary}
\label{lemma2}
For the families of elliptic curves $\tilde{X}_{141}$, $\tilde{X}_{431}$, $\tilde{X}_{321}$, and $\tilde{X}_{211}$ in Equation~(\ref{transfo_ref}), period integrals of $dX/Y$ are annihilated by the Picard-Fuchs operator 
 \beq
\label{L2tilde}
  \tilde{L}^{(2)}_t = \theta^2 - t \, \big( 2\,\theta^{2}+ 2\,\theta + \tilde{\mu}^2 - \tilde{\mu} +1\big) + t^2 \, \big(\theta+1\big)^2 \,.
\eeq
In particular, the period over $\tilde{\Sigma}_1(t)$ is holomorphic at $t=0$ and  given by
\beq
\label{rk2hgf_tilde}
 \tilde{\omega} =  \frac{2\pi i}{(1-t)^{1-\tilde{\mu}}} \; \hpg21{ \tilde{\mu}, \tilde{\mu}}{1}{\frac{t}{t-1}} 
 \eeq
 with $\tilde{\mu} \in \{ \frac{1}{2}, \frac{1}{3}, \frac{1}{4}, \frac{1}{6} \}$.
 \end{corollary}
\begin{proof}
The proof follows from the following well-known identity for the Gauss hypergeometric function, namely
\beqn
  \frac{2\pi i}{1-t} \; \hpg21{ \tilde{\mu}, 1-\tilde{\mu}}{1}{\frac{t}{t-1}} =  \frac{2\pi i}{(1-t)^{1-\tilde{\mu}}} \; \hpg21{ \tilde{\mu}, \tilde{\mu}}{1}{\frac{t}{t-1}} \,. 
\eeqn
\end{proof}
\begin{remark}
\label{rem:double_transfo_rational}
For the above families of elliptic curves, twisted families can be constructed as in Section~\ref{ssec:Mn-polarized}. The twisted families with generalized functional invariant $(i,j,\alpha)=(1,1,1)$ of $\tilde{X}_{141}$, $\tilde{X}_{431}$, $\tilde{X}_{321}$, and $\tilde{X}_{211}$ are families of $M_n$-lattice polarized K3 surfaces. A continuously varying family of closed two-cycles $\tilde{\Sigma}_2(t)$ can be constructed in each case such that the period over $\tilde{\Sigma}_2(t)$ is given by
\beq
 \tilde{\omega}=(2\pi i)^2 \;  {}_1F_0\!\left(\left. \frac{1}{2} \right| t \right) \star \left( \frac{1}{1-t} \; \hpg21{ \tilde{\mu}, 1-\tilde{\mu}}{1}{\frac{t}{t-1}} \right)
\eeq  
 with $\tilde{\mu} \in \{ \frac{1}{2}, \frac{1}{3}, \frac{1}{4}, \frac{1}{6} \}$.
\end{remark}
\subsection{Transformations of lattice polarized K3 surfaces}
\label{MMM-polarized_families}
We apply a linear or quadratic transformations, denoted by $t \mapsto f_k(t)$ with $k=1,\dots,5$, to the rational parameter space of the families of $M_n$-lattice polarized K3 surfaces in Lemma~\ref{lemmaECmixed} to obtain new Weierstrass models given by
\beq
\label{transfo_ref_K3}
 Y^2 = 4 X^3 - \underbrace{g_2\big(f_k(t), u\big) \, h_k(t)^2}_{=: \,\tilde{g}^{(k)}_2(t,u)}  X - \underbrace{g_3\big(f_k(t), u\big) \,  h_k(t)^3}_{=: \,\tilde{g}^{(k)}_3(t,u)} \,,
\eeq
where $g_2(t, u)$ and $g_3(t, u)$ are given in Lemma~\ref{lemmaK3mixed_eq}, and the polynomials $f_k(t)$ and $h_k(t)$ are given in Table~\ref{tab:MnK3_RationalTransfo}. It is readily checked that $\tilde{g}^{(k)}_2(t, u)$ and $\tilde{g}^{(k)}_3(t, u)$ define families of minimal Weierstrass model.  By construction, these new families remain families of $M_n$-lattice polarized K3 surfaces. 
\begin{table}[H]
\scalebox{\Scale}{
\begin{tabular}{|c|cc|c||l|l|} \hline
 &&&&& \\[-0.9em]
 $k$ & $f_k(t)$ & $h_k(t)$ & $(p, q)$ & $ \tilde{\omega}/(2\pi i)^2$ & $\tilde{L}^{(3)}_t$ \\[0.1em]
 \hline
 \hline
 &&&&& \\[-0.8em]
 $1$ & $t$ & $1-t$ & $\left( \frac{\mu}{2}, \frac{1-\mu}{2} \right)$ & $\phantom{=}  \frac{1}{\sqrt{1-t}}\, {}_3F_2\left( \mu, 1-\mu, \frac{1}{2}; 1, 1  \Big| \,t\right)$ & $ \theta^3 - t \, \big(\theta + \frac{1}{2} \big) \, \big( 2 \,\theta^2 + 2\, \theta - \mu^2 + \mu + 1\big)$ \\[0.4em]
 &&&& $= \left( \frac{1}{(1-t)^{\frac{1-\mu/2-(1-\mu)/2}{2}}} \, {}_2F_1\left(\left. \frac{\mu}{2}, \frac{1-\mu}{2}; 1 \right| \ t\right) \right)^2$ & \hfill $+ \; t^2 \, \big(\theta + 1\big) \,  \big(\theta + \mu + \frac{1}{2} \big) \, \big(\theta - \mu + \frac{3}{2} \big)$\\[0.2em]
  &&&&& \\[-0.8em]
 \hline
  &&&&& \\[-0.8em]
  $2$ & $\frac{t}{t-1}$ & $1-t$ & $\left( \frac{\mu}{2}, \frac{1+\mu}{2} \right)$  & $\phantom{=}  \frac{1}{\sqrt{1-t}} \,  {}_3F_2\left( \mu, 1-\mu, \frac{1}{2}; 1, 1  \Big| \,  \frac{t}{t-1} \right)$ & $ \theta^3 - t \, \big(\theta + \frac{1}{2} \big) \, \big( 2 \,\theta^2 + 2\, \theta + \mu^2 - \mu + 1\big) $ \\[0.4em]
  &&&& $= \left( \frac{1}{(1-t)^{\frac{1-\mu/2-(1+\mu)/2}{2}}} \, {}_2F_1\left(\left. \frac{\mu}{2}, \frac{1+\mu}{2}; 1 \right|\,t\right) \right)^2$ &  \hfill  $+ \; t^2 \, \big(\theta + 1\big) \,  \big(\theta + \frac{1}{2} \big) \, \big(\theta + \frac{3}{2} \big)  $ \\[0.2em]
   &&&&& \\[-0.8em]
 \hline
&&&&&  \\[-0.8em]
  $3$ & $4t(1-t)$ & $1$ & $\left( \mu, 1-\mu\right)$  & $\phantom{=}    {}_3F_2\left( \mu, 1-\mu, \frac{1}{2}; 1, 1  \Big| \, 4 \, t \, (1-t) \right)$ & $ \theta^3 - 2 \, t \, \big(\theta + \frac{1}{2} \big) \, \big( \theta^2 +  \theta  - 2\, \mu^2 - 2\, \mu \big)$ \\[0.4em]
  &&&& $= \left( \frac{1}{(1-t)^{\frac{1-\mu-(1-\mu)}{2}}} \,  {}_2F_1\left(\left. \mu, 1-\mu; 1 \right| \, t\right) \right)^2$&  \hfill  $ + \; t^2 \, \big(\theta + 1\big) \,  \big(\theta +  2\, \mu\big) \, \big(\theta + 2\, (1-\mu) \big) $\\[0.2em]
   &&&&& \\[-0.8em]
 \hline
 &&&&& \\[-0.8em]
  $4$ & $ \frac{t^2}{4(t-1)}$ & $1-t$ & $\left( \mu, \frac{1}{2} \right)$  & $\phantom{=}  \frac{1}{\sqrt{1-t}}\,  {}_3F_2\left( \mu, 1-\mu, \frac{1}{2}; 1, 1  \Big| \frac{t^2}{4(t-1)} \right) $ & $ \theta^3 - t \, \big(\theta + \frac{1}{2} \big) \, \big( 2 \,\theta^2 + 2\, \theta + 1\big)  $ \\[0.4em]
  &&&& $= \left( \frac{1}{(1-t)^{\frac{1-\mu-1/2}{2}}} \,  {}_2F_1\left(\left. \mu, \frac{1}{2}; 1 \right|\, t\right) \right)^2$&  \hfill  $+ \; t^2 \, \big(\theta + 1\big) \,  \big(\theta + \mu + \frac{1}{2} \big) \, \big(\theta - \mu + \frac{3}{2} \big)$\\[0.2em]
   &&&&& \\[-0.8em]
 \hline
  &&&&& \\[-0.8em]
  $5$ & $ - \frac{4 \, t}{(1-t)^2}$ & $(1-t)^2$ & $\left( \mu, \mu \right)$  &  $\phantom{=}  \frac{1}{1-t} \,  {}_3F_2\left( \mu, 1-\mu, \frac{1}{2}; 1, 1 \Big| \,  - \frac{4 \, t}{(1-t)^2} \right) $ & $  \theta^3 - 2 \, t \, \big(\theta + \frac{1}{2} \big) \, \big( \theta^2 +  \theta  + 2\, \mu^2 - 2\, \mu + 1\big) $ \\[0.4em]
&&&& $= \left( \frac{1}{(1-t)^{\frac{1-\mu-\mu}{2}}} \,  {}_2F_1\left(\left.\mu, \mu; 1 \right|\, t\right) \right)^2$&  \hfill  $+ \; t^2 \, \big(\theta + 1\big)^3   $\\[0.5em]
 \hline
\end{tabular}}
\caption{Rational transformations and periods of new family}\label{tab:MnK3_RationalTransfo}
\end{table}
\par Let $\tilde{\Sigma}_2(t)$ be the family of two-cycles obtained from the family $\Sigma_2(f_k(t))$ -- where $\Sigma_2(t)$ was given in Section~\ref{ssec:Mn-polarized} -- by rescaling $(u,X,Y)\mapsto (u, h_k(t) X, h_k(t)^{3/2} Y)$. For $t \in \mathbb{C}$ with $|f_k(t)|<1/2$, $\tilde{\Sigma}_2(t)$ defines a continuously varying family of two-cycles. We have the following:
\begin{corollary}
\label{cor:periodK3mixed_transform}
For the families of elliptic K3 surfaces in Equation~(\ref{transfo_ref_K3}), the period integrals of $du \wedge dX/Y$ are annihilated by the Picard-Fuchs operator 
\beq
\label{L3tilde}
\begin{split}
 \tilde{L}^{(3)}_t = \theta^3 &- t  \left( 2\,\theta+1 \right)  \left( \theta^{2}+\theta+2pq-p-q+1 \right)\\
 &+ t^2  \left( \theta+1 \right)  \left( \theta+1+q-p \right)  \left( \theta
+1-q+p \right).
\end{split}
\eeq
In particular, the period over $\tilde{\Sigma}_2(t)$ is holomorphic at $t=0$ and given by
\beq
\label{periodK3mixed_transform_eq}
 \tilde{\omega} =  (2\pi i)^2 \; \left( \frac{1}{(1-t)^{\frac{1-p - q}{2}}} \; \hpg21{p , q }{1}{t} \right)^2 \,,
\eeq
where $\mu$ and $(p, q)$ for $k=1,\dots,5$ are given in Table~\ref{tab:MnK3_RationalTransfo}. 
 \end{corollary}
 \begin{proof}
The construction is an application of the general construction in Section~\ref{ssec:twisting}.  The proof amounts to checking some classical and well-known hypergeometric function identities listed in Table~\ref{tab:MnK3_RationalTransfo}. The identities allow us to write each period as a symmetric square. 
\end{proof}
\subsection{Threefolds by combining twists and base transformations}
\label{ssec:3folds_hg_even}
To obtain families of elliptic Calabi-Yau threefolds, we start with a family of Jacobian elliptic K3 surfaces $X \to \mathbb{P}^1$, given as Weierstrass model
\beq
\label{eqn:K3_family}
 Y^2 = 4X^3-g_2(t,u)\, X-g_3(t,u)  \,.
\eeq
We will restrict ourselves to the cases where this K3 surface is chosen from Section~\ref{ssec:Mn-polarized}, Section~\ref{MMM-polarized_families}, or Remark~\ref{rem:double_transfo_rational}. Applying our twist construction, we obtain new Weierstrass models for twisted families with generalized functional invariant $(k, l,\beta)$ with $1 \le k, l  \le 6$, $\beta \in \{ \frac{1}{2}, 1\}$, and $c_{k l}=(-1)^k \, k^k \, l^l/(k+l)^{k+l}$ that are families of Gorenstein threefolds. We have the following:
\begin{lemma}
\label{lemmaCY3mixed}
For every family of elliptic K3 surfaces from Section~\ref{ssec:Mn-polarized}, Section~\ref{MMM-polarized_families}, or Remark~\ref{rem:double_transfo_rational}, the twisted family with generalized functional invariant $(k, l, \beta)$, given by the Weierstrass equation
\beq
\label{lemmaCYmixed_eq}
\begin{split}
 Y^2 = 4 X^3  - \underbrace{g_2\left(  \frac{c_{kl} t}{v^k (v+1)^l},u \right) \, v^4 (v+1)^{4\beta} }_{=: \, g_2(t,u,v)}  X  - \underbrace{g_3\left(  \frac{c_{kl} t}{v^k (v+1)^l},u  \right)  \, v^6  (v+1)^{6\beta} }_{=: \, g_3(t,u,v) }   \,,
 \end{split}
\eeq
defines a family (over $B$) of Jacobian elliptic Calabi-Yau threefolds over $\mathbb{P}^1 \times \mathbb{P}^1$. For K3 surfaces from Lemma~\ref{lemmaK3mixed}, we assumed $1 \le k \le 1/\mu$ and $1 \le l \le \beta/\mu$ with $\beta \in \lbrace \frac{1}{2}, 1 \rbrace$,  and for K3 surfaces from Section~\ref{MMM-polarized_families} or Remark~\ref{rem:double_transfo_rational} we set $(k, l, \beta)=(1,1,1)$.
\end{lemma}
\begin{proof}
The construction is an application of the general construction in Section~\ref{ssec:twisting}. 
\end{proof}
\par For each twisted family in Equation~(\ref{lemmaCYmixed_eq}), we define a family of closed three-cycles $\Sigma_3(t)$ as follows: For $t \in \mathbb{C}$ with $|t|<1/(2^{k+l+1}|c_{k l}|)$, we start with the circle $C=C_{1/2}(0)$, given by $|v|=\frac{1}{2}$ in the $v$-plane with counterclockwise orientation. For every $v \in C$, a two-cycle $\Sigma_2'(t, v)$ in the K3-fiber is obtained from $\Sigma_2(\frac{c_{k l} t}{v^k (v+1)^l})$, where $\Sigma_2(t)$ was defined in Section~\ref{ssec:Mn-polarized}, by rescaling $(u, X,Y) \to (u, v^2 (v+1)^{2\beta} X, v^3 (v+1)^{3\beta} Y)$. For $t\in \mathbb{C}$ with $|t|<1/(2^{i+j+1}|c_{i j}|)$, we obtain a continuously varying family of closed three-cycles as a warped product $\Sigma_3(t) = C \times_v \Sigma'_2(t, v)$. 
\subsubsection{Calabi-Yau operators of the hypergeometric case}
\label{sssec:hg}
Applying our twist construction to the elliptic K3 surfaces from Section~\ref{ssec:Mn-polarized}, we obtain the following:
\begin{corollary}
\label{hg-case}
For the twisted families in Lemma~\ref{lemmaCY3mixed} with generalized functional invariant $(k, l, \beta)$ of the $M_n$-lattice polarized K3 surfaces from Section~\ref{ssec:Mn-polarized}, the period integral~(\ref{eqn:period_new}) is annihilated by the Picard-Fuchs operator 
\beq
\label{L4mu}
 {}_1L_t^{(4)}(p, q) =\theta^4 - t \,  \big(\theta+p\big)  \big(\theta+q\big) \big(\theta+1-q\big) \big(\theta+1-p\big) \,.
\eeq
In particular, the period over $\Sigma_3(t)$ is holomorphic at $t=0$ and given by
\beq
\label{rk4hgf}
\omega= (2 \pi i)^3 \; \hpg43{p, q, 1-q, 1-p}{1, 1, 1}{t} \,.
\eeq
The values $(p, q)$ resulting from a twist with generalized functional invariant $(k, l, \beta)$ of a family $M_n$-lattice polarized K3 surface with $1 \le n \le 4$
are given in Table~\ref{tab:VHS}.  
\end{corollary}
\begin{proof}
The proof follows by applying Equation~(\ref{eq:period_new}) to the periods $\omega(t)$ computed in Corollary~\ref{cor:modular_K3_surfaces}. One then checks for which generalized functional invariants $(k, l, \beta)$ within the range provided by Lemma~\ref{lemmaCY3mixed}, the Hadamard product in Proposition~\ref{prop:period_integral_general} produces a hypergeometric function of type ${}_4F_3$ using Equation~(\ref{eqn:cancellation}).
\end{proof}
Twisted families of the Jacobian elliptic K3 surfaces from Section~\ref{ssec:M-polarized} can also be obtained from generalized functional invariants $(k, l, \beta)=(\frac{m}{2}, \frac{m}{2}, 1)$ in Equation~(\ref{lemmaCYmixed_eq}) where $m$ is an odd integer. In fact, if we set
\beq
 v= - \frac{1}{1+\tilde{v}^2} \,, \quad X = \frac{\tilde{v}^2 \tilde{X}}{(1+\tilde{v}^2)^4} \,,\quad Y = \frac{\tilde{v}^3 \tilde{Y}}{(1+\tilde{v}^2)^6} \,,
\eeq
we obtain $dv \wedge du \wedge dX/Y = 2 d\tilde{v} \wedge du \wedge d\tilde{X}/\tilde{Y}$, and Equation~(\ref{lemmaCYmixed_eq})  becomes the minimal and normal Weierstrass model given by
\beq
\label{lemmaCYmixed_eq_18}
\begin{split}
 \tilde{Y}^2 = 4 \tilde{X}^3  - \underbrace{g_2\left(  \frac{t (1+\tilde{v}^2)^m}{(2\tilde{v})^m},u \right) \, \tilde{v}^4 }_{=: \, \tilde{g}_2(t,u,\tilde{v})}  \tilde{X}  -  \underbrace{g_3\left(  \frac{t (1+\tilde{v}^2)^m}{(2\tilde{v})^m},u \right) \, \tilde{v}^6 }_{=: \, \tilde{g}_3(t,u,\tilde{v})}   \,.
 \end{split}
\eeq
We have the following:
\begin{lemma}
\label{lemmaCY3mixed_18}
For every family of elliptic K3 surfaces from Section~\ref{ssec:M-polarized}, the twisted family with generalized functional invariant $(k, l, \beta)=(\frac{m}{2}, \frac{m}{2}, 1)$, given by the Weierstrass equation~(\ref{lemmaCYmixed_eq_18}), defines a family of elliptic Calabi-Yau threefolds over $\mathbb{P}^1 \times \mathbb{P}^1$. Here, we set $m=1$, except for the family of $M'$-lattice polarized K3 surfaces in Lemma~\ref{lemmaK3mixed18_eq} where we have $m \in \{1,3\}$.
\end{lemma}
The construction of a continuously varying family of closed three-cycles $\tilde{\Sigma}_3(t)$ is analogous to the construction of $\Sigma_3(t)$ above. We also have the following:
\begin{corollary}
\label{hg-case_18}
For the twisted families in Lemma~\ref{lemmaCYmixed_eq_18} with generalized functional invariant  $(k, l, \beta)=(\frac{m}{2}, \frac{m}{2}, 1)$ of the elliptic K3 surfaces from Section~\ref{ssec:M-polarized}, the period integral of $d\tilde{v} \wedge du \wedge d\tilde{X}/\tilde{Y}$ is annihilated by the Picard-Fuchs operator 
\beq
\label{L4mu18}
 {}_1L_{t^2}^{(4)}(p, q) =\theta^4 - t^2 \,  \big(\theta+2 p\big) \, \big(\theta+2 q\big) \,\big(\theta+2-q\big) \,\big(\theta+2-p\big) \,.
\eeq
In particular, the period over $\tilde{\Sigma}_3(t)$ is holomorphic at $t=0$ and given by
\beq
\label{rk4hgf18}
\tilde{\omega} = (2 \pi i)^3 \; \hpg43{p, q, 1-q, 1-p}{1, 1, 1}{t^2}\,,
\eeq
for the values $(p, q)$ given in Table~\ref{tab:VHS}.  
\end{corollary}
\begin{proof}
We evaluate the period of the holomorphic three-from $d\tilde{v} \wedge du \wedge d\tilde{X}/\tilde{Y}$ over $\tilde{\Sigma}_3(t)$ by a residue computation. By construction of $\tilde{\Sigma}_3(t)$, it follows that for $|t|<1$ and $(\tilde{v},u, \tilde{X},\tilde{Y}) \in \tilde{\Sigma}_3(t)$ we have
\beqn
 \left| \dfrac{t \, (1+\tilde{v}^2)^m}{(2\tilde{v})^m} \, \left(1 +\frac{1}{2 \, u \, (u+1)}\right) \right|, \; \left| \dfrac{t \, (1+\tilde{v}^2)^m}{(2\tilde{v})^m} \right|  < 1 \,.
\eeqn
Using Corollary~\ref{cor:modular_K3_surfaces_18}, we obtain the period integral from the following computation
\beq
\begin{split}
 \frac{\tilde{\omega}}{(2\pi i)^3}= &=  \frac{1}{2\pi i} \,  \oint_{C_{1/2}(0)} \dfrac{d\tilde{v}}{\tilde{v}}  \; \hpg43{ \frac{\mu}{2}, \frac{1-\mu}{2}, \frac{1+\mu}{2}, 1-\frac{\mu}{2}}{1, 1, \frac{1}{2} }{ \frac{t^2 \, (1+\tilde{v}^2)^{2m}}{(2\tilde{v})^{2m}}}\\
  & =  \hpg{4+m}{3+m}{  \frac{\mu}{2}, \frac{1-\mu}{2},  \frac{1+\mu}{2}, 1-\frac{\mu}{2}, \frac{1}{2m}, \frac{3}{2m}, \dots,  \frac{2m-1}{2m}}{ 
  \frac{1}{2}, \frac{1}{m}, \frac{2}{m}, \dots, \frac{m}{m}, 1, 1}{ t^2  }\,.
 \end{split}
\eeq
We observe that for the given parameters $m, \mu$ there is a cancellation in the coefficients of the hypergeometric series, and we obtain Equation~(\ref{rk4hgf18}). For $m=1$, or $m=3, \mu=\frac{1}{3}$, the hypergeometric series reduce to
\beqn
  \hpg43{\frac{\mu}{2}, \frac{1-\mu}{2}, \frac{1+\mu}{2},  1-\frac{\mu}{2}}{ 1, 1, 1}{t^2} \,,  \quad \text{or} \quad \hpg43{\frac{1}{6}, \frac{1}{6}, \frac{5}{6},  \frac{5}{6}}{ 1, 1, 1}{t^2} \,.
\eeqn
\end{proof}
\begin{remark}
The Calabi-Yau operators~(\ref{L4mu}) obtained in Corollary~\ref{hg-case} and \ref{hg-case_18} for parameters $(p, q, q'=1-q, p'=1-p)$, with their classification number  in the AESZ database~\cite{Almkvist:aa}, are summarized in Table~\ref{tab:VHS}. The Calabi-Yau operators have degree one and are called Calabi-Yau operators of the hypergeometric case. In particular, Table~\ref{tab:VHS} includes the generalized functional invariants that were found in \cite{MR3431621} to construct threefolds fibered by $M_n$-polarized K3 surfaces using toric geometry.
\end{remark}
\begin{remark}
It was proved in \cite{MR2282973} that the regular singular points $t=0, 1, \infty$ of the Picard-Fuchs operator in Equation~(\ref{L4mu})
correspond to the conifold limit, large complex structure limit, and the orbifold point, respectively.  In particular, the monodromy around $t=\infty$ is maximally unipotent.
\end{remark}
\begin{table}
\scalebox{\Scale}{
\begin{tabular}{|c|c|c|cccc|cccc|cccc|} 
\hline
 &&&&&&&&&&&&&&\\[-0.9em]
 $\#$ & AESZ & ${(p, q, q', p')}$
& \multicolumn{4}{c}{twist of K3 with $\big(k, l,\beta=1\big)$} 
& \multicolumn{4}{|c|}{twist of K3 with $\big(k, l,\beta=\frac{1}{2}\big)$} & \multicolumn{4}{c|}{twist with  $\big(\frac{m}{2},\frac{m}{2},1\big)$} \\[0.4em]
\cline{4-15} 
 &&&&&&&&&&&&&&\\[-0.9em]
&  &   
 & $M_4$ 		& $M_3$ 			& $M_2$ 			& $M_1$ 			
 & $M_4$ 		& $M_3$			& $M_2$ 			& $M_1$ 
 & $\tilde{M}'$ 	& $M'$ 			& $\tilde{M}$ 		& $M$\\[0.4em]
\hline
\hline
&&&&&&&&&&&&&&\\[-0.9em]
$1$ & (3) & $\big(\frac{1}{2}, \frac{1}{2}, \frac{1}{2}, \frac{1}{2}\big)$	& $\big(1,1\big)$	
&						
&						&&&&&& &&&\\[0.4em]
\hline
&&&&&&&&&&&&&&\\[-0.9em]
$2$ & (5) & $\big(\frac{1}{3}, \frac{1}{2}, \frac{1}{2}, \frac{2}{3}\big)$ 	& $\big(2,1\big)$ 	& $\big(1,1\big)$		 
& 						&&&&&& &&&\\[0.4em]
\hline
&&&&&&&&&&&&&&\\[-0.9em]
$3$ & (4) & $\big(\frac{1}{3}, \frac{1}{3}, \frac{2}{3}, \frac{2}{3}\big)$	&				& $\big(2,1\big)$ 		 
& 						&&&&&&& &&\\[0.4em]
\hline
&&&&&&&&&&&&&&\\[-0.9em]
$4$ & (6) & $\big(\frac{1}{4}, \frac{1}{2}, \frac{1}{2}, \frac{3}{4}\big)$	& $\big(2,2\big)$	& $\big(3,1\big)$	& $\big(1,1\big)$  
& 													& ${\big(1,1\big)}$	&&&& 
&&&\\[0.4em]
\hline
&&&&&&&&&&&&&&\\[-0.9em]
$5$ & (11) & $\big(\frac{1}{4}, \frac{1}{3}, \frac{2}{3}, \frac{3}{4}\big)$ &				& $\big(2,2\big)$ 	& ${\big(2,1\big)}$	
&													&				& ${\big(1,1\big)}$	&&&&&& \\[0.4em]
\hline
&&&&&&&&&&&&&&\\[-0.9em]
$6$ & (10) & $\big(\frac{1}{4}, \frac{1}{4}, \frac{3}{4}, \frac{3}{4}\big)$ & 				&			   	& $\big(2,2\big)$	
&													&				&				& ${\big(1,1\big)}$	&
&${\big(1\big)}$&&&\\[0.4em]
\hline
&&&&&&&&&&&&&&\\[-0.9em]
$7$ & (14) & $\big(\frac{1}{6}, \frac{1}{2}, \frac{1}{2}, \frac{5}{6}\big)$ & 				& ${\big(3,3\big)}$	& 				& $\big(1,1\big)$ 
													& ${\big(2,1\big)}$ 	& 				& ${\big(1,2\big)}$ &  &&&&\\[0.4em]
\hline
&&&&&&&&&&&&&&\\[-0.9em]
$8$ & (8) & $\big(\frac{1}{6}, \frac{1}{3}, \frac{2}{3}, \frac{5}{6}\big)$ & 				& 				& $\big(4,2\big)$ 	& $\big(2,1\big)$ 
& & ${\big(2,1\big)}$ 
& & &&${\big(1\big)}$&&\\[0.4em]
\hline
&&&&&&&&&&&&&&\\[-0.9em]
$9$ & (12) & $\big(\frac{1}{6}, \frac{1}{4}, \frac{3}{4}, \frac{5}{6}\big)$ & 				& 				& 				& $\big(2,2\big)$ 
& & & ${\big(2,1\big)}$ 
& ${\big(1,1\big)}$ &&&& \\[0.4em]
\hline
&&&&&&&&&&&&&&\\[-0.9em]
$10$& (13) & $\big(\frac{1}{6}, \frac{1}{6}, \frac{5}{6}, \frac{5}{6}\big)$ & 				& 				& 				& 
& & & & ${\big(2,1\big)}$ 
&&$\big(3\big)$&&\\[0.4em]
\hline
&&&&&&&&&&&&&&\\[-0.9em]
$11$& (1) & $\big(\frac{1}{5}, \frac{2}{5}, \frac{3}{5}, \frac{4}{5}\big)$ & 				& $\big(3,2\big)$	& $\big(4,1\big)$	&
&& &&&&&&\\[0.4em]
\hline
&&&&&&&&&&&&&&\\[-0.9em]
$12$& (7) & $\big(\frac{1}{8}, \frac{3}{8}, \frac{5}{8}, \frac{7}{8}\big)$ & 				& 				& $\big(4,4\big)$	& 
&
& ${\big(3,1\big)}$ 
& ${\big(2,2\big)}$ & $\big(1,3\big)$ & & & ${\big(1\big)}$ &\\[0.4em]
\hline
&&&&&&&&&&&&&&\\[-0.9em]
$13$& (2) & $\big(\frac{1}{10}, \frac{3}{10}, \frac{7}{10}, \frac{9}{10}\big)$ & 			& 				& 				& 
&
& & ${\big(4,1\big)}$ 
& $\big(2,3\big)$ & &&&\\[0.4em]
\hline
&&&&&&&&&&&&&&\\[-0.9em]
$14$& (9) & $\big(\frac{1}{12}, \frac{5}{12}, \frac{7}{12}, \frac{11}{12}\big)$ & 			& 				& 				& 
&
& & ${\big(4,2\big)}$ 
&  & & & & $\big(1\big)$\\[0.4em]
\hline
\end{tabular}}
\caption{Twist parameters for operators ${}_1L_{t}^{(4)}(p, q)$ in the `hypergeometric case'}\label{tab:VHS}
\end{table}
\subsubsection{Calabi-Yau operators in the extra case}
Applying our twist construction to the elliptic K3 surfaces in Corollary~\ref{cor:periodK3mixed_transform}, we obtain the following:
\begin{corollary}
\label{cor:extra_case}
For the twisted families in Equation~(\ref{lemmaCYmixed_eq}) with generalized functional invariant $(k, l, \beta)=(1,1,1)$ of the elliptic K3 surfaces  in Corollary~\ref{cor:periodK3mixed_transform}, the period integral~(\ref{eqn:period_new}) is annihilated by the Picard-Fuchs operator 
\beq
\label{operator3tilde}
\begin{split}
  {}_2L^{(4)}_t(p, q) = \theta^4 & - 2 t  \Big( \theta+ \frac{1}{2} \Big)^2   \Big( \theta^2 + \theta + 2 p q - p - q + 1\Big) \\
  & + t^2 \Big( \theta+ \frac{1}{2} \Big) \Big( \theta+ \frac{3}{2} \Big) \Big( \theta +1 + p - q\Big) \, \Big( \theta+ 1 -p + q \Big) \,.
\end{split}
\eeq 
In particular, the period over $\Sigma_3(t)$ is holomorphic at $t=0$ and given by
\beq
\label{period3bb_tilde}
\begin{split}
  \omega =  (2\pi i)^3 \;  {}_1F_0\!\left(\left. \frac{1}{2} \right| t \right) \star 
   \left( \frac{1}{(1-t)^{\frac{1-p-q}{2}}}  \hpg21{p, q}{1 }{t} \right)^2\;,
 \end{split}
\eeq
for the values $(p, q)$ given in Table~\ref{tab:VHS5}.  
\end{corollary}
\begin{proof}
The proof follows by applying Equation~(\ref{eq:period_new}) to the periods $\omega(t)$ computed in Corollary~\ref{cor:periodK3mixed_transform}. 
\end{proof}
\begin{remark}
The Calabi-Yau operators~(\ref{operator3tilde}) obtained in Corollary~\ref{cor:extra_case} for parameters $(p, q)$, with classification number (and any alternative name used) in the AESZ database~\cite{Almkvist:aa}, are summarized in Table~\ref{tab:VHS5}. The Calabi-Yau operators are called Calabi-Yau operators of the extra case. 
\end{remark}
\begin{table}[H]
\scalebox{\Scale}{
\begin{tabular}{|c|cc|c|l|} 
\hline
$\#$ & AESZ & Name & $(p, q)$ & twist of K3 with $(M_n, k)$\\
\hline
\hline
&&&&\\[-0.9em]
1 & (17)& $35,3^*$		& $\left(\frac{1}{2}, \frac{1}{2}\right)$ & $(M_4, 3), (M_4, 4), (M_4, 5)$\\[0.4em]
2 & $-$ & $-$			& $\left(\frac{1}{3}, \frac{1}{2}\right)$ & $(M_3, 4)$\\[0.4em]
3 & (66) & $6^*$		& $\left(\frac{1}{4}, \frac{1}{2}\right)$ & $(M_2, 4)$\\[0.4em]
4 & $-$  & $14^*$ 		& $\left(\frac{1}{6}, \frac{1}{2}\right)$ & $(M_1, 4)$\\[0.4em]
5 & (39) & $4^*$		& $\left(\frac{1}{3}, \frac{1}{3}\right)$ & $(M_3, 5)$\\[0.4em]
6 & (20) & $46,4^{**}$	& $\left(\frac{1}{3}, \frac{2}{3}\right)$ & $(M_3, 3)$\\[0.4em]
7 & (45) & $8^*$		& $\left(\frac{1}{6}, \frac{1}{3}\right)$ & $(M_3, 1)$\\[0.4em]
8 & (34) & $8^{**}$		& $\left(\frac{1}{6}, \frac{2}{3}\right)$ & $(M_3, 2)$\\[0.4em]
\hline
\end{tabular}
\quad
\begin{tabular}{|c|cc|c|l|} 
\hline
$\#$ & AESZ & Name & $(p, q)$ & twist of K3 with $(M_n, k)$\\
\hline
\hline
9 & (38) & $10^*$		& $\left(\frac{1}{4}, \frac{1}{4}\right)$ & $(M_2, 5), (M_4, 1)$\\[0.4em]
10& (32) & $111,10^{**}$	& $\left(\frac{1}{4}, \frac{3}{4}\right)$ & $(M_2, 3), (M_4, 2)$\\[0.4em]
11& (40) & $13^*$		& $\left(\frac{1}{6}, \frac{1}{6}\right)$ & $(M_1, 5)$\\[0.4em]
12& (21) & $47,13^{**}$	& $\left(\frac{1}{6}, \frac{5}{6}\right)$ & $(M_1, 3)$\\[0.4em]
13& (44) & $7^*$		& $\left(\frac{1}{8}, \frac{3}{8}\right)$ & $(M_2, 1)$\\[0.4em]
14& (41) & $7^{**}$		& $\left(\frac{1}{8}, \frac{5}{8}\right)$ & $(M_2, 2)$\\[0.4em]
15& (43) & $9^*$		& $\left(\frac{1}{12}, \frac{5}{12}\right)$ & $(M_1, 1)$\\[0.4em]
16& (42) & $9^{**}$		& $\left(\frac{1}{12}, \frac{7}{12}\right)$ & $(M_1, 2)$\\[0.4em]
\hline
\end{tabular}}
\caption{Twist parameters for operators ${}_2L^{(4)}_t(p, q)$ in the `extra case'}\label{tab:VHS5}
\end{table}
\subsubsection{Calabi-Yau operators in the even case}
\label{ssec:even_case}
Applying our twist construction to the elliptic K3 surfaces in Remark~\ref{rem:double_transfo_rational}, we obtain the following:
\begin{corollary}
\label{even-case_a}
For the twisted families in Lemma~\ref{lemmaCYmixed_eq} with generalized functional invariant $(k, l, \beta)=(1,1,1)$ of the $M_n$-lattice polarized K3 surfaces in Remark~\ref{rem:double_transfo_rational}, the period integral~(\ref{eqn:period_new}) is annihilated by the Picard-Fuchs operator 
\beq
\label{L4even}
\begin{split}
 {}_3L_t^{(4)}\big(\mu, \tilde{\mu}\big)& =\theta^4 - t \, \left( 2\,\theta^{2}+2 \theta+{\tilde{\mu}}^{2}-\tilde{\mu}+1 \right)  \left( \theta + \mu \right)  \left( \theta -\mu+1 \right) \\
  & +t ^2 \,   \left( \theta+2-\mu \right)  \left( \theta+1+\mu \right)  \left( \theta + \mu \right)  \left( \theta + 1 - \mu \right) 
\end{split}
\eeq 
with $\mu=\frac{1}{2}$ and $\tilde{\mu} \in \{ \frac{1}{2}, \frac{1}{3}, \frac{1}{4}, \frac{1}{6} \}$. In particular, the period over $\Sigma_3(t)$ is holomorphic at $t=0$ and given by
\beq
\label{rk4even_a}
 \omega=(2\pi i)^3 \;\hpg21{\frac{1}{2} , \frac{1}{2} }{1}{t}  \star \left( \frac{1}{1-t} \; \hpg21{ \tilde{\mu}, 1-\tilde{\mu}}{1}{\frac{t}{t-1}} \right)\,.
\eeq  
\end{corollary}
\begin{proof}
The proof follows by applying Equation~(\ref{eq:period_new}) to the periods $\omega(t)$ computed in Remark~\ref{rem:double_transfo_rational}, and then using Equation~(\ref{eqn:cancellation}).
\end{proof}
\par To obtain the Calabi-Yau operators in Equation~(\ref{L4even}) with $\mu=\frac{1}{3}, \frac{1}{4}, \frac{1}{6}$ as Picard-Fuchs operators, we use the variant of our twist construction in Section~\ref{sssec:product_twists}. For the families of elliptic curves  $X=\tilde{X}_{141}$, $\tilde{X}_{431}$, $\tilde{X}_{321}$, $\tilde{X}_{211}$, and $X'_k$ in Table~\ref{tab:twist_rationalsurface}, we already constructed families of A-cycles $\tilde{\Sigma}_1(t)$ and $\Sigma_1(t)$, respectively, for $|t|<1$. 
\par For the twisted family in Equation~(\ref{CY3mtWEf}), we define a family of closed three-cycles $\hat{\Sigma}_3(t)$ as follows: Applying Lemma~\ref{lemma_cycle} to the elliptic curve $h^2= 4u^3 - g'_2(v)  u - g'_3(v)$, we obtain a family of A-cycles $\Sigma_1(v) \ni (u, h)$, such that $\Sigma_1(v)$ projects onto the circle $C_u=C_{1/2}(0)$, i.e., the circle $|u|=\frac{1}{2}$ in the $u$-plane with counterclockwise orientation, for every $|v|<1$. For $t \in \mathbb{C}$ and every $(v, u) \in C_v \times C_u$, a cycle $\hat{\Sigma}_1(t, v, u)$ in the elliptic fiber of Equation~(\ref{CY3mtWEf}) is obtained from $\hat{\Sigma}_1(\frac{t}{v})$, by rescaling $(X,Y) \to (h^2 v^2 X, h^3 v^3 Y)$ such that $(u, h) \in \Sigma_1(v)$. For $t \in \mathbb{C}$ with $|t|<1/2$, we obtain a continuously varying family of closed three-cycles as a warped product $\hat{\Sigma}_3(t) = C_v \times C_u \times_{(v, u)} \hat{\Sigma}_1(t, v, u)$. We have the following:
\begin{corollary}
\label{even-case_b}
For $X=\tilde{X}_{141}$, $\tilde{X}_{431}$, $\tilde{X}_{321}$, or $\tilde{X}_{211}$, the twist family of $X$ with $X'_k$ in Table~\ref{tab:twist_rationalsurface} given by Equation~(\ref{CY3mtWEf}) for $k=2,3,4$, is a family over $B$ of Jacobian elliptic Calabi-Yau threefolds over $\mathbb{F}_n$ with $n=0, \dots, k$. The period integral~(\ref{eqn:period_new_producttwist}) is annihilated by the Picard-Fuchs operator ${}_3L_t^{(4)}\big(\mu, \tilde{\mu}\big)$ in Equation~(\ref{L4even}). In particular, the period over $\hat{\Sigma}_3(t)$ is holomorphic at $t=0$ and given by
\beq
\label{rk4even_b}
 \hat{\omega}=(2\pi i)^3 \;\hpg21{\mu , 1-\mu }{1}{t}  \star \left( \frac{1}{1-t} \; \hpg21{ \tilde{\mu}, 1-\tilde{\mu}}{1}{\frac{t}{t-1}} \right)
\eeq  
with $\mu \in \{  \frac{1}{3}, \frac{1}{4}, \frac{1}{6} \}$ and $\tilde{\mu} \in \{ \frac{1}{2}, \frac{1}{3}, \frac{1}{4}, \frac{1}{6} \}$.
\end{corollary}
\begin{proof}
The proof follows by applying Proposition~\ref{prop:period_integral_general_producttwist} and Equation~(\ref{eq:period_new_producttwist}) to the periods $\omega(t)$ computed in Corollary~\ref{lemma2}.
\end{proof}
\begin{remark}
The Calabi-Yau operators~(\ref{L4even}) obtained in Corollary~\ref{even-case_a} and Corollary~\ref{even-case_b} for parameters $(\mu,\tilde{\mu})$, with their classification number (and any alternative name used) in the AESZ database~\cite{Almkvist:aa}, are summarized in Table~\ref{tab:VHS4}. The Calabi-Yau operators are called Calabi-Yau operators of the even case. 
\end{remark}
\begin{table}
\scalebox{\ScaleBig}{
\begin{tabular}{|c|cc|c|} 
\hline
& & & \\[-0.9em]
\# & AESZ & Name & $(\mu,\tilde{\mu})$ \\
\hline
\hline
& & & \\[-0.9em]
1& (32)  & 111	& $ \left(\frac{1}{2},\frac{1}{2}\right)$ \\[0.4em]
2& (31) & 110	& $ \left(\frac{1}{3},\frac{1}{2}\right)$ \\[0.4em]
3& (15) & 30 	& $ \left(\frac{1}{4}, \frac{1}{2}\right)$ \\[0.4em]
4& (33)&112	& $ \left(\frac{1}{6},\frac{1}{2}\right)$ \\[0.4em]
\hline
5& (34)  & 141, 8** 	& $ \left(\frac{1}{2},\frac{1}{3}\right)$ \\[0.4em]
6& (35)  & 142		& $ \left(\frac{1}{3},\frac{1}{3}\right)$ \\[0.4em]
7& -     & 196 		& $ \left(\frac{1}{4}, \frac{1}{3}\right)$ \\[0.4em]
8& (36)&143		& $ \left(\frac{1}{6},\frac{1}{3}\right)$ \\[0.4em]
\hline
\end{tabular}
\quad
\begin{tabular}{|c|cc|c|} 
\hline
& & & \\[-0.9em]
\# & AESZ & Name & $(\mu,\tilde{\mu})$ \\
\hline
\hline
9& (41)  & 189, 7** 	& $ \left(\frac{1}{2},\frac{1}{4}\right)$ \\[0.4em]
10& (46)  & 194		& $ \left(\frac{1}{3},\frac{1}{4}\right)$ \\[0.4em]
11& (48)& 197 		& $ \left(\frac{1}{4}, \frac{1}{4}\right)$\\[0.4em]
12 & (50)&199		& $ \left(\frac{1}{6},\frac{1}{4}\right)$ \\[0.4em]
\hline
13& (42)  & 190, 9**	& $ \left(\frac{1}{2},\frac{1}{6}\right)$  \\[0.4em]
14& (47) &	195	& $ \left(\frac{1}{3},\frac{1}{6}\right)$ \\[0.4em]
15& (49)& 198	& $ \left(\frac{1}{4}, \frac{1}{6}\right)$ \\[0.4em]
16 & (23)& 61	& $ \left(\frac{1}{6}, \frac{1}{6}\right)$\\[0.4em]
\hline
\end{tabular}}
\caption{Twist parameters for the operators ${}_3L^{(4)}_t(p, q)$ in the `even case'}\label{tab:VHS4}
\end{table}
\subsection{Calabi-Yau operators in the odd case}
In this section we describe a fourth step in our iterative construction that produces families of (singular) elliptic Calabi-Yau fourfolds with section over $\mathbb{P}^1 \times \mathbb{P}^1 \times \mathbb{P}^1$ that realize all 14 one-parameter variations of Hodge structure of weight four and type $(1, 1, 1, 1, 1)$ over a one-dimensional rational deformation space of the so-called odd case. The families arise as twisted families of the elliptic Calabi-Yau threefolds of the hypergeometric case, previously obtained in Section~\ref{sssec:hg}. Applying our twist construction, we obtain their Weierstrass model as twisted families with generalized functional invariant $(m, n, \gamma)=(1,1,1)$. We have the following:
\begin{lemma} 
\label{lemmaCY4mixed}
For every family of threefolds from Section~\ref{sssec:hg}, the twisted family with generalized functional invariant $(1,1,1)$, given by the Weierstrass equation
\beq
\label{lemmaCYodd_eq}
\begin{split}
 Y^2 = 4 X^3  - \underbrace{g_2\left(  - \frac{t}{w (w+1)}, u, v \right) \, v^4 (v+1)^{4} }_{=: \, g_2(t, u, v, w)}  X  - \underbrace{g_3\left(  - \frac{t}{w (w+1)},u  \right)  \, w^6  (w+1)^{6} }_{=: \, g_3(t, u, v, w) }   \,,
 \end{split}
\eeq
defines a family (over $B$) of Jacobian elliptic Calabi-Yau fourfolds over $\mathbb{P}^1 \times \mathbb{P}^1  \times \mathbb{P}^1$.
\end{lemma}
\begin{proof}
The construction is an application of the general construction in Section~\ref{ssec:twisting}. One first checks that Equation~(\ref{lemmaCYodd_eq}) defines a minimal Weierstrass model for every family of threefolds from Section~\ref{sssec:hg} with affine coordinates $u, v, w \in \mathbb{C}$ and $t \in B=\mathbb{P}^1 \backslash \lbrace 0, 1, \infty \rbrace$. The Weierstrass equation~(\ref{lemmaCYodd_eq}) extends to $\mathbb{P}^1 \times \mathbb{P}^1 \times \mathbb{P}^1$, since we obtain a minimal and normal Weierstrass equation when introducing projective variables $[u_0:u_1]\in\mathbb{P}^1$, $[v_0:v_1]\in\mathbb{P}^1$, $[w_0:w_1]\in\mathbb{P}^1$, and $[x:y:z] \in \mathbb{P}(2,3,1)$ and writing each fiber as the hypersurface
\beq
\label{CY4mtWEa}
y^{2} z  = 4 x^{3} - g_2\left( t, \frac{u_0}{u_1}, \frac{v_0}{v_1},  \frac{w_0}{w_1} \right) \, u_1^8 v_1^8 w_1^8 x z^2  
- g_3\left(t,  \frac{u_0}{u_1}, \frac{v_0}{v_1}, \frac{w_0}{w_1} \right) \, u_1^{12} v_1^{12} w_1^{12} z^3 \,.
\eeq
Four $\mathbb{C}^*$-groups act on the defining variables in Equation~(\ref{CY4mtWEa}) and are  given by the weights listed in Table~\ref{tab:deg4} where \emph{deg} denotes the total weight of Equation~(\ref{CY4mtWEa})  and \emph{sum} denotes the sum of weights of the defining variables. Since the conditions are satisfied that for each $\mathbb{C}^*$-group the total weight equals the sum of weights, a Calabi-Yau fourfold is obtained by removing the loci $\lbrace s_0 = s_1 = 0\rbrace$, $\lbrace u_0 = u_1 = 0\rbrace $, $\lbrace v_0 = v_1 = 0\rbrace$, $\lbrace x = y = z = 0\rbrace$ from the solution set of Equation~(\ref{CY4mtWEa}) and taking the quotient $(\mathbb{C}^*)^4$. 
\begin{table}[H]
\scalebox{\ScaleBig}{
\begin{tabular}{|c|c|ccc|cc|cc|cc|c|}
\hline
$\mathbb{C}^*$ & deg & $x$ & $y$ & $z$ & $u_0$ & $u_1$ & $v_0$ & $v_1$ & $w_0$ & $w_1$ & $\Sigma$ \\
\hline
\hline
$\lambda_1$ & $3$ & $1$ & $1$ & $1$ & $0$ & $0$ & $0$ & $0$ & $0$ & $0$ & $3$\\
$\lambda_2$ & $12$ & $4$ & $6$ & $0$ & $1$ & $1$ & $0$ & $0$ & $0$ & $0$ & $12$\\
$\lambda_3$ & $12$ & $4$ & $6$ & $0$ & $0$ & $0$ & $1$ & $1$ & $0$ & $0$ & $12$\\
$\lambda_4$ & $12$ & $4$ & $6$ & $0$ & $0$ & $0$ & $0$ & $0$ & $1$ & $1$ & $12$\\
\hline
\end{tabular}
\caption{Weights of variables in Weierstrass equation}\label{tab:deg4}}
\end{table}
\end{proof}
\par For each twisted family in Equation~(\ref{lemmaCYodd_eq}), we define a family of closed four-cycles $\Sigma_4(t)$ as follows: For $t \in \mathbb{C}$ with $|t|<1/2$, we start with the circle $C=C_{1/2}(0)$, given by $|w|=\frac{1}{2}$ in the $w$-plane with counterclockwise orientation. For every $w \in C$, a three-cycle $\Sigma_3'(t, v)$ in the fiber is obtained from $\Sigma_3(-\frac{t}{w (w+1)})$, $\Sigma_3(t)$ was defined in Section~\ref{ssec:3folds_hg_even}, by rescaling $(u, v, X,Y) \to (u, v, w^2 (w+1)^{2} X, w^3 (w+1)^{3} Y)$. For $t\in \mathbb{C}$ with $|t|<1/2$, we obtain a continuously varying family of closed four-cycles as a warped product $\Sigma_4(t) = C \times_w \Sigma'_3(t, w)$.  We have the following:
\begin{corollary}
\label{odd-case}
For the twisted families in Lemma~\ref{lemmaCY4mixed} with generalized functional invariant $(1,1,1)$, the period integral~(\ref{eqn:period_new}) is annihilated by the self-adjoint, rank-five Picard-Fuchs operator 
\beq
\label{operator4}
  L^{(5)}_t(p, q) = \theta^{5} - t   \Big( \theta+ \frac{1}{2} \Big)    \Big( \theta + p \Big)   \Big( \theta + q\Big)     \Big( \theta + 1-p \Big)   \Big( \theta + 1-q\Big)   \,.
\eeq
In particular, the period over $\Sigma_4(t)$ is holomorphic at $t=0$ and given by
\beq
\label{period_lemmaCY4pure}
 \omega  = (2 \pi i)^4 \;\hpg54{p, q, \frac{1}{2}, 1-q, 1-p}{ 1, 1, 1, 1}{t} \;
\eeq
for the values $(p, q)$ given in Table~\ref{tab:VHS}.  
\end{corollary}
\begin{proof}
The proof follows by applying Equation~(\ref{eq:period_new}) to the periods $\omega(t)$ computed in Corollary~\ref{hg-case}. One then checks that the Hadamard product in Proposition~\ref{prop:period_integral_general} produces a hypergeometric function of type ${}_5F_4$ using Equation~(\ref{eqn:cancellation}).
\end{proof}
We have the following:
\begin{corollary}
\label{YY5operators_sqr} 
The differential operators ${}_4L_t^{(p, q)}$, given by
\beq
\label{opYY1hat}
\begin{split}
& {}_4L^{(4)}_t(p, q)   =
 \theta^{4} 
-  \frac{1}{4} \, t\, \Big( 8 \, \theta^4 + 16 \, \theta^3 - 2 \, ({p}^{2}+q^2-p-q-9) \, \theta^{2}  \\
   & \quad  - 2\, (p^{2}+q^2-p-q-5)\, \theta +  
    2  +p+ q - p q - {p}^{2}- q^2 +p^2 q +p\,{q}^{2} +{p}^{2}{q}^{2} \Big)\\
+ & \frac{1}{16}\,t^2 \, \left( 2\,\theta+2+p-q \right)  \left( 2\,\theta+1+
p+q \right)  \left( 2\,\theta+2-p+q \right) 
 \left( 2\,\theta+3-p-q \right) ,
\end{split}
\eeq
are the Yifan-Yang pullbacks of the operators $L^{(5)}_t(p, q)$ in Corollary~\ref{odd-case} of minimal degree (in $t$), for the values $(p, q)$ given in Table~\ref{tab:VHS_YYPB}.
\end{corollary}
\begin{proof}
The proof follows directly from Proposition~\ref{YY5operators}.
\end{proof}
\begin{remark}
The Calabi-Yau operators~(\ref{opYY1hat}) obtained in Corollary~\ref{YY5operators_sqr} for parameters $(p, q)$, with their classification number (and any alternative name used) in the AESZ database~\cite{Almkvist:aa}, are summarized in Table~\ref{tab:VHS_YYPB}. The Calabi-Yau operators are called Calabi-Yau operators of the odd case. 
\end{remark}
\begin{table}
\scalebox{\ScaleBig}{
\begin{tabular}{|c|cc|c|}
\hline
&&& \\[-0.8em]
\# & AESZ & Name & $(p, q)$\\
\hline
\hline
&&&\\[-0.9em]
1 & (51)& $\tilde{3}, 204$ & $\left(\frac{1}{2},\frac{1}{2}\right)$ \\[0.2em]
2  & (92)& $\tilde{5}$ & $\left(\frac{1}{3},\frac{1}{2}\right)$\\[0.2em]
3  & (91)& $\tilde{4}$ & $\left(\frac{1}{3},\frac{1}{3}\right)$\\[0.2em]
4 & (93) & $\tilde{6}$ & $\left(\frac{1}{4},\frac{1}{2}\right)$ \\[0.2em]
5 & (98) & $\tilde{11}$ & $\left(\frac{1}{4},\frac{1}{3}\right)$ \\[0.2em]
6 & (97) & $\tilde{10}$ & $\left(\frac{1}{4},\frac{1}{4}\right)$ \\[0.2em]
7 & (101) & $\tilde{14}$ & $\left(\frac{1}{6},\frac{1}{2}\right)$ \\[0.2em]
\hline
\end{tabular}
\quad
\begin{tabular}{|c|cc|c|}
\hline
&&& \\[-0.8em]
\# & AESZ & Name & $(p, q)$\\
\hline
\hline
8 & (95) & $\tilde{8}$ & $\left(\frac{1}{6},\frac{1}{3}\right)$ \\[0.2em]
9 & (99) & $\tilde{12}$ & $\left(\frac{1}{6},\frac{1}{4}\right)$ \\[0.2em]
10 & (100) & $\tilde{13}$ & $\left(\frac{1}{6},\frac{1}{6}\right)$ \\[0.2em]
11 & (89) & $\tilde{1}$ & $\left(\frac{1}{5},\frac{2}{5}\right)$ \\[0.2em]
12 & (94)& $\tilde{7}$ & $\left(\frac{1}{8},\frac{3}{8}\right)$ \\[0.2em]
13 & (90) & $\tilde{2}$ & $\left(\frac{1}{10},\frac{3}{10}\right)$ \\[0.2em]
14 & (96) & $\tilde{9}$ & $\left(\frac{1}{12},\frac{5}{12}\right)$ \\[0.2em]
\hline
\end{tabular}}
\caption{Twist parameters for the operators ${}_4L^{(4)}_t(p, q)$ in the `odd case'}\label{tab:VHS_YYPB}
\end{table}
\section{Proof of Theorem~\ref{DoranMalmendier}}
\label{proof}
In Section~\ref{prelim} we have defined an iterative construction that produces families of elliptically fibered Calabi-Yau $n$-folds with section 
from families of elliptic Calabi-Yau varieties  of one dimension lower by a combination of a quadratic twist and a rational base transformation encoded in the
generalized functional invariant. Moreover, all Weierstrass models are obtained through a sequence of constructions that start with
the quadric pencil in Equation~(\ref{quadric_pencil}).  Each step $n= 1, 2, 3, 4$ of our iterative construction has also provided a family of a closed transcendental $n$-cycle 
for each family of $n$-folds as the warped product of the corresponding transcendental cycle in dimension $n-1$. Upon integration of this cycle with the holomorphic $n$-form we obtain a period for the family of elliptically fibered Calabi-Yau $n$-folds with section. By construction, the period is holomorphic on the unit disk about the point $t=0$ of maximally unipotent monodromy. Each holomorphic period is then annihilated by a Picard-Fuchs operator which is a Calabi-Yau operator in the sense of \cite{MR2282972}. 
\par The proof of Theorem~\ref{DoranMalmendier} proceeds as follows:
Bogner and Reiter classified all $\operatorname{Sp}(4,\mathbb{C})$-rigid, quasi-unipotent local systems which have a maximal unipotent element and are induced by fourth-order Calabi-Yau operators. In particular, they obtained explicit expressions for all Calabi-Yau operators and closed formulas for special solutions of them. 
We prove that we have realized all of these operators and holomorphic periods. There are the four cases:
\begin{enumerate}
\item The \emph{hypergeometric case} consist of 14 operators called $P_1(4,10,4)$ \cite{MR2980467}*{Theorem~6.1}. These operators precisely coincide with the 14 operators of Equations~(\ref{L4mu}) and (\ref{L4mu18}) obtained by the twist construction for parameters given in Table~\ref{tab:VHS}.  
\item The \emph{extra case} consist of 16 operators called $P_2(4,6,6)$ \cite{MR2980467}*{Theorem~6.3}.  These operators precisely coincide with the 16 operators of Equation~(\ref{operator3tilde}) obtained by the twist construction for parameters listed in Table~\ref{tab:VHS5}.  
\item The \emph{even case} consist of 16 operators called $P_2(4,6,8)$ \cite{MR2980467}*{Theorem~6.4}.  These operators precisely coincide with the 16 operators of Equation~(\ref{rk4even_a}) obtained by the twist construction for parameters listed in Table~\ref{tab:VHS4}. 
\item The \emph{odd case} consist of 14 operators called $P_1(4,8,4)$ \cite{MR2980467}*{Theorem~6.2}.  These operators precisely coincide with the Yifan-Yang pullbacks in Equations~(\ref{opYY1})  of the 14 operators of Equations~(\ref{operator4})  obtained by the twist construction for the parameters listed in Table~\ref{tab:VHS_YYPB}. 
\end{enumerate}
\hfill $\square$
\begin{remark}
\label{rem:rigid_irred}
We constructed all symplectically rigid Calabi-Yau operators as Picard-Fuchs operators of families of Calabi-Yau varieties. These operators are rank-four, degree-two, irreducible Calabi-Yau operators with three regular singular points.  In addition to these operators, there are four additional rank-four, degree-two, irreducible Calabi-Yau operators with three regular singular points in the AESZ database~\cite{Almkvist:aa}. The additional cases, 84, 254, 255, 406, have as degree-one term an irreducible polynomial (over $\mathbb{Q}[t]$) of degree four; see Remark~\ref{rem:non-rigid_irred}.
\end{remark}
\section{Beyond symplectically rigid Calabi-Yau operators}
\label{sec:beyond}
\subsection{Calabi-Yau operators from Heun's equation}
Heun's equation is the rank-two, linear ordinary differential equation of the form
\beq
\label{heun_ode}
 \left(  \dfrac{d^2}{dt^2} + \left(\dfrac{\gamma}{t} + \dfrac{\delta}{t-1} + \dfrac{\epsilon}{u-a} \right) \dfrac{d}{dt} + \dfrac{\alpha \beta t -q}{t (t-1) (t-a)} \right)   \omega(t) =0 \;,
\eeq
such that $\epsilon=\alpha+\beta-\gamma-\delta+1$ to ensure that the point at infinity is a regular singular point. The parameter $q \in \mathbb{C}$ is called the accessory parameter. For $a \in \mathbb{C}$ and $a \not= 0, 1$, Heun's equation has four regular singular points at $0, 1, a, \infty$ and the Riemann symbol
\beq
\label{RiemannSymbol_heun}
 \mathcal{P}\left. \left(\begin{array}{cccc}
 0 & 1 & a & \infty \\ \hline  
 1-\gamma & 1-\delta & 1-\epsilon & \alpha \\ 
 0 & 0 & 0 & \beta
  \end{array} \right| t \right) \,.
\eeq
Every rank-two linear ordinary differential equation with at most four regular singular points can be transformed into this equation by a change of variable.  The function $\heun(a, q; \alpha,\beta,\gamma,\delta \, | t)$ is the unique solution of Heun's differential equation that is holomorphic and 1 at the singular point $t=0$.  We have the following:
\begin{lemma}
\label{lem:Heun}
The function $\omega(t)=\heun(a, q; 1,1,1,1 \, | t)$ is the unique solution of $L^{(2)}_t \omega(t)=0$, holomorphic and 1 at $t=0$, with
\beq
\label{eqn:L2heun}
 L^{(2)}_t\big(a, q\big)   = \theta^2 -  \frac{t}{a}   \big((a+1) \theta^2 +(a+1) \theta + q \big) + \frac{t^2}{a}  \big(\theta+1\big)^2 \,.
\eeq 
For $\alpha \in (0,1) \cap \mathbb{Q}$, the function $\omega(t)=\heun(a,\frac{q}{4};\alpha,1-\alpha,1,\frac{1}{2} \, |t )^2$ is the unique solution of $L^{(3)}_t \omega(t)=0$, holomorphic and 1 at $t=0$, with
\beq
\label{eqn:1L3heun}
\begin{split}
 L^{(3)}_t\big(\alpha; a, q\big)   = \theta^3 & -  \frac{t}{2a}   \big(2\theta+1\big)  \big((a+1) \theta^2 +(a+1) \theta + q \big)\\
 & +  \frac{t^2}{a}  \big(\theta+2 \alpha\big)  \big(\theta+2(1-\alpha)\big) \big(\theta+1\big) \,.
\end{split}
\eeq
We also have the following identity involving a Hadamard product:
\beq
\label{eqn:heun_identity}
  \heun\left(a,\frac{q}{4};\frac{1}{4},\frac{3}{4},1,\frac{1}{2} \, \Big|t \right)^2 =  {}_1F_0\!\left(\left. \frac{1}{2} \right| t\right) \star  \heun(a, q; 1,1,1,1 \, | t) \,.
\eeq
\end{lemma}
\begin{proof}
The proof follows by explicit computation.
\end{proof}
To obtain Calabi-Yau operators, the following lemma is essential:
\begin{lemma}
\label{lem:CYO_nonrigid}
For $\alpha \in (0,1) \cap \mathbb{Q}$, the function 
\beqn
 \omega(t)=  {}_1F_0\!\left(\left. \frac{1}{2} \right| t\right) \star \heun\left(a,\frac{q}{4};\alpha,1-\alpha,1,\frac{1}{2} \, \Big|t \right)^2
 \eeqn
 is the unique solution of ${}_1L^{(4)}_t \omega(t)=0$, holomorphic and 1 at $t=0$, with
\beq
\label{eqn:CYO_4a}
\begin{split}
 {}_1L^{(4)}_{t} \big(\alpha; a, q \big)  = \theta^4  & - \frac{t}{4a}  \big(2\theta+1\big)^2 \big((a+1) \theta^2 + (a+1) \theta + q \big)  \\
 & + \frac{t}{4a} \big(2\theta+1\big)  \big(2\theta+3\big) \big(\theta+2\beta\big)  \big(\theta+2(1-\beta)\big) \,.
 \end{split}
\eeq
For $\alpha \in (0,1) \cap \mathbb{Q}$, the function 
\beqn
 \omega(t)=  \hpg21{\alpha, 1-\alpha}{1}{t}   \star \heun(a, q; 1,1,1,1 \, | t)
 \eeqn
 is the unique solution of $_{2}L^{(4)}_t \omega(t)=0$, holomorphic and 1 at $t=0$, with
\beq
\begin{split}
\label{eqn:CYO_4b}
 _{2}L^{(4)}_{t}\big(\alpha ; a, q \big)  =\theta^4  & - \frac{t}{a}  \big(\theta+\alpha \big)  \big(\theta+1-\alpha\big)  \big((a+1) \theta^2 +(a+1) \theta + q \big)   \\
 & +  \frac{t^2}{a}   \big(\theta+\alpha \big)  \big(\theta+1-\alpha\big)  \big(\theta+\alpha +1\big)  \big(\theta+2-\alpha\big)  \,.
\end{split}
\eeq
\end{lemma}
\begin{proof}
The proof follows by explicit computation.
\end{proof}
The rank-four operators in Equations~(\ref{eqn:CYO_4a}) and~(\ref{eqn:CYO_4b}) have four regular singular points at $0,1, a, \infty$. In particular, they are not symplectically rigid operators. Notice that rescaling $t \mapsto \lambda a t$ leaves the operator $\theta$ invariant and allows us to clear denominators. We then have the following:
\begin{proposition}
The rank-four operators ${}_iL^{(4)}_{\lambda a t}\big(\alpha ; a, q \big)$ with $i=1,2$ in Equation~(\ref{eqn:CYO_4a}) and Equation~(\ref{eqn:CYO_4b}), with parameters $(\alpha; a; q;\lambda)$ given in Table~\ref{tab:VHS6}, constitute all 33 rank-four, degree-two Calabi-Yau operators in the AESZ database~\cite{Almkvist:aa} with four regular singular points whose degree-one term is not an irreducible polynomial (over $\mathbb{Q}[t]$) of degree four.
\end{proposition}
\begin{proof}
The proof follows by explicit computation.
\end{proof}
\begin{remark}
The fact that there are two solutions for each entry in Table~\ref{tab:VHS6} is due to the following identity for Heun functions
\beqn
 \heun(a, q; 1,1,1,1 | t ) =  \heun\left(\frac{1}{a},\frac{q}{a}; 1,1,1,1 \Big| \frac{t}{a} \right) \,.
 \eeqn
\end{remark}
\begin{remark}
\label{rem:non-rigid_irred}
In the AESZ database~\cite{Almkvist:aa}, there are 36 rank-four, degree-two, irreducible Calabi-Yau operators with four regular singular points. Three additional cases, 18, 182, 205, that do not appear in Table~\ref{tab:VHS6} have as degree-one term an irreducible polynomial (over $\mathbb{Q}[t]$) of degree four.
\end{remark}
\begin{table}
\scalebox{\Scale}{
\begin{tabular}{|c|c|} 
\hline
AESZ & $(\alpha; a, q; \lambda)$ in $_{1}L^{(4)}_{\lambda a t}\big(\alpha ; a, q \big)$\\[0.2em]
\hline
\hline
&\\[-0.9em]
16 	& $\left(\frac{1}{2}, 4, 2, 64\right), \left(\frac{1}{2}, \frac{1}{4},\frac{1}{2},16\right)$  \\[0.4em]
29 	& $\left(\frac{1}{2}, 577\pm408 \sqrt{2}, 170 \pm 120 \sqrt{2}, 68 \pm 48 \sqrt{2}\right)$ \\[0.4em]
41 	& $\left(\frac{1}{2}, \frac{17}{81}\pm i \frac{56}{81} \sqrt{2}, \frac{14}{27} \pm i \frac{8}{27} \sqrt{2}, 28 \pm i \, 16 \sqrt{2}\right)$ \\[0.4em]
42 	& $\left(\frac{1}{2}, 17 \pm 12 \sqrt{2}, 6 \pm 4 \sqrt{2}, 48 \pm 32 \sqrt{2}\right)$ \\[0.4em]
184 	& $\left(\frac{1}{2}, \frac{117}{125}\pm i \frac{44}{125}, \frac{22}{25} \pm i \frac{4}{25}, 44 \pm i \,  8\right)$\\[0.4em]
185	& $\left(\frac{1}{2}, -7 \pm 4 \sqrt{3}, -2 \pm \frac{4}{3} \sqrt{3}, 36 \mp 24 \sqrt{3}\right)$\\[0.4em]
\hline
&\\[-0.9em]
25 	& $ \left(\frac{1}{4}, \frac{-123 \pm 55 \sqrt{5}}{2} , \frac{-33 \pm 15 \sqrt{5}}{2}, 88\mp 40 \, \sqrt{5} \right)$\\[0.4em]
36 	& $ \left(\frac{1}{4}, 2,1,128\right), \left(\frac{1}{4}, \frac{1}{2},\frac{1}{2}, 64\right)$\\[0.4em]
45 	& $ \left(\frac{1}{4}, -8,-2,128\right), \left(\frac{1}{4}, -\frac{1}{8},\frac{1}{4},-16\right)$\\[0.4em]
58 	& $ \left(\frac{1}{4}, 9,3,144\right), \left(\frac{1}{4}, \frac{1}{9},\frac{1}{3},16\right)$\\[0.4em]
133 	& $ \left(\frac{1}{4}, \frac{1 \pm i \sqrt{3}}{2} , \frac{3 \pm i \sqrt{3}}{6}, 72 \pm 34 \, i \sqrt{3}\right)$\\[0.4em]
137	& $ \left(\frac{1}{4}, \frac{9}{8},\frac{3}{4},144\right), \left(\frac{1}{2}, \frac{8}{9},\frac{2}{3},128\right)$\\[0.4em]
\cline{1-2}
&\\[-0.9em]
18 	& $\left(\frac{3}{8}, -4, -1, 64\right), \left(\frac{1}{2}, -\frac{1}{4},\frac{1}{4},-16\right)$\\[0.4em]
183 	& $\left(\frac{3}{8}, \frac{4}{3}, 1, 64\right), \left(\frac{3}{8}, \frac{3}{4},\frac{3}{4},48\right)$\\[0.4em]
\cline{1-2}
&\\[-0.9em]
26 	& $ \left(\frac{1}{3}, -27, -8, 108 \right), \left(\frac{1}{3}, -\frac{1}{27},\frac{8}{27},-4\right)$\\[0.4em]
\hline
\end{tabular}
\quad
\begin{tabular}{|c|c|} 
\hline
AESZ & $(\alpha; a, q; \lambda)$ in $_{2}L^{(4)}_{\lambda a t}\big(\alpha ; a, q \big)$\\[0.2em]
\hline
\hline
&\\[-0.9em]
48 	& $ \left(\frac{1}{3}, 2,1,216\right), \left(\frac{1}{3}, \frac{1}{2},\frac{1}{2},108\right)$\\[0.4em]
38 	& $ \left(\frac{1}{4}, 2,1,512\right), \left(\frac{1}{4}, \frac{1}{2},\frac{1}{2},256\right)$\\[0.4em]
65 	& $ \left(\frac{1}{6}, 2,1,3456\right), \left(\frac{1}{6}, \frac{1}{2},\frac{1}{2},1728\right)$\\[0.4em]
\hline
&\\[-0.9em]
134 	& $ \left(\frac{1}{3}, \frac{1 \pm i \sqrt{3}}{2} , \frac{3 \pm i \sqrt{3}}{6}, \frac{243 \pm 81 \, i \sqrt{3}}{2}\right)$\\[0.4em]
135 	& $ \left(\frac{1}{4}, \frac{1 \pm i \sqrt{3}}{2} , \frac{3 \pm i \sqrt{3}}{6}, 288 \pm 96 \, i \sqrt{3}\right)$\\[0.4em]
136 	& $ \left(\frac{1}{6}, \frac{1 \pm i \sqrt{3}}{2} , \frac{3 \pm i \sqrt{3}}{6}, 1944 \pm 648 \, i \sqrt{3}\right)$\\[0.4em]
\hline
&\\[-0.9em]
24	& $ \left(\frac{1}{3}, \frac{-123 \pm 55 \sqrt{5}}{2} , \frac{-33 \pm 15 \sqrt{5}}{2}, \frac{297 \mp 135 \sqrt{5}}{2}\right)$\\[0.4em]
51 	& $ \left(\frac{1}{4}, \frac{-123 \pm 55 \sqrt{5}}{2} , \frac{-33 \pm 15 \sqrt{5}}{2}, 352\mp 160 \sqrt{5} \right)$\\[0.4em]
63	& $ \left(\frac{1}{6}, \frac{-123 \pm 55 \sqrt{5}}{2} , \frac{-33 \pm 15 \sqrt{5}}{2}, 2376 \mp 1080 \sqrt{5}\right)$\\[0.4em]
\hline
&\\[-0.9em]
15 	& $ \left(\frac{1}{3}, -8,-2,216\right), \left(\frac{1}{3}, -\frac{1}{8},\frac{1}{4},-27\right)$\\[0.4em]
68 	& $ \left(\frac{1}{4}, -8,-2,512\right), \left(\frac{1}{4}, -\frac{1}{8},\frac{1}{4},-64\right)$\\[0.4em]
62 	& $ \left(\frac{1}{6}, -8,-2,3456\right), \left(\frac{1}{6}, -\frac{1}{8},\frac{1}{4},-432\right)$\\[0.4em]
\hline
&\\[-0.9em]
70 	& $ \left(\frac{1}{3}, 9,3,243\right), \left(\frac{1}{3}, \frac{1}{9},\frac{1}{3},27\right)$\\[0.4em]
69 	& $ \left(\frac{1}{4}, 9,3,576\right), \left(\frac{1}{4}, \frac{1}{9},\frac{1}{3},64\right)$\\[0.4em]
64 	& $ \left(\frac{1}{6}, 9,3,3888\right), \left(\frac{1}{6}, \frac{1}{9},\frac{1}{3},432\right)$\\[0.4em]
\hline
&\\[-0.9em]
138	& $ \left(\frac{1}{3}, \frac{9}{8},\frac{3}{4},243\right), \left(\frac{1}{3}, \frac{8}{9},\frac{2}{3},216\right)$\\[0.4em]
139	& $ \left(\frac{1}{4}, \frac{9}{8},\frac{3}{4},576\right), \left(\frac{1}{4}, \frac{8}{9},\frac{2}{3},512\right)$\\[0.4em]
140	& $ \left(\frac{1}{6}, \frac{9}{8},\frac{3}{4},3888\right), \left(\frac{1}{6}, \frac{8}{9},\frac{2}{3},3456\right)$\\[0.4em]
\hline
\end{tabular}}
\caption{Non-rigid Calabi-Yau operators from Lemma~\ref{lem:CYO_nonrigid}}\label{tab:VHS6}
\end{table}
\subsection{Realizing non-rigid Calabi-Yau operators}
The twisted families of Section~\ref{ssec:ec_families} are precisely the extremal families of elliptic curves with three singular fibers and rational total space  (up to quadratic twist and two-isogeny) classified in \cite{MR867347}*{Tab.~5.2}.  Miranda and Persson also classified the extremal families of elliptic curves with rational total space and four singular fibers, the highest number that can occur,  in \cite{MR867347}*{Tab.~5.3}. There are six of them, in the notation of Herfurtner denoted as $X_{5511}$, $X_{6321}$, $X_{4422}$, $X_{8211}$, $X_{3333}$, and $X_{9111}$. Analogous to Corollary~\ref{cor:modular_rational}, they are the modular elliptic surfaces for the subgroups $\Gamma_1(5)$, $\Gamma_0(6)$, $\Gamma_0(4)\cap \Gamma(2)$, $\Gamma_0(8)$, $\Gamma(3)$, and $\Gamma_0(9)$, respectively. 
\par The latter four, $X_{4422}$, $X_{8211}$, $X_{3333}$, and $X_{9111}$, are easily understood in terms of the construction in Section~\ref{sec:base_transformations}. For example, $X_{8211}$ and $X_{9111}$ are pull-backs of modular elliptic surfaces $X_{141}$ or $X_{431}$ in Table~\ref{tab:3ExtRatHg}, along the map $t \mapsto t^k$ with $k=2$ or $k=3$, respectively. Similar arguments apply to $X_{4422}$ and $X_{3333}$. Weierstrass models for the elliptic surfaces $X_{5511}$ and $X_{6321}$ for subgroups $\Gamma_1(5)$, $\Gamma_0(6)$ are given in Table~\ref{tab:4ExtRatHl}. 
\begin{table}
\scalebox{\Scale}{
\begin{tabular}{|c|c|lcl|c|c|c|c|} \hline
Surface & $G$ & \multicolumn{3}{|c|}{$g_2, g_3, \Delta, J$, sections} & \multicolumn{4}{|c|}{Ramification of $J$ and singular fibers} \\
\cline{6-9}
& $\mathrm{MW}(\pi)$ & & & & $t$ & $J$ & $m(j)$ & fiber   \\
\hline
&&&&&&&& \\[-0.9em]
$X_{5511}$            &  $\Gamma_0(5)$
  & $g_2$     & $=$ & $\frac{3}{4} c^4 t^4-9 c^3 t^3+\frac{21}{2} c^2 t^2+9 c t+\frac{3}{4}$  		& $(t^4-12 t^3+14 t^2+12 t+1=0)/c$ & $0$        & $3$ & smooth  \\[0.4em]
 & $\mathbb{Z}/5\mathbb{Z}$
  & $g_3$     & $=$ & $-\frac{1}{6} c^6 t^6+\frac{9}{4} c^5 t^5-\frac{75}{8} c^4 t^4-\frac{75}{8} c^2 t^2-\frac{9}{4} c t-\frac{1}{8}$	& $\pm\frac{i}{c}$ & $1$ & $2$ & smooth  \\[0.4em]
&& $\Delta$ & $=$ & $729 \, c^5 t^5 (c^2 t^2-11 c t-1)$ & $(t^4-18 t^3+74 t^2+18 t+1=0)/c$	& $1$ & $2$ & smooth\\[0.0em]
&& $J$         & $=$ & $\frac{(c^4 t^4-12 c^3 t^3+14 c^2 t^2+12 c t+1)^3}{1728 \, c^5 t^5 (c^2 t^2-11 c t-1)}$	  & $0, \infty$ & $\infty$  & $5$ & $2 \, I_5 \, (A_4)$     \\[0.4em]
&& $(X,Y)_{1,2}$ & $=$ & $( \frac{1}{4} c^2 t^2+\frac{3}{2} c t+\frac{1}{4}, \pm 3 \sqrt{3} c^2 t^2)$ & $(t^2-11t-1=0)/c$ & $\infty$ & $1$ & $2 \, I_1$ \\[0.1em]
\cline{6-6}
&&&&&&&& \\[-1em]
&& $(X,Y)_{3,4}$ & $=$ & $( \frac{1}{4} c^2 t^2-\frac{3}{2} c t+\frac{1}{4}, \pm 3 \sqrt{3} c t)$   & $c=\frac{11\mp 5\sqrt{5}}{2}, \; c'=\frac{11\pm 5\sqrt{5}}{2}$&&&\\[0.4em]
\hline
&&&&&&&& \\[-0.9em]
$X_{6321}$            &  $\Gamma_0(6)$
  & $g_2$     & $=$ & $\frac{3}{4} \big(t-4\big) \big(t^3+12 t^2 + 48 t-64\big)$   	& $4, 4(1-\sqrt[3]{2})$ 			& $0$	& $3$ & smooth  \\[0.4em]
  &  $\mathbb{Z}/6\mathbb{Z}$
  & $g_3$     & $=$ & $-\frac{1}{8}\big(t^2+4t-8\big)\big(t^4+8t^3+512t-512)$         & $-2 \sqrt[3]{2} (1\pm i\sqrt{3})-4$  	& $0$	& $3$ & smooth  \\[0.4em]
&& $\Delta$ & $=$ & $-729 \, t^6 \big(t-1\big) \big(t+8\big)^2$					& $4 \big(t^4+2 t^3+8 t-2=0\big)$ 	& $1$ 	& $2$ & smooth \\[0.2em]
&& $J$         & $=$ & $- \frac{(t-4)^3(t^3+12t^2+48t-64)^3}{1728t^6(t-1)(t+8)^2}$& $2 \pm 2 \sqrt{3}$				& $1$	& $2$ & smooth \\[0.4em]
&& $(X,Y)_{1}$  &  $=$ & $(-\frac{1}{2}t^2-2t+4,0)$							& $0$						& $\infty$	& $6$ & $I_6 \; (A_5)$ \\[0.4em]
&& $(X,Y)_{2,3}$ &  $=$ & $(\frac{1}{4}t^2-2t+4,\pm3\sqrt{3}t^2)$				& $1$						& $\infty$	& $1$ & $I_1$ \\[0.4em]
&& $(X,Y)_{4,5}$ &  $=$ & $(\frac{1}{4}t^2+4t+4,\pm3\sqrt{3}t(t+8))$			& $-8$						& $\infty$	& $2$ & $I_2 \; (A_1)$ \\[0.4em]
&&&&& $\infty$ & $\infty$ & $3$ & $I_3 \; (A_2)$ \\[0.4em]
\hline
\end{tabular}
\medskip
\caption{Extremal rational fibrations}\label{tab:4ExtRatHl}}
\end{table}
\par It is straight forward to work out corresponding versions of Corollary~\ref{curve_reduction}. In fact, the period integral of $dX/Y$ over a family of suitable A-cycles $\Sigma_1(t)$ in the neighborhood of $t=0$ are Heun functions, and we use rational transformations on the parameter curve to re-arrange the four singular points if necessary. We have the following:
\begin{corollary}
\label{lemma2heun}
For the families of elliptic curves over $\mathbb{P}^1 \backslash \{0, 1, a, \infty\}$, $X_{5511}$, $X_{6321}$, $X_{8211}$,  and $X_{9111}$, the period integrals of $dX/Y$ are annihilated by the Picard-Fuchs operator $L^{(2)}_t(a, q)$ in Equation~(\ref{eqn:L2heun}). For each family, the period over $\Sigma_1(t)$ is $\omega(t)=\heun(a, q; 1,1,1,1 \, | t)$ and holomorphic at $t=0$, with parameters $(a, q)$ and singular fibers at $t=0, 1, a, \infty$ given in Table~\ref{tab:HeunParameters}.  
\end{corollary}
\begin{proof}

For the families of Weierstrass models we use $dX/Y$ as the holomorphic  one-form on each regular fiber. It is well-known (cf.~\cite{MR927661}) that the Picard-Fuchs equation is given by the Fuchsian system 
\beq
\label{FuchsianSystem}
 \frac{d}{dt} \left( \begin{array}{c} \omega_1 \\ \eta_1 \end{array} \right) = \left( \begin{array}{ccc} - \frac{1}{12} \frac{d \ln\Delta}{dt} && \frac{3\,\delta}{2\,\Delta} \\ \\
 - \frac{g_2 \, \delta}{8 \, \Delta}& & \frac{1}{12} \frac{d \ln \Delta}{dt} \end{array} \right) \cdot \left( \begin{array}{c} \omega_1 \\ \eta_1 \end{array} \right)  \;,
  \eeq
 where $\omega_1 = \oint_{\, \Sigma_1} \frac{dX}{Y}$ and $\eta_1 = \oint_{\, \Sigma_1} \frac{X \, dX}{Y}$ for each one-cycle $\Sigma_1$ and with $\delta=3 \, g_3 \, g_2' - 2\, g_2 \, g_3' $. The rest follows by explicit computation.
\end{proof}
\begin{table}
\scalebox{\ScaleBig}{
\begin{tabular}{|c|l|c|}
\hline
& &\\[-0.9em]
$(a, q)$ & \multicolumn{1}{|c|}{Configuration} & Surface\\
\hline
& &\\[-0.9em]
$(2,1)$ & $I_1(t=0)\oplus I_2(t=1) \oplus I_1(t=a) \oplus I_8(t=\infty)$ & $X_{8211}$ \\[0.4em]
$\left( \frac{1}{2}, \frac{1}{2} \right)$ & $I_8(t=0)\oplus I_2(t=1) \oplus I_1(t=a) \oplus I_1(t=\infty)$ & $X_{8211}$\\[0.4em]
$\left(\frac{1\pm i\sqrt{3}}{2},\frac{3\pm i\sqrt{3}}{6}\right)$ & $I_1(t=0) \oplus I_1(t=1) \oplus I_1(t=a) \oplus I_9 (t=\infty)$& $X_{9111}$\\[0.4em]
$\left(-(c')^2, -\frac{3}{c}\right)$& $I_5(t=0) \oplus I_1(t=1) \oplus I_1(t=a) \oplus I_5(t=\infty)$& $X_{5511}$\\[0.4em]
& $c=\frac{11\mp 5\sqrt{5}}{2}, \; c'=\frac{11\pm 5\sqrt{5}}{2}$ & \\[0.4em]
$(-8,-2)$ & $I_6(t=0) \oplus I_1(t=1) \oplus I_2(t=a) \oplus I_3(t=\infty)$& $X_{6321}$\\[0.4em]
$\left( -\frac{1}{8}, \frac{1}{4} \right)$ & $I_3(t=0) \oplus I_1 (t=1) \oplus I_2 (t=a) \oplus I_6 (t=\infty)$& $X_{6321}$\\[0.4em]
$(9,3)$ & $I_1(t=0)\oplus I_6(t=1) \oplus I_2(t=a) \oplus I_3(t=\infty)$& $X_{6321}$\\[0.4em]
$\left( \frac{1}{9}, \frac{1}{3} \right)$ & $I_3(t=0) \oplus I_6 (t=1) \oplus I_2 (t=a) \oplus I_1 (t=\infty)$ & $X_{6321}$\\[0.4em]
$\left( \frac{9}{8}, \frac{3}{4} \right)$ & $I_2(t=0) \oplus I_6 (t=1) \oplus I_1 (t=a) \oplus I_3 (t=\infty)$ & $X_{6321}$\\[0.4em]
$\left( \frac{8}{9}, \frac{2}{3} \right)$ & $I_2(t=0) \oplus I_1(t=1) \oplus I_6 (t=a) \oplus I_3 (t=\infty)$ & $X_{6321}$\\[0.4em]
\hline
\end{tabular}
\medskip
\caption{Parameters of Heun functions for extremal elliptic surfaces}\label{tab:HeunParameters}}
\end{table}
\par Analogous to Lemma~\ref{lemmaK3mixed_eq}, it follows that twisted families with generalized functional invariant $(i,j,\alpha)=(1,1,1)$ of $X_{5511}$, $X_{6321}$, $X_{8211}$,  and $X_{9111}$ are families of $M_n$-lattice polarized K3 surfaces with $n=5, 6, 8, 9$, respectively. As before, we also obtain a continuously varying family of closed two-cycles $\Sigma_2(t)$. We therefore have the following:
\begin{corollary}
\label{MES+heun}
The twisted families with generalized functional invariant $(i,j,\alpha)=(1,1,1)$ given by Equation~(\ref{lemmaK3mixed_eq}), of the families in Corollary~\ref{lemma2heun} are families over the rational modular curves $\mathbb{H}/\Gamma_0(n)^+$ of $M_n$-lattice polarized K3 surfaces with $n=5, 6, 8, 9$. For each family, the period integral~(\ref{eqn:period_new}) is annihilated by the Picard-Fuchs operator $L^{(3)}_{t}(\frac{1}{4}; a, q)$ in Equation~(\ref{eqn:1L3heun}). In particular, the period over $\Sigma_2(t)$ is holomorphic at $t=0$ and given by
\beq
\label{period_lemmaK3heun}
  \omega  = (2 \pi i)^2 \;   {}_1F_0\!\left(\left. \frac{1}{2} \right| t\right) \star  \heun(a, q; 1,1,1,1 \, | t) \,,
\eeq
where parameters $(a, q)$ and singular fibers over $t=0, 1, a, \infty$ (before twisting) are given in Table~\ref{tab:HeunParameters}.  
\end{corollary}
\begin{proof}
The proof follows directly by checking that the singular fibers and Mordell-Weil groups for the families constructed in Lemma~\ref{lemmaK3mixed}
agree with the ones given by Dolgachev in \cite{MR1420220}. The rest of the proof is analogous to the proof of Corollary~\ref{cor:modular_K3_surfaces}.
\end{proof}
\begin{remark}
Identity~(\ref{eqn:heun_identity}) reflects the well-known decomposition of the Picard-Fuchs operator into a symmetric square for families of $M_n$-lattice polarized K3 surfaces of Picard-rank 19. If one considers non-rigid, smooth Calabi-Yau threefolds (non-isotrivially) fibered by K3 surfaces admitting a $M_n$-lattice polarization, then it wash shown in \cite{Doran:aa} that $1 \le n \le 11$, $n\not =10$ and that all such $n$ can be realized. Our twist construction provides explicit examples for such Calabi-Yau threefolds for $n \in \{1,2,3,4,5,6,8,9\}$.
\end{remark}
\par Applying our twist construction again yields families of Calabi-Yau threefolds whose Picard-Fuchs operators realize non-rigid Calabi-Yau operators with four regular singular points. As in Section~\ref{ssec:3folds_hg_even}, we also obtain a continuously varying family of closed three-cycles $\Sigma_3(t)$.  We have the following:
\begin{corollary}
\label{MES+heun_CY3}
The twisted families with generalized functional invariant $(k, l, \beta)=(1,1,1)$, given by Equation~(\ref{lemmaCYmixed_eq}), of the elliptic K3 surfaces in Corollary~\ref{MES+heun} are families over $\mathbb{P}^1 \backslash \{0, 1, a, \infty\}$ of Jacobian elliptic Calabi-Yau threefolds over $\mathbb{P}^1 \times \mathbb{P}^1$. For each family, the period integral~(\ref{eqn:period_new}) is annihilated by the Picard-Fuchs operator ${}_1L^{(4)}_{t}(\frac{1}{4}; a, q)$ in Equation~(\ref{eqn:CYO_4a}). In particular, the period over $\Sigma_3(t)$ is holomorphic at $t=0$ and given by
\beq
\label{period_lemmaCY3heun}
\begin{split}
  \omega  & =  (2 \pi i)^3 \; {}_1F_0\!\left(\left. \frac{1}{2} \right| t\right) \star \heun\left(a,\frac{q}{4};\frac{1}{4},\frac{3}{4},1,\frac{1}{2} \, \Big|t \right)^2
\end{split}  
\eeq
with parameters $(a, q)$ given in Table~\ref{tab:HeunParameters}.  
\end{corollary}
\begin{proof}
The proof is analogous to the proof of Lemma~\ref{lemmaCY3mixed} and Corollary~\ref{hg-case}.
\end{proof}
\par Applying base transformations between twists again greatly improves the scope of our twist construction. We make the following:
\begin{remark}
\label{rem:transfo}
Special function identities for the Heun function can be used to realize Picard-Fuchs operators ${}_1L^{(4)}_{t}(\alpha; a, q)$ in Equation~(\ref{eqn:CYO_4a}) for values other than $\alpha=\frac{1}{4}$ in Table~\ref{tab:4ExtRatHl}. As an example, we consider the case $\alpha=\frac{1}{2}$ where we use a sequence of identities for the Heun function found in~\cite{MR2367417}. For $\beta \in (0,1) \cap \mathbb{Q}$ and $a\not= 1$, we use the linear identity 
\beq
\begin{split}
 \heun\left(a, q; 1,1,1,1 \Big| x \right) = \dfrac{1}{1 - \frac{x}{a}} \; \heun\left(1-a,1-q;1,1,1,1 \Big| T_1(x) \right) 
\end{split}
\eeq
with $T_1(x)=\frac{(1-a)\,x}{x-a}$, the quadratic identity
\beq
 \heun\left(a,\frac{q}{4};\beta,1-\beta,1,\frac{1}{2} \Big| T_2(x) \right) = \left(1 - x\right)^{\beta} \; \, \heun\left(a',q'; 2\beta,1,1,2\beta \Big| x \right)
\eeq
with $T_2(x)=\frac{R\, x \, (a-x)}{1-x}$, where $a$ and $a'$ are related by
\beq
(a')^2 \, (1-a)^2 -16 \, (1-a') \, a =0 \;,
\eeq
and $q=4 \, R \, (q'-\beta a')$ and $R=\frac{1+a}{2(2-a')}$,  combined with the bi-quadratic quartic identity 
\beq
\begin{split}
 \heun\left(a,\frac{q}{4};\frac{1}{4},\frac{3}{4},1,\frac{1}{2} \Big| T_4(x) \right) = \left(1 - \frac{x^2}{a}\right)^{1/2} \; \heun\left(a, q; 1,1,1,1 \Big| x \right) \;,
\end{split}
\eeq
with $T_4(x)=\frac{4\, a \, x \, (1-x) \, (a-x)}{(a-x^2)^2}$. This implies that a Heun function of the form
\beqn
  \heun\left(a,\frac{q}{4};\frac{1}{2},\frac{1}{2},1,\frac{1}{2} \Big| T_2(x) \right)
\eeqn
is related to the Heun function
\beqn  
\dfrac{\sqrt{1-x}}{(1-\frac{x}{a})\sqrt{1-T_1(x)^2/a}} \; 
 \heun\left(1-a',\frac{1-q'}{4};\frac{1}{4},\frac{3}{4},1,\frac{1}{2} \Big|  T_4\big(T_1(x)\big) \right) \,.
\eeqn
In turn, the latter is realized as holomorphic period of an extremal rational surface after pullback by a base transformation and twist. The Picard-Fuchs operator of the twisted family, constructed analogously to Corollary~\ref{MES+heun_CY3}, then realizes the Calabi-Yau operators ${}_1L^{(4)}_{t}(\frac{1}{2}; a, q)$ for parameters $(a, q)$ given in Table~\ref{tab:HeunParameters}.
\end{remark}
\par Applying the variant of the twist construction for Section~\ref{sssec:product_twists} yields other families of Calabi-Yau threefolds whose Picard-Fuchs operators realize more non-rigid Calabi-Yau operators with four regular singular points. As in Section~\ref{ssec:even_case}, we also obtain a continuously varying family of closed three-cycles $\hat{\Sigma}_3(t)$.  We have the following:
\begin{corollary}
\label{MES+heun_CY3extra}
For every family $X \to \mathbb{P}^1 \backslash \{0, 1, a, \infty\}$ in Corollary~\ref{lemma2heun}, the twist family of $X$ with $X'_k$ in Table~\ref{tab:twist_rationalsurface} given by Equation~(\ref{CY3mtWEf}) for $k=2,3,4$, is a family over $\mathbb{P}^1 \backslash \{0, 1, a, \infty\}$ of Jacobian elliptic Calabi-Yau threefolds over $\mathbb{F}_n$ with $n=0, \dots, k$. The period integral~(\ref{eqn:period_new_producttwist}) is annihilated by the Picard-Fuchs operator ${}_2L^{(4)}_{t}(\mu; a, q)$ in Equation~(\ref{eqn:CYO_4b}). In particular, the period over $\hat{\Sigma}_3(t)$ is holomorphic at $t=0$ and given by
\beq
\label{period_lemmaCY3heun_extra}
\begin{split}
  \hat{\omega}  & =   (2 \pi i)^3 \;  \hpg21{\mu, 1-\mu}{1}{t}   \star \heun(a, q; 1,1,1,1 \, | t) \,,
\end{split}  
\eeq
where parameters $(a, q)$ are given in Table~\ref{tab:HeunParameters} and $\mu \in \{ \frac{1}{3}, \frac{1}{4}, \frac{1}{6} \}$.
\end{corollary}
\begin{proof}
The proof is analogous to the proof of Corollary~\ref{even-case_b} in Section~\ref{ssec:even_case}.
\end{proof}
\begin{remark}
Corollary~\ref{MES+heun_CY3}, Corollary~\ref{MES+heun_CY3extra}, and Remark~\ref{rem:transfo} realize 30 non-rigid Calabi-Yau operators with four regular singular points as Picard-Fuchs operators of families of Calabi-Yau threefolds obtained by our twist construction, namely all operators ${}_1L^{(4)}_{t}(\mu; a, q)$  in Equation~(\ref{eqn:CYO_4a}) with $\mu \in \{ \frac{1}{2}, \frac{1}{4} \}$ and operators ${}_2L^{(4)}_{t}(\mu; a, q)$ in Equation~(\ref{eqn:CYO_4b}) with $\mu \in \{ \frac{1}{2}, \frac{1}{3}, \frac{1}{4}, \frac{1}{6} \}$, for parameters $(a, q)$ given in Table~\ref{tab:HeunParameters} up to rescaling $t \mapsto \lambda a t$ of the affine base coordinate according to Table~\ref{tab:VHS6}.
\end{remark}
\section{Discussion and Outlook}
We introduced a twist construction to iteratively obtain families of Calabi-Yau $n$-folds over $\mathbb{P}^1 \backslash \{0,1\infty\}$, internally elliptically fibered by Calabi-Yau $(n-1)$-folds. Our construction is a geometric generalization of Katz’s middle convolution combined with an additional rational pullback operation on the internal fibration. By computing the periods of a holomorphic top-form over explicit topological cycles and expressing the results in hypergeometric terms, we produced Weierstrass models whose Picard-Fuchs operators realize all 60 Calabi-Yau operators inducing $\operatorname{Sp}(4,\mathbb{C})$-rigid, quasi-unipotent local systems of weight three and rank four having a maximal unipotent element. This is important because the usual middle convolution only is guaranteed to produce the $\operatorname{GL}(4,\mathbb{C})$-rigid monodromy tuples. Our iterative construction provides a unifying construction for many examples of elliptic curves, K3 surfaces, and Calabi-Yau threefolds  considered in the context of mirror symmetry, e.g., families considered in \cites{MR1877764, MR1738869, 1501.04024} and examples in~\cite{MR1860046}.   Moreover,  by restricting the family parameter to special values one readily obtains elliptic curves, K3 surfaces, and Calabi-Yau  threefolds with  properties such as CM, admitting a Shioda-Inose structure associated with abelian surfaces with quaternionic multiplication, or rigidity. Some of these arithmetic properties were investigated in \cite{MR1684614}, and their fiberwise Picard-Fuchs equations were computed in \cite{MR1881612}. These results are all reproduced by our iterative construction. 
\par  We also used our iterative construction to obtain families with four singular fibers, such as all extremal Jacobian rational elliptic surfaces with four singular fibers from the Miranda-Persson list \cite{MR867347}, and models for families of $M_n$-lattice polarized K3 surfaces for $n \le 9$ with $n\not = 7$. On the level of periods, the role of the Gauss hypergeometric function was then replaced by the Heun function. Identities for the hypergeometric function were replaced by identities for the Heun equation, for example relations found in \cites{MR2367417,MR2506172}. In this way, our iterative construction again provided a unified geometric approach for many differential equations associated with K3 surfaces studied in isolation \cites{MR749676, MR1034260, MR783555,MR875089,MR2428514}. Our iterative construction also reproduced many of the classical examples of threefolds investigated in the context of mirror symmetry, for example in \cites{MR1372822, MR1621573}.  We hope that the obtained new families and our iterative technique itself could be of interest for ``global mirror symmetry” frameworks, e.g., see \cites{MR3210178, MR3112509}, curve-counting on 3-folds, F-theory, for studying thin vs.\!~arithmetic monodromy, and maybe in the future even homological mirror symmetry.
\par However, the geometric realization of the Calabi-Yau operators in the odd case in Theorem~\ref{DoranMalmendier} is not yet completely satisfactory: instead of producing families of threefolds whose Picard-Fuchs operators realize the fourth-order Calabi-Yau operators of the odd case directly, we constructed families of fourfolds instead, such that the Yifan-Yang pullback of their rank-five Picard-Fuchs operators realized the Calabi-Yau operators. The observant reader might have noticed that a similar situation already occurred at lower dimension in our iterative construction. The twist construction applied to any family of elliptic curves from Table~\ref{tab:3ExtRatHg} produced families of K3 surfaces whose Picard-Fuchs operators were symmetric squares of  rank-two and degree-one Calabi-Yau operators. In fact, Clausen's identity~(\ref{Clausen}) expresses the holomorphic K3 periods as squares of Gauss hypergeometric functions. Using the hypergeometric function identity~(\ref{rel_quad}), we were able to relate the (symmetric) square root back to the holomorphic solution of the Picard-Fuchs equation before the twist. In this sense, carrying out a twist and taking a (symmetric) square root is equivalent to carrying out a quadratic transformation on the parameter space of the original family.  Hodge-theoretically this is due to the fact that the K3 surfaces of Picard-rank $19$ or $18$ admit Shioda-Inose structures relating them to Kummer surfaces associated to products of elliptic curves. We have already proved that, at least in one case,  a similar Hodge-theoretic interpretation exists for the exterior square root of the simplest Calabi-Yau operator in the odd case as well.
\par Moreover, we expect that our iterative construction of a transcendental cycle for the holomorphic period can be extended to obtain a full basis of transcendental cycles.  Such bases would in turn allow us to construct the integral monodromy matrices for each family. This is important because the full period lattices are a powerful tool to distinguish examples of different geometric variations of Hodge structure over $\mathbb{Z}$ that are isomorphic over $\mathbb{R}$. For example, the quintic-mirror and the quintic-mirror twin family share the exact same Picard-Fuchs equation, but they have different ranks in the even dimensional cohomologies \cite{MR2282973}. These results will be the subject of a forthcoming article.
\bibliographystyle{amsplain}
\bibliography{references}{}
\end{document}